\documentclass[10pt]{amsart}
\usepackage{amssymb}\usepackage{stmaryrd}\usepackage[all]{xy}
\usepackage{amsfonts,amssymb,stmaryrd,amscd,amsmath,latexsym,amsbsy,bm}

\newtheorem{theorem}{Theorem}[section]\newtheorem{lemma}[theorem]{Lemma}\newtheorem{proposition}[theorem]{Proposition}\theoremstyle{definition}\newtheorem{definition}[theorem]{Definition}

\newtheorem{remark}[theorem]{Remark}

\theoremstyle{plain}\newtheorem{thm}{Theorem}[section]\newtheorem{prop}{Proposition}[section]\newtheorem{cor}{Corollary}[section]\theoremstyle{definition}\theoremstyle{remark}

\newcommand{\solu}[1]{\begin{sol}{\bf (\ref{#1})}}

\def\g{\mathfrak{g}}\def\A{\mathcal{A}}

\makeatother
\input cyracc.def
\DeclareFontFamily{U}{russian}{}
\DeclareFontShape{U}{russian}{m}{n}
        { <5><6> wncyr5
        <7><8><9> wncyr7
        <10><10.95><12><14.4><17.28><20.74><24.88> wncyr10 }{}
\DeclareSymbolFont{Russian}{U}{russian}{m}{n}
\DeclareSymbolFontAlphabet{\mathcyr}{Russian}
\makeatletter
\let\@math@cyr\mathcyr
\renewcommand{\mathcyr}[1]{\@math@cyr{\cyracc #1}}
\makeatother
\newcommand{\sh}{{\mathcyr{sh}}} 

\begin{document}


\title{A Tannakian interpretation of the elliptic infinitesimal braid Lie algebras}
\author{Benjamin Enriquez}
\address{D\'epartement de math\'ematiques, Universit\'e de Strasbourg, et  IRMA (CNRS), 
7 rue Ren\'e Descartes, 67000 Strasbourg, France}
\email{b.enriquez@math.unistra.fr}
\author{Pavel Etingof}
\address{Department of Mathematics, Massachusetts Institute of Technology,
Cambridge, MA 02139, USA}
\email{etingof@math.mit.edu}

\dedicatory{To Aleksandr Aleksandrovich Kirillov on his $3^{2^2}$-th birthday, with admiration}

\maketitle 

\begin{abstract}

Let $n\geq 1$. The pro-unipotent completion of the pure braid group of $n$ points on a genus 1
surface has been shown to be isomorphic to an explicit pro-unipotent group with graded Lie algebra
using two types of tools: (a) minimal models (Bezrukavnikov), (b) the choice of a complex structure on
the genus 1 surface, making it into an elliptic curve $E$, and an appropriate flat connection on the
configuration space of $n$ points in $E$ (joint work of the authors with D. Calaque). Following a suggestion
by P. Deligne, we give an interpretation of this isomorphism in the framework of the Riemann-Hilbert 
correspondence, using the total space $E^\#$ of an affine line bundle over $E$, which identifies with the
moduli space of line bundles over $E$ equipped with a flat connection. 

\end{abstract}

\subsection*{Introduction}

Let $T$ be a topological 2-torus, i.e., a closed, compact topological surface of genus 1. For $n\geq 1$, let $C_n(T)$ be its 
configuration space, defined as the complement of the union of all the diagonals of $T^n$. A base point $x\in C_n(T)$ being fixed, 
we denote by $\mathrm{PB}_{1,n}^x$ the fundamental group of $C_n(T)$ relative to $x$; it is called
the pure braid group of genus 1. One attaches to this group its prounipotent completion over $\mathbb Q$, which is a 
prounipotent $\mathbb Q$-group, and the Lie algebra of this $\mathbb Q$-group, which is a pronilpotent
$\mathbb Q$-Lie algebra, which we will denote $\mathrm{Lie}\mathrm{PB}_{1,n}^x$. This Lie algebra is equipped with the descending 
filtration associated with the lower central series. The associated graded Lie algebra $\mathrm{gr}\mathrm{Lie}\mathrm{PB}_{1,n}^x$
is then a positively graded $\mathbb Q$-Lie algebra. 

We denote by $\mathrm{Lie}\mathrm{PB}_{1,n}^x\hat\otimes\mathbb C$ the completed tensor product of 
$\mathrm{Lie}\mathrm{PB}_{1,n}^x$ with $\mathbb C$ (inverse limit of the tensor products with $\mathbb C$ of the 
filtered quotients of $\mathrm{Lie}\mathrm{PB}_{1,n}^x$). Then $\mathrm{Lie}\mathrm{PB}_{1,n}^x\hat\otimes\mathbb C$
is a complete, filtered complex Lie algebra, and is associated graded is isomorphic to $\mathrm{gr}\mathrm{Lie}\mathrm{PB}_{1,n}^x
\otimes\mathbb C$.

Let $\mathfrak t_{1,n}$ be the Lie algebra with generators $x_i,y_i$ ($i\in[1,n]$), $t_{ij}$ ($i\neq j\in[1,n]$), 
and relations
\begin{equation}\label{rel:t:xx:yy}
\forall i,j\in[1,n], \quad [x_i,x_j]=[y_i,y_j]=0, 
\end{equation}
\begin{equation}\label{rel:t:xy:t}
\forall i\neq j\in[1,n],\quad [x_i,y_j]=t_{ij}=t_{ji}, 
\end{equation}
\begin{equation}\label{rel:t:xy:tt}
\forall i\in[1,n],\quad [x_i,y_i]=-\sum_{j|j\neq i}t_{ij}, 
\end{equation}
\begin{equation}\label{rel:t:x:t:y:t}
\forall i,j,k\in[1,n]\quad\text{with}\quad\#\{i,j,k\}=3,\quad [x_k,t_{ij}]=[y_k,t_{ij}]=0, 
\end{equation}
\begin{equation}\label{rel:t:xx:t:yy:t}
\forall i\neq j\in[1,n],\quad [x_i+x_i,t_{ij}]=[y_i+y_j,t_{ij}]=0. 
\end{equation}
We set $\mathrm{deg}(x_i)=1$, $\mathrm{deg}(y_i)=0$ for $i\in[n]$, which induces a grading on 
$\mathfrak t_{1,n}$. 

We set $\mathfrak t_{1,n}^{\mathbb C}:=\mathfrak t_{1,n}\otimes\mathbb C$. We denote by $\hat{\mathfrak t}_{1,n}^{\mathbb C}$ 
the degree completion of $\mathfrak t_{1,n}^{\mathbb C}$. 

 \begin{thm}[\cite{Bez}] \label{thm:bez}
\begin{itemize}
\item
The graded Lie algebras $\mathrm{gr}\mathrm{Lie}\mathrm{PB}_{1,n}^x\otimes\mathbb C$ and $\mathfrak t_{1,n}^{\mathbb C}$
are isomorphic. 
\item $\mathrm{Lie}\mathrm{PB}_{1,n}^x\hat\otimes\mathbb C$ is isomorphic, as a completed filtered Lie algebra, to 
$\hat{\mathfrak t}_{1,n}^{\mathbb C}$. 
\end{itemize}
\end{thm}

The proof from \cite{Bez} is based on the minimal model theory. In \cite{CEE}, we gave another proof of this theorem
by choosing a complex structure on $T$, making it into an elliptic curve $E$, and by constructing a 
suitable connection on a principal bundle over $C_n(E)$ with structure group $\mathrm{exp}(\hat{\mathfrak t}_{1,n}^{\mathbb C})$
(for a survey of the construction of this connection see \cite{Hain}). The monodromy of this connection defines a morphism 
\begin{equation}\label{morph}
\mathrm{PB}_{1,n}^x\to\mathrm{exp}(\hat{\mathfrak t}_{1,n}^{\mathbb C}),
\end{equation} 
which gives rise to isomorphisms $\mathrm{Lie}\mathrm{PB}_{1,n}^x\hat\otimes\mathbb C\simeq \hat{\mathfrak t}_{1,n}^{\mathbb C}$ 
and $\mathrm{gr}\mathrm{Lie}\mathrm{PB}_{1,n}^x\otimes\mathbb C\simeq\mathfrak t_{1,n}^{\mathbb C}$
enabling one to prove the announced statements. The whole construction arises as an elliptic analogue of the similar genus zero 
construction, and may also be viewed as a universal version of the KZB connection. 

One of the main features of Theorem \ref{thm:bez} is that it says that $\mathrm{Lie}\mathrm{PB}_{1,n}^x\hat\otimes\mathbb C$ is 
isomorphic to the completion of a graded Lie algebra. 

The purpose of this paper is, following a suggestion of P. Deligne,  to give an interpretation of this isomorphism in the framework of the 
Riemann-Hilbert (RH) correspondence, thus providing a categorical approach to the results of \cite{CEE}. Let us recall the framework of 
this correspondence. 

Let $X$ be a smooth complex algebraic variety. According to \cite{Del}, there is an equivalence of tensor categories (the RH correspondence) between: 

(i) the category $\mathrm{VBFC}(X)$ of vector bundles with a flat connection on $X$ with regular singularities;  

(ii) the category $\mathrm{LS}(X)$ of topological local systems on $X$; 

One attaches to each tensor category its unipotent part (see \S\ref{subsect:unip}). The RH correspondence then induces an equivalence
between the unipotent parts of its two sides, namely
\begin{equation}\label{RH:equiv}
\mathrm{RH}_{unip}:\mathrm{VBFC}(X)_{unip}\stackrel{\sim}{\to}\mathrm{LS}(X)_{unip};   
\end{equation}
it attaches, to each object of $\mathrm{VBFC}(X)$, the local system of its horizontal sections. 

Any point $x\in X$ gives rise to a fiber functors $F^{ls}_x:\mathrm{LS}(X)\to\mathrm{Vec}_{\mathbb C}$
and $F_x^{vb}:\mathrm{VBFC}(X)\to\mathrm{Vec}_{\mathbb C}$, equipped with a canonical isomorphism 
$F^{ls}_x\circ\mathrm{RH}\simeq F_x^{vb}$.

The Tannakian group corresponding to $F^{ls}_x$ is $\mathrm{Aut}^{\otimes}(F_x^{ls})\simeq\pi_1^B(X,x)$ (this is the Betti 
fundamental group of $X$ with base point $x$). 

Set $X:=C_n(E)$. Then $\pi_1^B(X,x)=\mathrm{PB}_{1,n}^{x,unip}(\mathbb C)$, where the exponent $unip$ means the prounipotent 
completion of a discrete group and $-(\mathbb C)$ denotes the group of $\mathbb C$-points, so that one of the sides of 
the isomorphism of Theorem \ref{thm:bez} relates to the left-hand side of the RH equivalence (\ref{RH:equiv}). 

We prove: 
\begin{thm}
\begin{itemize}
\item[1)]There exists: \begin{itemize}\item[a)] an explicit tensor functor 
$$ F:\mathrm{VBFC}(C_n(E))_{unip}\to \mathrm{Vec}_{\mathbb C}  $$
\item[b)] a natural isomorphism 
\begin{equation}\label{equiv}
\mathrm{VBFC}(C_n(E))_{unip}\ni (\mathcal E,\nabla)\mapsto i_{(\mathcal E,\nabla)}\in\mathrm{Iso}_{\mathrm{Vec}_{\mathbb C}}(F(\mathcal E,\nabla),F_x^{vb}(\mathcal E,\nabla))
\end{equation}
between the functors $F$ and $F_x^{vb}$, 
\item[c)] a canonical isomorphism $\mathrm{Aut}^\otimes(F)\simeq\mathrm{exp}(\hat{\mathfrak t}_{1,n}^{\mathbb C})$. 

\end{itemize}
\item[2)] The composed isomorphim 
$$
\xymatrix{
\mathrm{exp}(\hat{\mathfrak t}_{1,n}^{\mathbb C})\ar^{\sim}[r]  &
\mathrm{Aut}^\otimes(F)\ar^{\sim}[r]  &
\mathrm{Aut}^\otimes(F_x^{vb})\ar^{\sim}_{\mathrm{RH}}[r] &
\mathrm{Aut}^\otimes(F_x^{ls})\ar^{\sim}[r]  &
\mathrm{PB}_{1,n}^{x,unip}(\mathbb C)
}
$$
coincides with the inverse of the completion of (\ref{morph}).  
\end{itemize}
\end{thm}

The group $\mathrm{Aut}^\otimes(F_x^{vb})$ is the de Rham fundamental group of $C_n(E)$ with base point $x$, denoted 
$\pi_1^{DR}(C_n(E),x)$, and the isomorphism $\pi_1^{DR}(C_n(E),x)\simeq\mathrm{Aut}^\otimes(F_x^{vb}) \stackrel{\mathrm{RH}}{\to}
\mathrm{Aut}^\otimes(F_x^{ls})\simeq\pi_1^{B}(C_n(E),x)$ is the 'comparison isomorphism'. 

The construction of the functor $F$ depends on some geometric background. We fix a resolution of singularities $\pi_0:\tilde X_0\to E^n$, 
such that if $D_0\subset E^n$ is the union of all diagonals, then $\tilde D_à:=\pi_0^{-1}(D_0)$ is a normal crossing 
divisor. Let $E^\#$ be the universal additive extension of $E$. This is a 2-dimensional commutative algebraic group, fitting in an exact sequence 
$0\to H^0(E,\mathcal O)^\vee\to E^\#\to E\to 0$ (see \S\ref{sect:ellmat}). It gives rise to a morphism $p:(E^\#)^n\to E^n$.
Let $\tilde X:=(E^\#)^n\times_{E^n}\tilde X_0$ and let $D:=D_0\times_{E^n}(E^\#)^n$, $\tilde D:=D_0\times_{E^n}\tilde X$. 
Then $\tilde D\subset\tilde X$ is a normal crossing divisor and there is a commutative diagram (see \S\ref{sect:geomdata})
$$
\xymatrix{ \tilde D\ar[r]\ar_{\subset}[d]& D\ar[r]\ar_{\subset}[d] & D_0 \ar_{\subset}[d]\\ \tilde X\ar[r]& (E^\#)^n\ar[r] & E^n}
$$

Let us now explain the construction of the functor $F$, which was proposed by P. Deligne in \cite{Del1,Del2}. 
Let $(\mathcal E,\nabla)$ be a unipotent vector bundle with flat connection on 
$E^n-D_0$. One lifts $(\mathcal E,\nabla)$ to a unipotent vector bundle with flat connection on $\tilde X-\tilde D$. It canonically 
gives rise to a vector bundle on $\tilde X$ with flat connection on $\tilde X-\tilde D$ admitting simple poles at $\tilde D$ (see 
\S\ref{sect:summary}, isomorphism $(c)$). This object is the lift of a vector bundle on $(E^\#)^n$, equipped with a flat connection on 
$(E^\#)^n-D$ and simple poles at $D$. Since this is a unipotent object, and a unipotent vector bundle on $(E^\#)^n$ is trivial 
by the homological properties of $(E^\#)^n$, the obtained vector bundle $\mathcal E^\#$ on $(E^\#)^n$ is trivial. We have then 
$\mathcal E^\#\simeq V\otimes\mathcal O$, where $V:=H^0((E^\#)^n,\mathcal E^\#)$. We then set $F(\mathcal E,\nabla):=V$. 

So we see that the homological properties of $E^\#$ enable one to apply to the 
elliptic situation the framework from \cite{Del3}, \S12.

The equivalence (\ref{equiv}) is then given by the specialization map 
$$
(\mathcal E,\nabla)\mapsto [F(\mathcal E,\nabla)=V\simeq (V\otimes\mathcal O)_x\simeq\mathcal E_x=F_x^{vb}(\mathcal E,\nabla)]. 
$$

In order to prove the isomorphism $\mathrm{Aut}^\otimes(F)\simeq\mathrm{exp}(\hat{\mathfrak t}_{1,n}^{\mathbb C})$, we construct a category 
equivalence between $\mathrm{VBFC}(C_n(E))_{unip}$ and the category $\mathrm{Vec}((E^\#)^n,D)$ of flat connections on 
trivial vector bundles over $(E^\#)^n$, unipotent and with simple poles at $D$. The computation of the latter category relies on the 
study of the algebra of differential forms over $(E^\#)^n$.  Computation then shows that the Lie algebra $\mathrm{Der}^\otimes(F)$
is graded; this fact originates from the graded structure of differential forms over $(E^\#)^n$ with poles at $D$, which comes from 
homogeneity properties of the Fay relations. 

The organization of the text is described in the following table of contents.  

\tableofcontents

\subsection*{Acknowledgements}
We are very grateful to P. Deligne for his letters to one of us \cite{Del1,Del2}, which initiated this work and contain many of the 
main ideas used there. PE 
was partially supported by NSF grant no DMS-1502244. 

\section{Basic material and constructions}\label{sect:background}

In this paper, we work over an algebraically closed field $\mathbf k$ of characteristic 0. 

\subsection{Categories}

In this \S, we attach, to each tensor functor $\mathcal C\to\mathcal D$, a tensor functor $\mathcal C_{unip}\to\mathcal D_{unip}$. 
We will denote by $\mathrm{Vec}$ the tensor category of finite dimensional $\mathbf k$-vector spaces. 

\subsubsection{Tensor categories}

Let $\mathcal C$ be a locally finite, $\mathbf k$-linear, abelian, rigid monoidal category, such that the endomorphism ring of
its unit object $\mathbf 1_{\mathcal C}$ is isomorphic to $\mathbf k$ and such that  its tensor product 
$\otimes:\mathcal C\times\mathcal C\to\mathcal C$ is bilinear on morphisms. Such a category is called {\it tensor} 
in \cite{EGNO}. According to {\it loc. cit.}, Prop. 4.21, the tensor product bifunctor $\otimes$ is then biexact. 

\subsubsection{Unipotent parts of tensor categories}\label{subsect:unip}

Define $\mathrm{Ob}(\mathcal C_{unip})$ to be the subclass of $\mathrm{Ob}(\mathcal C)$ consisting of all objects 
$O$ that admit a filtration $0=O_0\subset O_1\subset\cdots\subset O_n=O$, such that each quotient $O_i/O_{i-1}$
is isomorphic to $\mathbf 1_{\mathcal C}$. One checks that $\mathrm{Ob}(\mathcal C_{unip})$ 
is stable under the tensor product of $\mathcal C$. 

Define $\mathcal C_{unip}$ to be the full subcategory of $\mathcal C$ whose class of objects is $\mathrm{Ob}(\mathcal C_{unip})$. 
Then $\mathcal C_{unip}$ is again a tensor category.  

Let $\mathcal D$ be another tensor category. A tensor functor from $\mathcal C$ to $\mathcal D$ is a pair $(F,J)$ of an exact and faithful 
$\mathbf k$-linear functor $F:\mathcal C\to\mathcal D$ and a functorial isomorphism $J:F(-)\otimes F(-)\to F(-\otimes-)$, satisfying 
diagram (2.23) in \cite{EGNO}. One checks that such a tensor functor induces a tensor functor $(F_{unip},J_{unip})$ 
from $\mathcal C_{unip}$ to 
$\mathcal D_{unip}$. One then has a diagram of tensor functors
$$
\xymatrix{
\mathcal C_{unip}\ar^{F_{unip}}[r]\ar[d] & \mathcal D_{unip}\ar[d]\\ \mathcal C\ar_{F}[r] & \mathcal D
}
$$
where the vertical functors are fully faithful. 

\subsection{Divisors and residues} Let $X$ be a smooth irreducible $\mathbf k$-variety. 

\subsubsection{Sheaves}\label{sect:sheaves}

We denote by $\mathcal O_X$ the structure sheaf of $X$ and by $\mathcal K_X$ its sheaf of rational functions; 
this is a constant sheaf. If $\delta$ is a divisor of $X$ (union of codimension 1 subvarieties), we denote by $\mathcal O_{X,\delta}$ the subsheaf 
of $\mathcal K_X$, 
such that for each open subset $U$ of $X$, the space $\Gamma(U,\mathcal O_{X,\delta})$ is the space of rational functions over $U$ (or $X$), 
that are regular on a dense open set in $U\cap\delta$. Restriction to $\delta$ induces a 
$\mathcal O_X$-sheaf morphism $\mathcal O_{X,\delta}\to i_*(\mathcal K_\delta)$, where $i:\delta\to X$ is the canonical inclusion
and $\mathcal K_\delta$ is the direct sum $\oplus_i\mathcal K_{\delta_i}$, where $(\delta_i)_i$ are the irreducible components of $\delta$. 

If $\mathcal E$ is a quasi-coherent $\mathcal O_X$-sheaf, we define $\mathcal E^{rat}$ to be the sheaf 
$\mathcal E\otimes_{\mathcal O_X}\mathcal K_X$ and $\mathcal E_{\delta}$ to be the sheaf 
$\mathcal E\otimes_{\mathcal O_X}\mathcal O_{X,\delta}$. Then $(\mathcal O_X)_\delta=\mathcal O_{X,\delta}$. 

We also denote by $\Gamma_{rat}(X,\mathcal E)$ the space of rational sections of $\mathcal E$. So 
$\Gamma_{rat}(X,\mathcal E)=\Gamma(X,\mathcal E^{rat})$. 

\subsubsection{Divisors}\label{sect:divisors}

A {\it special divisor} (SD) in $X$ is a divisor whose components are non-singular and such that any 
pair of components intersects transversally. 

A {\it reduced normal crossing divisor} (RNCD) is a divisor $D=\cup_{i\in I}D_i$, whose components are non-singular 
and satisfying the following condition. For each point $p$ of $X$, let $J(p)$ be the set of indices $j$ such that $p$ lies in $D_j$. 
Then $p$ should have a neighborhood $U(p)$, in which each $D_j$, $j\in J(p)$ may be defined by an 
equation $f_j=0$, and the collection of differentials $(df_j(p))_{j\in J(p)}$ is a linearly independent family in $T^*_p(X)$.

\subsubsection{Logarithmic differential forms for RNCDs}\label{sect:ldffrncd}
\label{sect:1:2}

Let $D$ be a divisor. For $k$ an integer $\geq 0$, the sheaf $\Omega_X^k(\mathrm{log}D)$ is 
the subsheaf of $\Omega_X^k(*D)$ whose local sections 
are differentials $\alpha$ such that both $\alpha$ and $d\alpha$ are regular except for a possible simple pole along $D$. 
The definition shows that if $D$ is a RNCD, then $\Omega_X^k(\mathrm{log}D)$ can also be defined as follows (see \cite{EV}). For $p$ 
as above, the space 
of sections of this sheaf over $U(p)$ is given by the linear span, over all subsets $S\subset J(p)$, of all the differentials 
$(\bigwedge_{s\in S}(df_s/f_s))\wedge a_S$, where $a_S$ lies in $\Gamma(U(p),\Omega^{k-|S|}_X)$. 
This alternative definition shows that the collection of sheaves $\Omega_X^\bullet(\mathrm{log}D)$ is stable under the 
differential and the wedge product.  

\subsubsection{Poincar\'e residue: a sheaf morphism $\Omega^k(X,\mathrm{log}\delta)^{rat}\to(\Omega^{k-1}_\delta)^{rat}$}
\label{subsect:pr}

Let $\delta$ be a smooth irreducible codimension 1 subvariety of $X$. 

Let $k\geq 0$. According to \S\ref{sect:sheaves}, $\Omega^k_X(\mathrm{log}\delta)_\delta$ is the subsheaf of $(\Omega^k_X)^{rat}$ 
defined as follows. 
For $U$ an open subset of $X$, in which $\delta$ is defined by an equation $f=0$, where $f\in\Gamma(U,\mathcal O_X)$, 
the space $\Gamma(U,\Omega^k_X(\mathrm{log}\delta)_\delta)$ is the set of differentials in $\Gamma_{rat}(X,\Omega^k_X)$ of the form 
$\omega=(df/f)\wedge a+b$, where $a\in\Gamma(U,(\Omega^{k-1}_X)_\delta)$ and $b\in\Gamma(U,(\Omega^k_X)_\delta)$. 
In other terms, this space is the space of rational forms $\alpha$ on $U$, such that both $\alpha$ and $d\alpha$ have at most a simple 
pole at $\delta\cap U$. The map taking $\omega$ to the restriction of $a$ to $\delta$ is well-defined and induces a sheaf morphism to 
the sheaf of rational differentials on $\delta$
$$
\mathrm{Res}^{(k)}_\delta:\Omega^k_X(\mathrm{log}\delta)_\delta\to(\Omega^{k-1}_\delta)^{rat}.  
$$
Taking global sections over $X$ of this sheaf morphism, we obtain a linear map 
$$
\Gamma_{rat}(X,\Omega^k_X)\hookleftarrow \Gamma(X,\Omega^k_X(\mathrm{log}\delta)_\delta)\stackrel{\mathrm{Res}^{(k)}_\delta}{\to}
\Gamma(\delta,(\Omega^{k-1}_\delta)^{rat})=\Gamma_{rat}(\delta,\Omega^{k-1}_\delta). 
$$

\subsection{Vector spaces and maps attached to special divisors} 

Let $X$ be an irreducible, smooth $\mathbf k$-variety. We fix a SD $D=\cup_{i\in I}D_i$ in $X$ (for convenience, 
we assume $I$ to be ordered). 

\subsubsection{Vector spaces attached to special divisors}

Recall that for each $i\in I$, the space of global sections $\Gamma(X,\Omega^1_X(\mathrm{log}D_i))$ is a subspace of
$\Gamma_{rat}(X,\Omega^1_X)$. The sum of these spaces is a subspace 
\begin{equation}\label{space:H1}
\bm{\Omega}^1:=\sum_{i\in I}\Gamma(X,\Omega^1_X(\mathrm{log}D_i))\subset\Gamma_{rat}(X,\Omega^1_X). 
\end{equation} 
Similarly, for each pair $\{i,j\}\in\mathcal P_2(I)$ (the set of parts of $I$ with cardinality 2), the space of global sections 
$\Gamma(X,\Omega^2_X(\mathrm{log}D_i\cup D_j))$ is a subspace of
$\Gamma_{rat}(X,\Omega^2_X)$. The sum of these spaces is a subspace 
\begin{equation}\label{space:H2}
\bm{\Omega}^2:=\sum_{\{i,j\}\in\mathcal P_2(I)}\Gamma(X,\Omega^2_X(\mathrm{log}D_i\cup D_j))\subset\Gamma_{rat}(X,\Omega^2_X). 
\end{equation}

\subsubsection{Maps attached to special divisors}\label{subsect:maps}

Let $i\in I$. 
As $\Omega_X^\bullet(\mathrm{log}D_i)$ is stable under the wedge product, there is a commutative diagram 
$$
\xymatrix{
\Lambda^2(\Gamma(X,\Omega^1_X(\mathrm{log}D_i)))
\ar[d]\ar@{^(->}[r]& \ar[d]\Lambda^2(\Gamma_{rat}(X,\Omega^1_X))\\
\Gamma(X,\Omega^2_X(\mathrm{log}D_i))\ar@{^(->}[r]& \Gamma_{rat}(X,\Omega^2_X)}
$$
Let $\{i,j\}\in\mathcal P_2(I)$. As there are injective morphisms $\Omega_X^\bullet(\mathrm{log}D_i)\to\Omega_X^\bullet(\mathrm{log}D_i\cup D_j)$, 
$\Omega_X^\bullet(\mathrm{log}D_j)\to\Omega_X^\bullet(\mathrm{log}D_i\cup D_j)$, and as $\Omega_X^\bullet(\mathrm{log}D_i\cup D_j)$ 
is stable under the wedge product, there is a commutative diagram 
$$
\xymatrix{
\Gamma(X,\Omega^1_X(\mathrm{log}D_i))\otimes \Gamma(X,\Omega^1_X(\mathrm{log}D_j))
\ar@{^(->}[d]\ar@{^(->}[r]& \ar@{=}[d]\Gamma_{rat}(X,\Omega^1_X)^{\otimes 2}\\
\Gamma(X,\Omega^1_X(\mathrm{log}D_i\cup D_j))^{\otimes 2}\ar[d]& \Gamma_{rat}(X,\Omega^1_X)^{\otimes 2}\ar[d]\\ 
\Gamma(X,\Omega^2_X(\mathrm{log}D_i\cup D_j))\ar@{^(->}[r]& \Gamma_{rat}(X,\Omega^2_X)}
$$
The direct sum of these diagrams, together with the decomposition 
\begin{align*}
& \Lambda^2(\oplus_{i\in I}\Gamma(X,\Omega^1_X(\mathrm{log}D_i))) \\ & \simeq
\oplus_{i\in I}\Lambda^2(\Gamma(X,\Omega^1_X(\mathrm{log}D_i)))\oplus\Big(\oplus_{\{i,j\}\in\mathcal P_2(I)}
\Gamma(X,\Omega^1_X(\mathrm{log}D_i))\otimes \Gamma(X,\Omega^1_X(\mathrm{log}D_j))\Big)
\end{align*}
gives a commutative diagram 
$$
\xymatrix{
\Lambda^2(\oplus_{i\in I}\Gamma(X,\Omega^1_X(\mathrm{log}D_i)))
\ar[d]\ar[r]& \ar[d]\Lambda^2(\Gamma_{rat}(X,\Omega^1_X))\\
\oplus_{i\in I}\Gamma(X,\Omega^2_X(\mathrm{log}D_i))\oplus
\Big( \oplus_{\{i,j\}\in\mathcal P_2(I)}\Gamma(X,\Omega^2_X(\mathrm{log}D_i\cup D_j)) \Big) \ar[r]& \Gamma_{rat}(X,\Omega^2_X)}
$$
As the images of the horizontal maps are respectively $\Lambda^2(\sum_{i\in I}\Gamma(X,\Omega^1_X(\mathrm{log}D_i)))$ and 
$$
\sum_{\{i,j\}\in\mathcal P_2(I)}\Gamma(X,\Omega^2_X(\mathrm{log}D_i\cup D_j)),
$$ 
this gives: 
\begin{lemma} There is a commutative diagram 
\begin{equation}\label{comm:diag:wedge}
\xymatrix{
\Lambda^2(\bm{\Omega}^1)=\Lambda^2(\sum_{i\in I}\Gamma(X,\Omega^1_X(\mathrm{log}D_i)))
\ar[d]\ar@{^(->}[r]& \ar[d]\Lambda^2(\Gamma_{rat}(X,\Omega^1_X))\\
\bm{\Omega}^2=\sum_{\{i,j\}\in\mathcal P_2(I)}\Gamma(X,\Omega^2_X(\mathrm{log}D_i\cup D_j))\ar@{^(->}[r]& \Gamma_{rat}(X,\Omega^2_X)}
\end{equation}
\end{lemma}
We denote the resulting map by $\owedge:\Lambda^2(\bm{\Omega}^1)\to\bm{\Omega}^2$. 
The product in $\Lambda^\bullet(\bm{\Omega}^1)$
will be denoted $\underline\wedge$, and we use the notation $a\owedge b:=\owedge(a\underline\wedge b)$ for 
$a,b\in\bm{\Omega}^1$. 

As $\Omega_X^\bullet(\mathrm{log}D_i)$ is stable under the differential, there is a commutative diagram 
$$
\xymatrix{
\Gamma(X,\Omega^1_X(\mathrm{log}D_i))
\ar_{d}[d]\ar@{^(->}[r]& \ar^d[d]\Gamma_{rat}(X,\Omega^1_X)\\
\Gamma(X,\Omega^2_X(\mathrm{log}D_i))\ar@{^(->}[r]& \Gamma_{rat}(X,\Omega^2_X)}
$$
The direct sum of these diagrams yields a commutative diagram 
$$
\xymatrix{
\oplus_{i\in I}\Gamma(X,\Omega^1_X(\mathrm{log}D_i))
\ar_{d}[d]\ar[r]& \ar^d[d]\Gamma_{rat}(X,\Omega^1_X)\\
\oplus_{i\in I}\Gamma(X,\Omega^2_X(\mathrm{log}D_i))\ar[r]& \Gamma_{rat}(X,\Omega^2_X)}
$$
Taking images of the horizontal maps, we obtain: 
\begin{lemma} There is a commutative diagram 
\begin{equation}\label{comm:diag:diff}
\xymatrix{
\bm{\Omega}^1=\sum_{i\in I}\Gamma(X,\Omega^1_X(\mathrm{log}D_i))
\ar_{d}[d]\ar@{^{(}->}[r]& \ar^d[d]\Gamma_{rat}(X,\Omega^1_X)\\
\bm{\Omega}^2=\sum_{i\in I}\Gamma(X,\Omega^2_X(\mathrm{log}D_i))\ar@{^{(}->}[r]& \Gamma_{rat}(X,\Omega^2_X)}
\end{equation}
The resulting map is denoted $d:\bm{\Omega}^1\to\bm{\Omega}^2$. 
\end{lemma}

Let $i\in I$. For any $\{k,l\}\in\mathcal P_2(I)$, both sheaves $\Omega^2_X(\mathrm{log}D_k\cup D_l)$ and $\Omega^2_X(\mathrm{log}D_i)_{D_i}$ 
are subsheaves of $(\Omega^2_X)^{rat}$. Inspection shows that there exists a natural sheaf morphism 
$$
\Omega^2_X(\mathrm{log}D_k\cup D_l)\to\Omega^2_X(\mathrm{log}D_i)_{D_i}
$$
making the diagram 
$$\xymatrix{
\Omega^2_X(\mathrm{log}D_k\cup D_l)\ar[d]\ar[r]& (\Omega^2_X)^{rat}\\ 
\Omega^2_X(\mathrm{log}D_i)_{D_i}\ar[ur]& }
$$
commute. Taking global sections over $X$, one obtains a linear map 
$$
\Gamma(X,\Omega^2_X(\mathrm{log}D_k\cup D_l))\to \Gamma(X,\Omega^2_X(\mathrm{log}D_i)_{D_i}). 
$$
Taking sums over $\{k,l\}\in\mathcal P_2(I)$, one obtains the upper morphism of the following diagram
$$
\xymatrix{
\oplus_{\{k,l\}\in\mathcal P_2(I)}\Gamma(X,\Omega^2_X(\mathrm{log}D_k\cup D_l))\ar[r]\ar@{>>}[d]& \Gamma(X,\Omega^2_X(\mathrm{log}D_i)_{D_i})
\ar@^{^(->}[d]\\ 
\sum_{\{k,l\}\in\mathcal P_2(I)}\Gamma(X,\Omega^2_X(\mathrm{log}D_k\cup D_l))\ar@^{^(->}[r]&\Gamma_{rat}(X,\Omega^2_X)
}
$$ 
The composition of the left and bottom arrows of this diagram is the decomposition of the map 
$\oplus_{\{k,l\}\in\mathcal P_2(I)}\Gamma(X,\Omega^2_X(\mathrm{log}D_k\cup D_l))\to\Gamma_{rat}(X,\Omega^2_X)$ induced by the 
sum, and the composition of the top and right arrows is also shown to coincide with the same map, 
so that the diagram commutes. Inspecting the diagram, one obtains the inclusion 
$$
\sum_{\{k,l\}\in\mathcal P_2(I)}\Gamma(X,\Omega^2_X(\mathrm{log}D_k\cup D_l))\subset \Gamma(X,\Omega^2_X(\mathrm{log}D_i)_{D_i}). 
$$
One may compose this map with $\mathrm{Res}_{D_i}^{(2)}:\Gamma(X,\Omega^2_X(\mathrm{log}D_i)_{D_i})\to
\Gamma_{rat}(D_i,\Omega^1_{D_i})$. 
Examining the form of the images of elements of $\Gamma(X,\Omega^2_X(\mathrm{log}D_k\cup D_l))$ through the composed
map, one sees that these images are 0 if $k,l\neq i$ and lie in $\Gamma(D_i,\Omega^1_{D_i}(\mathrm{log}D_i\cap D_l))$
(resp., in $\Gamma(D_i,\Omega^1_{D_i}(\mathrm{log}D_i\cap D_k))$) 
if $k=i$ (resp., if $l=i$), one obtains: 
\begin{lemma} \label{lemma:143}
There is a unique map 
\begin{equation}\label{residue:2}
\mathrm{Res}_{D_i}^{(2)}:\bm{\Omega}^2=\sum_{\{k,l\}\in\mathcal P_2(I)}\Gamma(X,\Omega^2_X(\mathrm{log}D_k\cup D_l))\to\sum_{j\in I|j\neq i}\Gamma(D_i,\Omega^1_{D_i}(\mathrm{log}D_i\cap D_j))=:\mathbf D_i^1
\end{equation}
making the following diagram commute 
$$
\xymatrix{
\Gamma(X,\Omega^2_X(\mathrm{log}D_i)_{D_i})\ar^{\mathrm{Res}_{D_i}^{(2)}}[r] 
& 
\Gamma_{rat}(D_i,\Omega^1_{D_i}) 
\\ 
\sum_{\{k,l\}\in\mathcal P_2(I)}\Gamma(X,\Omega^2_X(\mathrm{log}D_k\cup D_l))
\ar^{\mathrm{Res}_{D_i}^{(2)}}[r]\ar@^{^(->}[u]
& 
\sum_{j\in I|j\neq i}\Gamma(D_i,\Omega^1_{D_i}(\mathrm{log}D_i\cap D_j))
\ar@^{^(->}[u]
}
$$
\end{lemma}

Dropping in this Subsection the normal crossing condition from the hypotheses on $\overline D$ and arguing as 
in the construction leading to Lemma \ref{lemma:143}, one obtains: 
\begin{lemma}
There is a unique linear map 
\begin{equation}\label{residue:1}
\mathrm{Res}_{D_i}^{(1)}:\sum_{k\in I}\Gamma(X,\Omega^1_X(\mathrm{log}D_k))\to \Gamma(D_i,\mathcal O_{D_i})
\end{equation}
making the following diagram 
$$
\xymatrix{
\Gamma(X,\Omega^1_X(\mathrm{log}D_i)_{D_i})\ar^{\mathrm{Res}_{D_i}^{(1)}}[r] 
& 
\Gamma_{rat}(D_i,\mathcal O_{D_i}) 
\\ 
\sum_{k\in I}\Gamma(X,\Omega^1_X(\mathrm{log}D_k))
\ar^{\mathrm{Res}_{D_i}^{(1)}}[r]\ar@^{^(->}[u]
& 
\Gamma(D_i,\mathcal O_{D_i})
\ar@^{^(->}[u]
}
$$
commute.
The restriction of $\mathrm{Res}_{D_i}^{(1)}$ to each $\Gamma(X,\Omega^1_X(\mathrm{log}D_k))$, $k\neq i$, is zero. 
\end{lemma}

Pick $i\in I$. Replacing the collection $(X,I,(D_i)_{i\in I},D,i)$ by $(D_i,I-\{i\},(D_j\cap D_i)_{j\in I-\{i\}},
D\cap D_i,j)$, the map (\ref{residue:1}) yields a linear map 
\begin{equation}\label{residue:1:i}
\mathrm{Res}_{D_i\cap D_j}^{(1)}:\mathbf D_i^1=\sum_{j\in I-\{i\}}\Gamma(D_i,\Omega^1_{D_i}(\mathrm{log}D_i\cap D_j))
\to \Gamma(D_i\cap D_j,\mathcal O_{D_i\cap D_j})=:\mathbf D_{ij}^0. 
\end{equation}

The sum of residue maps is a map  
$$
\Gamma_{rat}(X,\Omega^1_X)\hookleftarrow\sum_{i\in I} \Gamma(X,\Omega^1_X(\mathrm{log}D_i))\to \oplus_{i\in I}
\Gamma(D_i,\mathcal O_{D_i}). 
$$
The elements of the middle space are rational differentials on $X$, which are regular except for simple poles at $D_i$. 
The elements of the kernel of the map $\sum_{i\in I} \Gamma(X,\Omega^1_X(\mathrm{log}D_i))\to \oplus_{i\in I}
\Gamma(D_i,\mathcal O_{D_i})$ are therefore regular everywhere, that is, elements of $\Gamma(X,\Omega^1_X)$. We derive from there
the exact sequence 
\begin{equation}\label{exact:sequence}
0\to \Gamma(X,\Omega^1_X)\to \sum_{i\in I} \Gamma(X,\Omega^1_X(\mathrm{log}D_i))\to \oplus_{i\in I}
\Gamma(D_i,\mathcal O_{D_i}). 
\end{equation}

Summarizing the results of this \S, we obtain: 
\begin{lemma}
To $X$ and its SD $D=\cup_{i\in I}D_i$ are attached: 
\begin{itemize}
\item vector spaces $\bm{\Omega}^1$ and $\bm{\Omega}^2$ given by (\ref{space:H1}) and (\ref{space:H2}); 
\item linear maps $\owedge:\Lambda^2(\bm{\Omega}^1)\to\bm{\Omega}^2$ and $d:\bm{\Omega}^1\to\bm{\Omega}^2$ 
(see (\ref{comm:diag:wedge}) 
and (\ref{comm:diag:diff})); 
\item vector spaces $(\mathbf D_i^1)_{i\in I}$ and $(\mathbf D_{ij}^0)_{\{i,j\}\in\mathcal P_2(I)}$ (see (\ref{residue:2}) and (\ref{residue:1:i})); 
\item maps $\mathrm{Res}_{D_i}^{(2)}:\bm{\Omega}^2\to\mathbf D_i^1$ and $\mathrm{Res}_{D_i\cap D_j}^{(1)}:
\mathbf D_i^1\to\mathbf D_{ij}^0$;  
\item an exact sequence $0\to \Gamma(X,\Omega_X^1)
\to\bm{\Omega}^1\stackrel{\oplus_{i\in I}\mathrm{Res}_{D_i}^{(1)}}{\longrightarrow}\oplus_{i\in I}\Gamma(D_i,\mathcal O_{D_i})$ (see (\ref{exact:sequence})).  
\end{itemize}
\end{lemma}

\subsection{Tensor categories of geometric origin} For 
$X$ a smooth, irreducible quasiprojective variety over $\mathbf k$, we define $\mathrm{VBFC}(X)$ to be the category of pairs 
$(\mathcal F,\nabla_{\mathcal F})$, where $\mathcal F$ is a vector bundle over $X$ and $\nabla_{\mathcal F}:
\mathcal F\to\mathcal F\otimes_{\mathcal O_X}\Omega^1_X$ is a flat connection on $\mathcal F$. When equipped with the tensor 
product of vector bundles with connection, it is a symmetric tensor category with unit object the pair $(\mathcal O_X,d:\mathcal O_X\to
\Omega^1_X)$. 

If $D\subset X$ is a divisor, we define $\mathrm{VBFC}(X,D)$ to be the tensor category of vector bundles over $X$ with flat connection over 
$X-D$ and simple poles at $D$.

\section{The geometric setup}

\subsection{Geometric data}\label{sect:geomdata}

We give ourselves the following data: 
\begin{itemize}
\item a smooth projective variety $X_0$; 
\item an affine fibration $p:X\to X_0$; 
\item a special divisor $D_0\subset X_0$ (see \S\ref{sect:divisors}); we set $D_0:=\cup_{\alpha\in I} D_{0\alpha}$, so that the $D_{0\alpha}$
are smooth and pairwise normal crossing;  
\item a resolution of singularities $\pi_0:\tilde X_0\to X_0$ such that $\tilde D_0:=\pi^{-1}_0(D_0)$ is a normal crossing divisor (NCD). 
\end{itemize} 
We define $D\subset X$ by $D:=p^{-1}(D_0)$. We then set $\tilde X:=\tilde X_0\times_{X_0}X$ and $\tilde D\subset\tilde X$
by $\tilde D:=\tilde D_0\times_{X_0}X$. Then there is a commutative diagram 
\begin{equation}\label{diag:ddxxxxdd}
\xymatrix{
\tilde D\ar[rrr]\ar[ddd]\ar@{^{(}->}[rd] & & & D\ar@{_{(}->}[ld]\ar[ddd]\\ 
 & \tilde X\ar^\pi[r]\ar_{\tilde p}[d]&X\ar^p[d] & \\ 
 & \tilde X_0\ar^{\pi_0}[r]&X_0 & \\ 
\tilde D_0\ar@{^{(}->}[ru]\ar[rrr]  & & & D_0\ar@{_{(}->}[lu]}
\end{equation}
where the squares are Cartesian, and the divisors $\tilde D_0\subset\tilde X_0$ and  $\tilde D\subset\tilde X$ are NCDs. 

\subsection{A class of examples}\label{sect:examples}

We will work with the following example, based on the datum of an elliptic curve $E$ over $\mathbf k$ and an integer $n\geq 1$: 
\begin{itemize}
\item $X_0:=E^n$; 
\item the affine fibration $p:X\to X_0$ is $(E^\sharp)^n\to E^n$ (see \S\ref{sect:ellmat}); 
\item the divisor $D_0\subset X_0$ is $\cup_{i<j\in[n]}D_{ij}$;    
\item the resolution of singularities $\pi_0:\tilde X_0\to X_0$ is obtained by the Hironaka desingularization theorem (explicit examples 
may be derived from \cite{FM} or \cite{U}).  
\end{itemize} 

\section{The main results}\label{sect:tmr}

In \S\ref{sect:summary}, we explain the main result of this paper, which consists of the construction of a sequence of equivalences 
of tensor categories. The construction of these equivalences is explained  in the following subsections of the present \S. In the last 
subsection (\S\ref{sect:rwtuzc}), we explain the relation between this construction and the KZB connection from \cite{CEE}.    

\subsection{The main result: construction of a tensor equivalence} \label{sect:summary}

Let $X_0,D_0,\ldots$ be as in \S\ref{sect:examples}. 

We will construct a sequence of equivalences of tensor categories
$$
\mathrm{VBFC}(X_0-D_0)_{unip}
\stackrel{(a)}{\simeq} \mathrm{VBFC}(X-D)_{unip}
\stackrel{(b)}{\simeq} \mathrm{VBFC}(\tilde X-\tilde D)_{unip}
\stackrel{(c)}{\simeq}\mathrm{VBFC}(\tilde X,\tilde D)_{unip}
$$
\begin{equation}\label{sequence:equivalences}
\stackrel{(d)}{\simeq} \mathrm{VBFC}(X,D)_{unip}\stackrel{(e)}{\simeq} 
\mathrm{Vec}(X,D)\stackrel{(f)}{\simeq} \mathrm{Mod}(\mathfrak G)_{unip}\stackrel{(g)}{\simeq} \mathrm{Mod}(\mathfrak t_{1,n})_{unip},   
\end{equation}
where $\mathrm{Vec}(X,D)$ is the category defined in \S\ref{acvo} and $\mathrm{Mod}(\mathfrak g)$ denote the category of finite 
dimensional modules over a Lie algebra $\mathfrak g$. In this diagram, (a) comes from \S\ref{tcaf}, (b) comes from 
$X-D=\tilde X-\tilde D$, (c) comes from taking the unipotent part of 
the isomorphism from \S\ref{de}, (d) is the result from \S\ref{sect:ceibd}, (e) comes from \S\ref{acvo}, (f) comes from \S\ref{alagatapxd}, 
(g) comes from \S\ref{cogitpc}. All this induces an equivalence of tensor categories 
$\mathrm{VBFC}(C(E,n))_{unip}\simeq\mathrm{Mod}(\mathfrak t_{1,n}^{\mathbb C})_{unip}$. 

The fiber functor on $\mathrm{VBFC}(X_0-D_0)_{unip}$ by this equivalence is the functor 
$$
F:\mathrm{VBFC}(X_0-D_0)_{unip}\to\mathrm{Vec}
$$ 
constructed as follows. To a unipotent bundle with flat connection $(\mathcal E,\nabla)$ on $X_0-D_0$, we attach its lift 
$(\tilde{\mathcal E},\tilde\nabla)$ to $X-D$. One then extends this bundle with flat connection to a pair $(\overline{\mathcal E},\overline\nabla)$
of a bundle $\overline{\mathcal E}$ over $X$ with a flat connection $\overline\nabla$ with regular singularities on $D$. Unipotent 
bundles over $X$ are trivial, so $\overline{\mathcal E}$ is canonically isomorphic to $\Gamma(X,\overline{\mathcal E})\otimes\mathcal O_X$. 
One then sets 
\begin{equation}\label{def:F}
F(\mathcal E,\nabla):=\Gamma(X,\overline{\mathcal E}). 
\end{equation}

It follows from (\ref{sequence:equivalences}) and from the computation (\ref{def:F}): 
\begin{theorem} 
The functor $F:\mathrm{VBFC}(X_0-D_0)_{unip}\to\mathrm{Vec}$ defined by (\ref{def:F}) is a fiber functor. There is a tensor equivalence
$$
\mathrm{VBFC}(C(E,n))_{unip}\simeq\mathrm{Mod}(\mathfrak t_{1,n}^{\mathbb C})_{unip},  
$$
compatible with the fiber functors on both sides. 
\end{theorem}

\remark{This result could as well be obtained as the following sequence of equivalences of tensor categories
$$
\mathrm{VBFC}(X_0-D_0)_{unip}\stackrel{(b')}{\simeq}\mathrm{VBFC}(\tilde X_0-\tilde D_0)_{unip}
\stackrel{(c')}{\simeq}\mathrm{VBFC}(\tilde X_0,\tilde D_0)_{unip}
\stackrel{(d')}{\simeq}\mathrm{VBFC}(X_0,D_0)_{unip}
$$
$$
\stackrel{(a')}{\simeq}\mathrm{VBFC}(X,D)_{unip}\stackrel{(e)}{\simeq}\mathrm{Vec}(X,D)
\stackrel{(f)}{\simeq}\mathrm{Mod}(\mathfrak G)_{unip}
\stackrel{(g)}{\simeq}\mathrm{Mod}(\mathfrak t_{1,n}^{\mathbb C})_{unip}, 
$$
based on the "left-bottom" part of diagram (\ref{diag:ddxxxxdd}) (as opposed to the "top-right" part). }

\subsection{Tensor categories and affine fibrations (equiv. (a))} \label{tcaf}

If $Y$ is an affine variety over $\mathbf k$, then $H^1(Y,\mathcal O_Y)=\mathrm{Ext}^1_{Y}(\mathcal O_Y,\mathcal O_Y)=0$. 
Applying this for $Y=\mathbb A^m$, we obtain that if $(\mathcal E,\nabla)$ is an object of $\mathrm{VBFC}(\mathbb A^m)_{unip}$, 
and if $0=(\mathcal E_0,\nabla_0)\subset\cdots\subset(\mathcal E_N,\nabla_N)=(\mathcal E,\nabla)$ is a filtration whose subquotients 
are isomorphic to the trivial object, then the diagram $0=\mathcal E_0\subset\cdots\subset\mathcal E_N=\mathcal E$ is isomorphic 
to the diagram $0=\mathcal O_{\mathbb A^m}^{\oplus 0}\subset\cdots\subset\mathcal O_{\mathbb A^m}^{\oplus N}$. The image 
of the connection $\nabla$ is then necessarily of the form $d+A$, where $A$
is a strictly upper-triangular $N\times N$ matrix with coefficients in $\Omega^1_{\mathbb A^m}$, such that $(d+A)^2=0$. Using the
acyclicity of $\Omega^0_{\mathbb A^m}\to\Omega^1_{\mathbb A^m}\to\Omega^2_{\mathbb A^m}$ (vanishing of the first De Rham 
cohomology group of $\mathbb A^m$), one constructs a unipotent $N\times N$ matrix $n$ such that $d+A=ndn^{-1}$. Therefore 
$(\mathcal E,\nabla)\simeq(\mathcal O_{\mathbb A^m}^{\oplus N},d)$: any object of $\mathrm{VBFC}(\mathbb A^m)$ is isomorphic to 
$(\mathcal O_{\mathbb A^m},d)^{\oplus N}$ for some $N\geq 0$. All this implies: 
\begin{lemma}\label{lemma:eq}
The functor $\mathrm{VBFC}(\mathbb A^m)\to\mathrm{Vec}$ of global sections defines an equivalence of categories 
$$
\mathrm{VBFC}(\mathbb A^m)_{unip}\stackrel{\sim}{\to}\mathrm{Vec}. 
$$
A quasi-inverse is given by the operation $-\otimes(\mathcal O_{\mathbb A^m},d)$ of taking the tensor product of a 
vector space with the unit object $(\mathcal O_{\mathbb A^m},d)$. 
\end{lemma}

Let $Y$ be a $\mathbf k$-variety and $\mathcal E$ be a vector bundle over $Y$.  According to 
\cite{Del}, \S I.2.2.4, p. 6, a connection $\nabla$ over $\mathcal E$ is equivalent to the data, for any $\mathbf k$-scheme $S$, 
any pair of morphisms $x,y:S\to Y$, infinitely close in the sense that the resulting morphism $S\to Y\times Y$
factorizes through the first infinitesimal neighborhood of the diagonal,  
of an isomorphism $\gamma_{x,y}:x^*\mathcal E\simeq y^*\mathcal E$, 
functorial with respect to base change and such that $\gamma_{x,x}$ is the identity. Moreover, $\nabla$ is flat iff one has 
$\gamma_{y,z}\circ\gamma_{x,y}=\gamma_{x,z}$. 

Let $p:Y^\#\to Y$ be a morphism of smooth, quasiprojective varieties over $\mathbf k$, locally isomorphic to a projection 
$\mathbb A^m\times U\to U$. Let $({\mathcal E}^\#,\nabla^\#)\in\mathrm{VBFC}(Y^\#)_{unip}$. For $x:S\to Y$, 
set $\mathcal E_x:=\Gamma_{hor}({\mathcal E}_{Y_x^\#},\nabla^\#_{|Y_x^\#})$, where
$Y^\#_x:=Y^\#\times_x S$ and $({\mathcal E}^\#_{Y_x^\#},\nabla^\#_{|Y^\#_x})$ is the pull-back of 
$({\mathcal E}^\#,\nabla^\#)$ to $Y_x^\#$. Using a trivialization of the projection, one obtains an isomorphism 
$\gamma_{x,y}:\mathcal E_x\simeq \mathcal E_y$, satisfying the base change and flatness conditions. In this way, and using 
Lemma \ref{lemma:eq}, we define a tensor functor 
$$
p_*:\mathrm{VBFC}(Y^\#)_{unip}\to\mathrm{VBFC}(Y)_{unip}. 
$$
On the other hand, pull-back defines a tensor functor $\mathrm{VBFC}(Y)\to\mathrm{VBFC}(Y^\#)$, which 
restricts to a tensor functor 
$$
p^*:\mathrm{VBFC}(Y)_{unip}\to\mathrm{VBFC}(Y^\#)_{unip}. 
$$
Arguing locally as in the proof of Lemma \ref{lemma:eq}, one obtains: 
\begin{lemma}\label{lemma:eq:global}
Let $p:Y^\#\to Y$ be a morphism of smooth, quasiprojective varieties over $\mathbf k$, locally isomorphic to a projection 
$\mathbb A^m\times U\to U$. The functors $p_*,p^*$ define quasi-inverse category equivalences between 
$\mathrm{VBFC}(Y)_{unip}$ and $\mathrm{VBFC}(Y^\#)_{unip}$. 
\end{lemma}

\remark{
The present functor $p_*$ is related to the direct image functor $p_*$ for $D$-modules as follows
$$
\xymatrix{
\text{VBFC}(Y^\#)_{unip}\ar_{p_*}[d]\ar^{can}[rr] & & \mathcal D_{Y^\#}\text{-mod}\ar^{p_*}[d]\\ 
\text{VBFC}(Y)_{unip}\ar^{can}[r]& \mathcal D_Y\text{-mod}\ar^{-[-d]}[r]& D^b(\mathcal D_Y\text{-mod})}
$$
where $-[-d]$ is the shift by degree $-d$. }

\subsection{Deligne extension (equiv. (c))}\label{de}

In this \S, we assume that $Y$ is a smooth quasi-projective $\mathbf k$-variety and that $D\subset Y$ is a NCD.  

Let $\tau:\mathbb C/\mathbb Z\to\mathbb C$ be a lift of the canonical projection $\mathbb C\to\mathbb C/\mathbb Z$. 

Let $\mathrm{VBFC}(Y-D)_{unip.mon.}$ be the full subcategory of $\mathrm{VBFC}( Y- D)$ 
corresponding to the vector bundles with flat connection over $ Y- D$ for which the monodromy around each 
component of $ D$ is unipotent. 

Let $\mathrm{VBFC}( Y, D)_{nilp.res.}$ (resp., $\mathrm{VBFC}( Y, D)_\tau$) be the full subcategory of 
$\mathrm{VBFC}( Y, D)$ corresponding to the vector bundles over $ Y$ with flat connection on $ Y- D$ 
with at most simple poles at $ D$, whose residue at each component of $ D$ is nilpotent (resp., has spectrum contained 
in the image of $\tau$). Contrary to $\mathrm{VBFC}( Y,D)_\tau$, the categories $\mathrm{VBFC}(Y,D)_{nilp.res.}$ and 
$\mathrm{VBFC}(Y-D)_{unip.mon.}$ are tensor subcategories of their ambient categories. 

We decorate by the superscript $an$ the analytic analogues of these categories. 

Restriction induces a tensor functor 
\begin{equation}\label{res}
\widetilde{res}:\mathrm{VBFC}(Y,D)^{an}\to\mathrm{VBFC}(Y-D)^{an},
\end{equation} 
which further restricts to a functor 
$$
res:\mathrm{VBFC}(Y,D)_{nilp.res.}^{an}\to\mathrm{VBFC}(Y-D)_{unip.mon.}^{an}.
$$ 
In \cite{Del}, Prop. 5.2, p. 91, Deligne constructed a tensor functor 
\begin{equation}\label{canext}
can.ext.:\mathrm{VBFC}(Y-D)_{unip.mon.}^{an}\to
\mathrm{VBFC}(Y,D)_{nilp.res.}^{an}
\end{equation}
(canonical extension), such that $res\circ can.ext.=id$. This functor is extended in \cite{Del}, Prop. 5.4, to a functor 
$$
\widetilde{can.ext.}:\mathrm{VBFC}(Y-D)^{an}\to\mathrm{VBFC}(Y,D)_\tau^{an}
$$
such that $\widetilde{res}\circ\widetilde{can.ext.}=id$. In \cite{Bri}, Cor. 2, it is proved that $\widetilde{res}$ is a category equivalence. 
It follows that $res$ is a (tensor) category equivalence. Taking unipotent parts, one obtains:

\begin{lemma}
If $Y$ is a quasiprojective $\mathbf k$-variety and if $D\subset Y$ is a NCD, 
then there is an equivalence of tensor categories
$$
\mathrm{VBFC}(Y-D)_{unip}^{an}\simeq \mathrm{VBFC}(Y,D)_{unip}^{an}
$$ 
\end{lemma}

\remark{The category $\mathrm{VBFC}(Y,D)_{unip}$ is "much smaller" than $\mathrm{VBFC}(Y,D)_{unip.mon.}$. 
For example, if $Y=\mathbb P^1$ and $D$ is a set of points in $\mathbb A^1$, if $N\geq 1$, if $(A_p)_{p\in D}$ is a collection of 
nilpotent matrices in $M_N(\mathbf k)$ with $\sum_{p\in D}A_p=0$, and if $A:=\sum_{p\in D}A_p\cdot dz/(z-p)$, then the pair 
$(\mathcal O_Y^{\oplus N},d+A)$
lies in $\mathrm{Ob}(\mathrm{VBFC}(Y,D)_{unip.mon.})$, but it lies in $\mathrm{Ob}(\mathrm{VBFC}(Y,D)_{unip})$
only if the matrices $(A_p)_{p\in D}$ can be simultaneously be upper-triangularized. }

\remark{
The canonical extension map from \cite{Del,Bri} makes sense in the analytic category. To make sense of it in the 
algebraic category (in the unipotent case), one proceeds as follows: (1) start with $(\mathcal E,\nabla)$ over $Y-D$; 
(2) as $(\mathcal E,\nabla)$ has regular singularities, it admits an extension $(\overline{\mathcal E},\overline\nabla)$ with 
simple poles; (3) because of unipotence of monodromy, residues should have integer eigenvalues; 
(4) this induces a splitting of the restriction of $\overline{\mathcal E}$ to $D$ according to generalized eigenvalues
of the residues; (5) one modifies $\overline{\mathcal E}$ in order that the eigenvalues of the residues become�$\ 0$; 
(6) the uniqueness of the extension is proved as follows: an isomorphism (over $Y-D$) between two extensions satisfies 
a differential equation, which implies that it has logarithmic growth. On the other hand, it is rational. Hence it is regular.  }

\subsection{Category equivalences induced by desingularization (equiv. (d))}\label{subsect:ceibd}

One proves in \S\ref{sect:ceibd}: 
\begin{lemma}\label{equivd}
Let $Y$ be a smooth quasiprojective $\mathbf k$-variety and $D\subset Y$ be a divisor. Let us fix a desingularization of 
$(Y,D)$, that is a morphism $\pi_Y:\tilde Y\to Y$ of smooth quasiprojective $\mathbf k$-varieties, such that the preimage 
$\tilde D:=\pi_Y^{-1}(D)$ of $D$ under $\pi_Y$  is a NCD in $\tilde Y$ and such that $\pi_Y$ restricts to an isomorphism 
$\tilde Y-\tilde D\simeq Y-D$.  There is a tensor equivalence 
$$
\mathrm{VBFC}(\tilde Y,\tilde D)_{unip}\simeq \mathrm{VBFC}(Y,D)_{unip}. 
$$
\end{lemma}

\subsection{A tensor category $\mathrm{Vec}(Y,D)$ (equiv. (e))}\label{acvo}

\subsubsection{A tensor category $\mathrm{Vec}(Y,D)$}\label{atcv}

Let $Y$ be a quasiprojective $\mathbf k$-variety such that $H^0(Y,\mathcal O_Y)=\mathbf k$ and $H^1(Y,\mathcal O_Y)=0$, 
and let $D\subset Y$ be a divisor. Let $\mathrm{Vec}(Y,D)$ be the following category: 
\begin{itemize}
\item objects are the pairs $(V,\omega)$, where $V$ is a finite dimensional vector space, and $\omega$ is an element of 
$\Gamma(Y,\Omega^1_Y(D))\otimes\mathrm{End}(V)$, such that $\omega$ is strictly compatible with some filtration 
of $V$ (i.e., satisfies $\omega(V^i)\subset \Gamma(Y,\Omega^1_Y(D))\otimes V^{i+1}$ where 
$V=V^0\supset V^1\supset\cdots\supset V^N=0$ 
is the filtration of $V$) and satisfies the Maurer-Cartan equation $d\omega+[\omega,\omega]=0$ 
(equality in $\Gamma(Y,\Omega^2_X(2D))\otimes\mathrm{End}(V)$); 

\item the set of morphisms from $(V,\omega)$ to $(V',\omega')$ is the 
set of linear maps $f:V\to V'$, such that $f\omega=\omega'f$ (equality in $\Gamma(Y,\Omega^1_Y(D))\otimes
\mathrm{Hom}_{\mathbf k}(V,V')$). 
\end{itemize}
If $(V,\omega)$ and $(V',\omega')$ are two objects, then: 
\begin{itemize}
\item 
their direct sum is defined as $(V,\omega)\oplus(V',\omega'):=(V\oplus V',\omega+\omega^{\prime})$, where
$\omega$ is induced by the canonical map $\mathrm{End}(V)\to\mathrm{End}(V\oplus V')$, and $\omega^{\prime}$
is defined in a similar way; 
\item their tensor product is defined as $(V,\omega)\otimes(V',\omega')
:=(V\otimes V',\omega\otimes 1+1\otimes\omega^{\prime})$, where $\omega\otimes 1$ is induced by the map 
$\Gamma(Y,\Omega^1_Y(D))\otimes\mathrm{End}(V)\to \Gamma(Y,\Omega^1_Y(D))\otimes\mathrm{End}(V)\otimes
\mathrm{End}(V')\simeq \Gamma(Y,\Omega^1_Y(D))\otimes\mathrm{End}(V\oplus V')$ given by tensor product with $\mathrm{id}_{V'}$
and $1\otimes\omega^{\prime}$ is defined similarly. 
\end{itemize}

\begin{lemma}
Let $Y$ be a projective $\mathbf k$-variety and let $D\subset Y$ be a divisor. Then 
$\mathrm{Vec}(Y,D)$ is a tensor category with unit object $(\mathbf k,0)$. 
\end{lemma}

\subsubsection{Isomorphism $\mathrm{VBFC}(X,D)_{unip}\simeq \mathrm{Vec}(X,D)$}

Assume that $(X,D)$ are as in \S\ref{sect:examples}. 

It is known that $H^1(E^\#,\mathcal O_{E^\#})=\mathbf k$. Together with the K\"unneth formula and $X=(E^\#)^n$, this implies 
that $H^1(X,\mathcal O_X)=0$. It follows that any diagram $0=\mathcal E_0\subset\cdots\subset \mathcal E_N=\mathcal E$
of vector bundles over $X$ with subquotients $\simeq\mathcal O_X$ is isomorphic to a diagram $0=\mathcal O_X\otimes 
V_0\subset\cdots\subset\mathcal O_X\otimes V_N=\mathcal O_X\otimes V$, where $0=V_0\subset\cdots\subset V_N=V$
is a maximal filtration of a finite dimensional vector space $V$. 

If then $(\mathcal E,\nabla)$ is an object of $\mathrm{VBFC}(X,D)_{unip}$, the above isomorphism necessarily takes $\nabla$ to a 
connection on $\mathcal O_X\otimes V$ of the form $d+\omega$, where $\omega$ is as described in the definition of the 
category $\mathrm{Vec}(X,D)$. We obtain in this way: 

\begin{lemma}\label{1201}
If $(X,D)$ are as in \S\ref{sect:examples}, then there is an equivalence of tensor categories 
$$
\mathrm{VBFC}(X,D)_{unip}\simeq \mathrm{Vec}(X,D). 
$$
\end{lemma}

\subsubsection{Full subcategories of $\mathrm{Vec}(X,D)$}

Let us come back to the framework of \S\ref{atcv}, so $Y$ is a projective $\mathbf k$-variety and $D\subset Y$ is a divisor.

Let $\Sigma\subset \Gamma(Y,\Omega^1_Y(D))$ be a vector subspace. Define $\mathrm{Vec}_\Sigma(Y,D)$
as the full subcategory of $\mathrm{Vec}(Y,D)$, where the objects are the pairs $(V,\omega)$ as in the definition of 
$\mathrm{Vec}(Y,D)$, such that $\omega\in\Sigma\otimes\mathrm{End}(V)$. Then:
\begin{lemma}
Let $Y$ be a projective $\mathbf k$-variety, let $D\subset Y$ be a divisor, and let $\Sigma\subset \Gamma(Y,\Omega^1_Y(D))$ 
be a vector subspace. Then $\mathrm{Vec}_\Sigma(Y,D)$ is a tensor subcategory of $\mathrm{Vec}(Y,D)$. 
\end{lemma}

\subsubsection{Equality $\mathrm{Vec}(Y,D)=\mathrm{Vec}_{\Sigma_{log}}(Y,D)$}

Let $(D_i)_{i\in I}$ be the components of $D$, so $D=\cup_{i\in I}D_i$; we assume that each divisor $D_i$ is smooth and that these divisors 
intersect pairwise transversally. Let 
$\Sigma_{log}\subset \Gamma(Y,\Omega^1_Y(D))$ be the subspace of differentials $\alpha$ defined by the 
following conditions: 
\begin{itemize}
\item[a)] $\alpha$ is a logarithmic form with poles only at $D$ (see \S\ref{sect:ldffrncd}); 
\item[b)] for each $i\in I$, $\mathrm{res}_{D_i}(\alpha)$ is regular of the whole of $D_i$. 
\end{itemize}

Lemmas \ref{l1} and \ref{l2} from \S\ref{sect:description} imply: 

\begin{lemma}\label{lemma:lxbapkv}
Let $Y$ be a projective $\mathbf k$-variety, let $D\subset Y$ be a divisor, then 
$\mathrm{Vec}(Y,D)=\mathrm{Vec}_{\Sigma_{log}}(Y,D)$. 
\end{lemma}

\begin{remark}
There are bundle injections $\Omega^1_Y(\mathrm{log}D_i)\hookrightarrow\Omega^1_Y(D_i)
\hookrightarrow\Omega^1_Y(\sum_{i\in I}D_i)$, inducing an injection 
$\sum_{i\in I}\Omega^1_Y(\mathrm{log}D_i)\hookrightarrow \Omega^1_Y(\sum_{i\in I}D_i)$. Then 
$\Sigma_{log}=\Gamma(Y,\sum_{i\in I}\Omega^1_Y(\mathrm{log}D_i))$. 
\end{remark}

\subsubsection{Computation of $\Sigma_{log}$}\label{sect:cos}

Let $(X,D)$ be as in \S\ref{sect:examples}. Recall that for each pair $(i,j)$ with $i<j\in[n]$, the space of global sections 
$\Gamma(X,\Omega^1_X(\mathrm{log}D_{ij}))$ is a subspace of $\Gamma_{rat}(X,\Omega^1_X)$. Recall from \S\ref{subsect:maps} 
that the sum of these spaces is a subspace 
$$
\bm{\Omega}^1=\sum_{i<j\in[n]}\Gamma(X,\Omega^1_X(\mathrm{log}D_{ij}))\subset \Gamma(X,\Omega^1_X(D)). 
$$

In \S\ref{sect:ehs}, we prove: 
\begin{lemma}\label{lemma:atxdaais}
Assuming that $(X,D)$ are as in \S\ref{sect:examples}, one has $\Sigma_{log}=\bm{\Omega}^1$. 
\end{lemma}

Combining Lemmas \ref{1201}, \ref{lemma:lxbapkv} and \ref{lemma:atxdaais}, one gets:
\begin{lemma}\label{lemma:ixdaais}
If $(X,D)$ are as in \S\ref{sect:examples}, then one has  
$\mathrm{Vec}(X,D)=\mathrm{Vec}_{\bm{\Omega}^1}(X,D)$. 
\end{lemma}

\subsection{A Lie coalgebra attached to a pair $(Y,D)$ (equiv. (f))}\label{alagatapxd}

Let again $(Y,D)$ be a pair of a quasiprojective $\mathbf k$-variety and a divisor $D\subset Y$, and let $\Sigma
\subset \Gamma(Y,\Omega^1_Y(D))$ be a vector subspace. 

\subsubsection{Relation between $\mathrm{Vec}_\Sigma(Y,D)$ and a category of comodules over a Lie coalgebra}

The composed map $\Sigma\subset \Gamma(Y,\Omega^1_X(D))\stackrel{d}{\to} \Gamma(Y,\Omega^2_X(2D)$ induced by the 
differential will be denoted 
$$
\Sigma\stackrel{d}{\to} \Gamma(Y,\Omega^2_X(2D)); 
$$
the wedge product of forms will be denoted $\owedge:\Lambda^2(\Gamma(Y,\Omega^1_X(D)))\to \Gamma(Y,\Omega^2_X(2D))$, 
and its composed map with the inclusion $\Lambda^2(\Sigma)\subset\Lambda^2(\Gamma(Y,\Omega^1_X(D)))$
will be denoted 
$$
\Lambda^2(\Sigma)\stackrel{\owedge}{\to} \Gamma(Y,\Omega^2_X(2D)).  
$$

The tensor algebra $T(\Sigma)$ is a commutative bialgebra, when equipped with the shuffle product $\sh$ and the 
deconcatenation coproduct $\Delta_{conc}$. 

Let 
$$
\mu:T(\Sigma)\to T(\Sigma)\otimes \Gamma(Y,\Omega^2_X(2D))\otimes T(\Sigma)
$$
be the map defined by 
$$
\mu([h_1|\ldots|h_n]):=\sum_{i=1}^n[h_1|\ldots|h_{i-1}]\otimes dh_i\otimes[h_{i+1}|\ldots|h_n]
+{1\over 2}\sum_{i=1}^{n-1}[h_1|\ldots|h_{i-1}]\otimes(h_i\owedge h_{i+1})\otimes[h_{i+2}|\ldots|h_n], 
$$
for $n\geq 0$, $h_1,\ldots,h_n\in\Sigma$ (in particular, $\mu(1)=0$). Then the following diagram commutes
$$
\xymatrix{
T(\Sigma)\ar^{\!\!\!\!\!\!\!\!\!\!\!\!\!\!\!\!\!\!\!\!\!\!\!\!\mu}[r]\ar_{\Delta_{conc}}[d]& T(\Sigma)\otimes \Gamma(Y,\Omega^2_X(2D))
\otimes T(\Sigma)
\ar^{\mathrm{id}\otimes\mathrm{id}\otimes\Delta_{conc}\oplus\Delta_{conc}\otimes\mathrm{id}\otimes\mathrm{id}}[d]\\
T(\Sigma)^{\otimes 2}\ar_{\!\!\!\!\!\!\!\!\!\!\!\!\!\!\!\!\!\!\!\!\!\!\!\!\mu\otimes\mathrm{id}\oplus\mathrm{id}\otimes\mu}[r]& 
\stackrel{(T(\Sigma)\otimes \Gamma(X,\Omega^2_Y(2D))\otimes T(\Sigma))\otimes T(\Sigma)\oplus}
{\stackrel{T(\Sigma)\otimes(T(\Sigma)\otimes \Gamma(Y,\Omega^2_Y(2D))\otimes T(\Sigma))}
{}}}
$$
This implies that $\mathbf C_\Sigma:=\mathrm{Ker}(\mu)$ is a subbialgebra of the shuffle bialgebra $T(\Sigma)$. 
One associates to it the Lie coalgebra $\mathfrak C_\Sigma:=\mathrm{Coprim}(\mathbf C_\Sigma)=\mathrm{Coker}(\mathbf C_\Sigma^{\otimes 2}\to\mathbf C_\Sigma,
a\otimes b\mapsto a\sh b-a\epsilon(b)-\epsilon(a)b)$, where $\sh$ is the product of $\mathbf C_\Sigma$ and $\epsilon:\mathbf C_\Sigma\to\mathbf k$
is its counit map.   

Finite dimensional comodules over $\mathbf C_\Sigma$ bijectively correspond to finite dimensional comodules 
over the Lie coalgebra $\mathfrak C_\Sigma$. 
These form a tensor category, denoted $\mathrm{Comod}(\mathfrak C_\Sigma)$. One then has a tautological equivalence of
tensor categories 
\begin{equation}\label{iso:16012017}
\mathrm{Vec}_\Sigma(Y,D_Y)\simeq\mathrm{Comod}(\mathfrak C_\Sigma)_{unip}. 
\end{equation} 

\subsubsection{Gradedness in a particular situation}

Assume that $(X,D)$ is as in \S\ref{sect:geomdata}, so $D=\cup_{i\in I}D_i$ is a special divisor in $X$. Then for each 
pair $\{i,j\}\in\mathcal P_2(I)$, the divisors $D_i$ and $D_j$ intersect transversally. The space of global sections 
$\Gamma(X,\Omega^2_X(\mathrm{log}D_i\cap D_j))$ is then a vector subspace of $\Gamma(X,\Omega^2_X(2D))$. 
Recall from \S\ref{subsect:maps} the space 
$$
\bm{\Omega}^2=\sum_{\{i,j\}\in\mathcal P_2(I)}\Gamma(X,\Omega^2_X(\mathrm{log}D_i\cup D_j))\subset  \Gamma(X,\Omega^2_X(2D)) 
$$
and the maps 
$$
d:\bm{\Omega}^1\to\bm{\Omega}^2,\quad\owedge:\Lambda^2(\bm{\Omega}^1)\to\bm{\Omega}^2. 
$$
In \S\S\ref{sec:5:1} and \ref{subsect:5:4}, we prove: 
\begin{lemma}
If $(X,D)$ is as in \S\ref{sect:examples}, then the maps $d,\owedge$ corestrict to maps 
$\bm{\Omega}^1\stackrel{d}{\to}\mathbf I$ and $\Lambda^2(\bm{\Omega}^1)\stackrel{\owedge}{\to}\mathbf I$, 
where $\bm{\Omega}^1$ and $\mathbf I$ are graded by $\mathbb N$ (the integers $\geq 0$) 
and have finite dimensional components. These maps have degree $0$. 
\end{lemma}
It follows that the Lie coalgebra $\mathfrak C_{\bm{\Omega}^1}$ is graded and has finite dimensional components. 

Define $\mathfrak G:=\mathfrak C_{\bm{\Omega}^1}^\vee$, the graded dual of $\mathfrak C_{\bm{\Omega}^1}$. 
This is a graded Lie algebra with finite dimensional components. Dualization sets up a tensor equivalence 
\begin{equation}\label{16012017}
\mathrm{Comod}(\mathfrak C_{\bm{\Omega}^1})_{unip}\simeq\mathrm{Mod}(\mathfrak G)_{unip}. 
\end{equation}
Combining (\ref{iso:16012017}) for $\Sigma=\bm{\Omega}^1$ and (\ref{16012017}), we get: 

\begin{lemma}
If $(X,D)$ is as in \S\ref{sect:examples}, then there is a tensor equivalence
$$
\mathrm{Vec}_{\bm{\Omega}^1}(X,D)\simeq\mathrm{Mod}(\mathfrak G)_{unip}. 
$$
\end{lemma}
This tensor equivalence identifies the obvious (forgetful) fiber functors of both sides with each other. 

Let $\omega\in(\mathfrak G\otimes\bm{\Omega}^1)^\wedge$ be the canonical element (where $(-)^\wedge$ is the degree completion). 
Then $d\omega+[\omega,\omega]=0$ (equality in the degree completion of $\mathfrak G\otimes\bm{\Omega}^2$). Therefore: 
 
\begin{lemma}\label{lemma:diafcottpebov}
$d+\omega$ is a flat connection on the trivial principal $\mathrm{exp}(\mathfrak G)$-bundle over $X-D$. 
\end{lemma}

\subsection{Computation of $\mathfrak G$ in the context of \S\ref{sect:examples} (equiv. (g))}\label{cogitpc} 

Recall the Lie algebra $\mathfrak t_{1,n}^{\mathbb C}$ from the Introduction. In \S\ref{section:iso:G:t} (Prop. \ref{prop:iso:G:t}), we prove: 

\begin{lemma}
There is an isomorphism of graded Lie algebras $\mathfrak G\simeq\mathfrak t_{1,n}^{\mathbb C}$. 
\end{lemma}

\subsection{Relation with the universal KZB connection}\label{sect:rwtuzc}

For $\tau\in\mathfrak H$, set $E_\tau:=\mathbb C/(\mathbb Z+\tau\mathbb Z)$. This is an analytic elliptic curve. 
There is a commutative diagram of analytic varieties 
$$
\xymatrix{
\mathbb C^{2n} \ar[r]\ar[d]& (E_\tau^\#)^n\ar[d]\\ 
\mathbb C^n\ar[r] & E_\tau^n}
$$
By \cite{CEE}, $E_\tau^n$ is equipped with a principal $\mathrm{exp}(\hat{\mathfrak t}_{1,n}^{\mathbb C})$-bundle 
$\mathcal P_{\mathrm{KZB}}$ with flat connection $\nabla_{\mathrm{KZB}}$; the lift of $(\mathcal P_{\mathrm{KZB}},\nabla_{\mathrm{KZB}})$ 
to $\mathbb C^n$ identifies with a principal bundle with flat connection over $\mathbb C^n$  
$$
(\text{trivial }\mathrm{exp}(\hat{\mathfrak t}_{1,n}^{\mathbb C})\text{-bundle},
d+A_{\mathrm{KZB}}).
$$ 
On the other hand, Lemma \ref{lemma:diafcottpebov} gives rise to a principal bundle with flat connection over $(E_\tau^\#)^n$, 
$$
(\text{trivial }\mathrm{exp}(\mathfrak G)\text{-bundle},
d+\omega).
$$
In \S\ref{sect:rwtukc}, we prove: 
\begin{thm}\label{tisaibtfpbwfcoC}
There is an isomorphism between the following principal bundles with flat connections over $\mathbb C^{2n}$: 
\begin{itemize}
\item the pull-back under $\mathbb C^{2n}\to\mathbb C^{n}\to(E_\tau)^n$ of $(\mathcal P_{\mathrm{KZB}},\nabla_{\mathrm{KZB}})$; 
\item the pull-back under $\mathbb C^{2n}\to(E_\tau^\#)^n$ of $((E_\tau^\#)^n\times\mathrm{exp}(\mathfrak G),
d+\omega)$. 
\end{itemize}
\end{thm}

\section{Category equivalences induced by desingularization (equiv. ($\mathrm d$)) }\label{sect:ceibd}

\subsection{A geometric result}

We work over $\mathbb C$. Let $X$ be a smooth variety. 

We define a unipotent vector bundle on $X$ to be an iterated extension of copies of ${\mathcal O}_X$, i.e., 
a vector bundle which admits a filtration with associated graded $\mathcal O_X^{\oplus N}$ for some $N\geq 0$. 

Similarly, a unipotent connection is a pair (bundle, connection) which admits a filtration with associated 
graded $(\mathcal O_X,d)^{\oplus N}$ for some $N\geq 0$.

Let $X$ be a smooth variety and $D$ a divisor in $X$.
Let $\pi: \widetilde{X}\to X$ be a resolution of singularities, such that $\widetilde{D}:=\pi^{-1}(D)$ is a normal crossing divisor and 
$\pi: \tilde X-\tilde D\to X-D$ is an isomorphism. 
Let $V\in \mathrm{VBFC}(\tilde X,\tilde D)_{unip}$, i.e., $V$ is a unipotent vector bundle on $\tilde X$ with a flat unipotent connection $\nabla$ 
outside $\tilde D$, which has first order poles and nilpotent residues at $\tilde D$.  

The following proposition was suggested to us by P. Deligne.

\begin{proposition}\label{prop:Deligne} 
The bundle $V$ descends to a unipotent vector bundle $\overline{V}$ on $X$ with respect to which 
the connection $\nabla$ on $\tilde X- \tilde D\cong X-D$ has first order poles at $D$ with nilpotent residues.   
\end{proposition} 

\begin{proof}  
Let $U\subset X$ be an affine open set. We claim that the restriction of the bundle $V$ to $\pi^{-1}(U)$ is trivial. This follows immediately 
from the following lemma. 

\begin{lemma}\label{van}
One has $H^i(\pi^{-1}(U),{\mathcal O})=0$ for $i>0$.  
\end{lemma}

\begin{proof} It is well known that $\pi_*{\mathcal O}={\mathcal O}$, i.e., there is no higher direct images; see e.g. the beginning and 
Theorem 1 in \cite{CR}. Thus, by adjunction $H^i(\pi^{-1}(U),{\mathcal O})=H^i(U,{\mathcal O})=0$ for $i>0$ since $U$ is affine.  
\end{proof} 
   
Indeed, by Lemma \ref{van}, we have ${\rm Ext}^1_{\pi^{-1}(U)}(\mathcal{O},\mathcal{O})=0$, so 
$V|_{\pi^{-1}(U)}$ is trivial, as it is unipotent. 

Thus, we see that $V=\pi^*\overline{V}$, where $\overline{V}:=\pi_*V$. In other words, the fibers of $V$ at all points 
of $\pi^{-1}(p)$ for any $p\in X$ are canonically isomorphic to each other, and thus give rise to a well defined vector space, which is the fiber of $\overline{V}$ at $p$, and it varies algebraically in $p$. It is clear that the connection $\nabla$ on $X-D$ has simple poles and nilpotent residues at $D$ with respect to $V$. This proves the proposition. 
\end{proof} 

\subsection{Proof of Lemma \ref{equivd}}

Let $V\in \mathrm{VBFC}(\tilde X,\tilde D)_{unip}$. By Proposition \ref{prop:Deligne},
we can canonically attach to it a bundle $\overline{V}\in \mathrm{VBFC}(X,D)_{unip}$.
This gives rise to a functor 
$$\pi_*: V\in \mathrm{VBFC}(\tilde X,\tilde D)_{unip}\mapsto
\overline{V}\in \mathrm{VBFC}(X,D)_{unip}.$$
 It is clear that this functor is fully faithful, so it remains to show that it is essentially
surjective. To this end, it suffices to note that for any $W \in \mathrm{VBFC}(X,D)_{unip}$, we have $W\cong \pi_*\pi^*W$, where $\pi^*W\in \mathrm{VBFC}(\tilde X,\tilde D)_{unip}$ is the ordinary pullback of $W$. Indeed, this is definitely so outside of a set of codimension $2$ on $X$ (since there the map $\pi$ is an isomorphism). But any isomorphism of vector bundles on $X$ defined outside of a set of codimension $2$ extends to the whole $X$. Thus, the functor $\pi_*$ is an equivalence whose inverse is $\pi^*$. 

\remark{
This argument automatically yields that the flat connection on $\pi^*W$ has simple poles and nilpotent residues on each component of
the exceptional divisor $D'$ of $\pi$. Let us prove this fact independently. Let $C$ be a smooth algebraic curve on $\tilde X$
passing transversally through a generic point $P$ of some component $D_j'$ of $D'$ and having no other intersection
points with $\tilde D$. Then $\pi^*W|_C$ is a vector bundle on $C$ with a unipotent connection outside of $P$, and our job is show that this connection
has a simple pole with respect to $\pi^*W|_C$, and moreover the corresponding residue is nilpotent.
To this end, consider the point $\pi(P)\in D$. Let $D_{i_1},...,D_{i_r}$ be the components of $D$ containing $\pi(P)$.
Let $D_{i_m}$ be described near $P$ by the equation $z_m=0$, where $z_m$ is a rational function on $X$ regular at $P$ with $dz_m(P)\ne 0$. Then, trivializing $W$ near $P$, we get that the connection form takes the form
$$
\omega=\sum_m \frac{A_{i_m}}{z_m}dz_m+...
$$
where ... is the regular part and $A_{i_m}$ is the residue of the connection at $D_{i_m}$. Let $d_m$ be the degree of intersection of $D_{i_m}$ with $C$, i.e.
the order of vanishing of $z_m|_C$ at $P$. Then it is easy to see that the connection on $C$ has a simple pole at $P$ with residue
$A=\sum_m d_mA_{i_m}$. (In fact, if $C$ is generic, then $d_m=1$.) Since the connection on $W$ is unipotent, $A_{i_m}$ are strictly 
upper triangular in the same basis. Hence, so is $A$. Thus $A$ is nilpotent, as desired. }

\section{Elliptic material}\label{sect:ellmat}

This \S\  presents the material related to elliptic curves alluded to in \S\ref{sect:examples}: (a) elliptic curves in char. 0 (\S\ref{subsect:2}); 
(b) the construction of the functor $E\mapsto E^\#$, where for each elliptic curve $E$, $E^\#$ is a surface equipped with an affine fibration 
$E^\#\to E$ (\S\ref{sec:3}). In \S\ref{sect:afioe}, we prove an identity between functions on $E^\#$ derived from the Fay identity, which will be 
used in the sequel of the paper. 


\subsection{Elliptic curves in characteristic zero}
\label{subsect:2}

\subsubsection{Universal elliptic and theta functions}
\label{subsect:uetf}

Let $\underline g_2,\underline g_3$ be formal commutative variables; they generate the polynomial ring 
$\mathbb Q[\underline g_2,\underline g_3]$. 

\begin{lemma}\label{lemma:wp:univ}
There exists a unique family $(a_n(\underline g_2,\underline g_3))_{n\geq 0}$ of elements of $\mathbb Q[\underline g_2,\underline g_3]$, 
such that the element 
\begin{equation}\label{def:wp:univ}
\wp_{univ}:={1\over p^2}+\sum_{n\geq 0}a_n(\underline g_2,\underline g_3)p^n\in
\mathbb Q[\underline g_2,\underline g_3]((p))
\end{equation}
satisfies 
\begin{equation}\label{id:F}
(\wp_{univ}')^2=4\wp_{univ}^3-\underline g_2\wp_{univ}-\underline g_3
\end{equation}
(identity in $\mathbb Q[\underline g_2,\underline g_3]((p))$), where $\wp_{univ}'$ is the derivative of $\wp_{univ}$ 
with respect to $p$. 
\end{lemma}

This statement is well-known. We give a proof for completeness. 

{\em Proof.} Explicit computation imposes $a_n=0$ for $n$ odd, and 
$$
a_0=0, \quad a_2={1\over 20}\underline g_2,\quad a_4={1\over 28}\underline g_4.
$$
The equation 
\begin{equation}\label{cons:id:F}
\wp_{univ}''=6\wp_{univ}^2-{1\over 2}\underline g_2,
\end{equation} 
consequence of the identity (\ref{id:F}), then implies for 
$n\geq 6$
$$
(n^2-n-12)a_n=6\sum_{p,q\geq 2|p+q=n-2}a_pa_q, 
$$
which determine $(a_n)_{n\geq 0}$ uniquely. Substituting these values in (\ref{def:wp:univ}), one then obtains a solution of (\ref{cons:id:F}), 
which by multiplying by $\wp'_{univ}$ and integrating is such that $(\wp_{univ}')^2-4\wp_{univ}^3+\underline g_2\wp_{univ}$ is 
constant with respect to $p$; as the constant term in this expression is $-\underline g_3$, this solution of (\ref{cons:id:F}) satisfies
(\ref{id:F}). \hfill \qed\medskip 

\begin{remark}\label{rem1}If the variables $\underline g_2,\underline g_3$ are given weights $4,6$, then $a_n$ is a polynomial of 
weight $n+2$. 
\end{remark}

Define 
$$
\zeta_{univ}:={1\over p}-\sum_{n\geq 0}a_n(\underline g_2,\underline g_3){p^{n+1}\over n+1}\in\mathbb Q[\underline g_2,\underline g_3]((p)), 
$$
and 
$$
\tilde\theta_{univ}:=p\cdot \mathrm{exp}\big(-\sum_{n\geq 0}a_n(\underline g_2,\underline g_3){p^{n+2}\over (n+1)(n+2)}\big) 
\in\mathbb Q[\underline g_2,\underline g_3][[p]], 
$$
then 
$$
\wp_{univ}=-\partial_p^2\mathrm{log}\tilde\theta_{univ}. 
$$

\subsubsection{Specializations}\label{specializ}

Let $\mathbf k$ be a field of characteristic 0, let $(E,0,\omega)$ be a triple of: an elliptic curve $E$ over $\mathbf k$, an element 
$0\in E(\mathbf k)$, and a nonzero regular differential $\omega$ on $E$. Assume that $E$ has a model $y^2=4x^3-g_2x-g_3$, where 
$g_2,g_3\in\mathbf k$, such that $0$ corresponds to the point at infinity and $\omega=dx/y$. Then the ring of regular functions on 
$E-\{0\}$ is 
\begin{equation}\label{def:A}
A:=\Gamma(E,\mathcal O_E(*0))=\mathbf k[x,y]/(y^2-4x^3+g_2x+g_3). 
\end{equation}
The elements of $\mathbf k((p))$ obtained from $\wp_{univ}(p),\wp'_{univ}(p)$ by the specialization  
$\mathbb Q[\underline g_2,\underline g_3]\to \mathbf k$, induced by $\underline g_2\mapsto g_2$, $\underline g_3\mapsto g_3$
depend only on the pair $(E,\omega)$, and will henceforth be denoted $\wp_{E,\omega}(p)$, $\wp'_{E,\omega}(p)$. 

If $(E,0,\omega)$ corresponds to $(g_2,g_3)$ and $\alpha\in\mathbf k^\times$, then $(E,0,\alpha\omega)$ corresponds to 
$(\alpha^{-4}g_2,\alpha^{-6}g_3)$. Remark \ref{rem1} in \S\ref{subsect:uetf} then implies the identity $\tilde\theta_{E,\alpha\omega}(p)
=\alpha^2\tilde\theta_{E,\omega}(\alpha^{-2}p)$. 

There is a unique ring morphism 
$$
A\to\mathbf k((p)), 
$$
$$
x\mapsto\wp_{E,\omega}(p), \quad y\mapsto\wp'_{E,\omega}(p). 
$$
This morphism is injective. 

Moreover, there is a derivation $\partial$ of $A$, uniquely determined by 
$$
\partial : x\mapsto y, \quad y\mapsto 6y^2-{1\over 2}g_2. 
$$
This derivation is compatible with the derivation $d/dp$ of 
$\mathbf k((p))$, so that the following diagram commutes
$$
\xymatrix{A\ar[r]\ar_\partial[d]&\mathbf k((p))\ar^{d/dp}[d]\\A\ar[r]&\mathbf k((p))}
$$
The derivation $\partial$ (resp., $d/dp$) corresponds to the regular differential $dx/y$ of $E$ (resp., $dp$ of the formal disc); the 
differentials $dx/y$ and $dp$ correspond to each other under the parametrization of the formal disc around $0$ induced by 
$A\to\mathbf k((p))$. 

\subsubsection{Relation with uniformization}

For $\tau$ a complex number with positive imaginary part, let $\theta(-|\tau):\mathbb C\to\mathbb C$ be the holomorphic
function defined by 
\begin{equation}\label{def:theta}
\theta(z|\tau):={1\over{2\pi\mathrm{i}}}(\mathbf e({z\over 2})-\mathbf e(-{z\over 2}))
\prod_{j\geq 1}(1-\mathbf e(z+j\tau))\prod_{j\geq 1}(1-\mathbf e(-z+j\tau)), 
\end{equation}
where $\mathrm i:=\sqrt{-1}$ and $\mathbf e(z):=\mathrm{exp}(2\pi\mathrm{i}z)$. It is such that 
$$
\theta(z+1|\tau)=-\theta(z|\tau), \quad \theta(z+\tau|\tau)=-e^{-i\pi\tau}e^{-2\pi i z}\theta(z|\tau), \quad \theta'(0|\tau)=1,  
$$
where $'$ means the partial derivative with respect to $z$. 

Assume that $\mathbf k=\mathbb C$ and that $(E,0,\omega)$ is such that $E_{an}\simeq\mathbb C/(\mathbb Z+\tau\mathbb Z)$ 
and that $\omega$ corresponds to $dp$. Then 
\begin{equation}\label{rel:tilde:theta}
\tilde\theta_{\mathbb C/(\mathbb Z+\tau\mathbb Z),dp}(p)=\theta(p|\tau)e^{{1\over 2}G_2(\tau)p^2}
\end{equation}
for $p$ formal near 0 and $G_2(\tau)$ is the quasimodular Eisenstein series given by 
$${1\over 2}G_2(\tau)={\pi^2\over 6}-(2\pi)^2\sum_{n\geq 1}\sigma_1(n)q^n,$$ where
$q=\mathbf e(\tau)$. 

\subsubsection{Analytic Fay identity}

Let $\mathfrak H$ be the set of complex numbers with positive imaginary part. Let $\tau\in\mathfrak H$. 
Fay's identity (see \cite{Fay}) is the identity 
\begin{align}\label{fay:prim}
& \nonumber \theta(p+z|\tau)\theta(p+p'+z'|\tau)\theta(p'|\tau)\theta(z+z'|\tau)
-\theta(p'+z'|\tau)\theta(p+z+z'|\tau)\theta(z|\tau)\theta(p+p'|\tau)
\\ & +\theta(p+p'+z+z'|\tau)\theta(p'-z|\tau)\theta(p|\tau)\theta(z'|\tau)=0
\end{align}
in $\mathrm{Hol}(\mathbb C^4)$, where $(z,z',p,p')$ is the current variable in $\mathbb C^4$. 

The function
$$
\mathbb C^2\times\mathfrak H\to\mathbb C, \quad (p,z,\tau)\mapsto F(p,z|\tau):={\theta(p+z|\tau)\over
\theta(p|\tau)\theta(z|\tau)}
$$
is meromorphic. It expands for $z\to 0$ as a series $\sum_{k\geq-1}F_k(p|\tau)z^k$, where $F_k(p|\tau)\in
\frac{1}{p}\mathrm{Hol}(\mathfrak H)[[p]]$. It follows that $F(p,z|\tau)$ may be viewed as an element 
of ${1\over {pz}}\mathrm{Hol}(\mathfrak H)[[p,z]]$. 

View now $z,z',p,p'$ in (\ref{rel:tilde:theta}) as formal variables. Dividing this equation by the product
$\theta(p|\tau)\theta(p'|\tau)\theta(z|\tau)\theta(z'|\tau)\theta(p+p'|\tau)\theta(z+z'|\tau)$, we obtain the identity
\begin{equation}\label{fay:formal}
F(p,z|\tau)F(p+p',z'|\tau)-F(p',z'|\tau)F(p,z+z'|\tau)+F(p+p',z+z'|\tau)F(p',-z|\tau)=0
\end{equation}
in ${1\over{pp'zz'(p+p')(z+z')}}\mathrm{Hol}(\mathfrak H)[[p,p',z,z']]$.  

\subsubsection{Formal series Fay identity}

Let $(g_2,g_3)\in\mathbb C^2$. Let $(E,0,\omega)$ be the triple of the elliptic curve $y^3=4x^3-g_2x-g_3$, the point at infinity, 
and the differential $\omega=dx/y$. Choose a uniformization $i:E_{an}\stackrel{\sim}{\to}\mathbb C/(\mathbb Z+\tau\mathbb Z)$, 
taking $0$ to the origin. Then there exists a unique $\alpha\in\mathbb C^\times$, such that $\omega=\alpha\cdot i^*(dp)$. 
Computation then yields 
$$
\tilde\theta_{E,\omega}(p)=\alpha^2\theta(\alpha^{-2}p|\tau)e^{{1\over 2}G_2(\tau)(\alpha^{-2}p)^2}
$$
(identity in $\mathbb C[[p]]$). 

View now $p,z$ as formal variables and set 
\begin{equation}\label{def:tilde:F:E:omega}
\tilde F_{E,\omega}(p,z):={\tilde\theta_{E,\omega}(p+z)\over\tilde\theta_{E,\omega}(p)\tilde\theta_{E,\omega}(z)} 
\end{equation}
(element of ${1\over{pz}}\mathbb C[[p,z]]$). 
Then 
$$
\tilde F_{E,\omega}(p,z)=\alpha^{-2}F(\alpha^{-2}p,\alpha^{-2}z|\tau)e^{G_2(\tau)\alpha^{-4}pz}
$$
(identity in ${1\over{pz}}\mathbb C[[p,z]]$). Combining this identity with (\ref{fay:formal}), one gets 
\begin{equation}\label{id:fay:premixte}
\tilde F_{E,\omega}(p,z)\tilde F_{E,\omega}(p+p',z')-\tilde F_{E,\omega}(p',z')\tilde F_{E,\omega}(p,z+z')
+\tilde F_{E,\omega}(p+p',z+z')\tilde F_{E,\omega}(p',-z)=0
\end{equation}
(identity in ${1\over{pp'zz'(p+p')(z+z')}}\mathbb C[[p,p',z,z']]$). 

The coefficients of this identity are polynomials in $(g_2,g_3)$ and therefore arise from specialization of the analogous identity, 
where $\tilde\theta_{E,\omega}$ is replaced by $\tilde\theta_{univ}$, under specialization $\mathbb Q[\underline g_2,\underline g_3]\to
\mathbb C$. It follows that the identity analogous to (\ref{id:fay:premixte}) holds with $\tilde\theta_{E,\omega}$ replaced by $\tilde\theta_{univ}$, 
and then upon further specialization that identity (\ref{id:fay:premixte}) also holds when $(E,0,\omega)$ is defined over a field $\mathbf k$ of 
characteristic $0$.

\subsection{The functor $E\mapsto E^\#$}\label{sec:3}

In this \S\ and in the next one, we choose an elliptic curve $E$ as in \S\ref{specializ}. It is therefore defined over 
a field $\mathbf k$ of characteristic 0, has fixed origin $0$, and a nonzero regular differential $\omega$. 

\subsubsection{A section $\sigma\in\Gamma_{rat}(E\times E,\Omega^1_E\boxtimes\mathcal O_E)$}

\begin{lemma} 
There exists a unique rational section $\sigma$ of the bundle $\Omega^1_E\boxtimes\mathcal O_E$ 
over $E\times E$, regular except for: 
\begin{itemize}
\item a simple pole at $E_{diag}$ with residue $1$, 
\item a simple pole at $\{0\}\times E$ with residue $-1$, 
\item a simple pole at $E\times\{0\}$,
\end{itemize}
and such that the ratio $\sigma/(\omega\otimes 1)$ (a rational function on $E\times E$) is antisymmetric w.r.t. the exchange of variables. 
\end{lemma}

{\em Proof.} If $\sigma_1,\sigma_2$ are two such rational functions, then their difference $\sigma_3$ is regular on $E^2$ and 
is therefore of the form $c\cdot(\omega\otimes 1)$, where $c\in\mathbf k$. The antisymmetry condition then implies that it is zero. 

If $E$ has a model $y^2=4x^3-g_2x-g_3$ and $0$ is the point at infinity in this model, then $\sigma$ is given by 
$$
\sigma(P,Q):={{y_P+y_Q}\over{x_P-x_Q}}\cdot{{dx_P}\over{2y_P}},
$$
where $P=(x_P,y_P)$ and $Q=(x_Q,y_Q)$. \hfill\qed\medskip 

\remark{Assume that $\mathbf k=\mathbb C$ and that a uniformization $E_{an}\simeq\mathbb C/(\mathbb Z+\tau\mathbb Z)$ 
is fixed. Then the image of $\sigma$ under $(E_{an})^2\simeq(\mathbb C/(\mathbb Z+\tau\mathbb Z))^2$ is the differential 
$$
\left({\theta'\over\theta}(p_1-p_2|\tau)-{\theta'\over\theta}(p_1|\tau)+{\theta'\over\theta}(p_2|\tau)\right)dp_1, 
$$ 
where $(p_1,p_2)$ are the coordinates on $\mathbb C^2$. }

\subsubsection{The functor $E\mapsto E^\#$}\label{sect:functor:EEsharp}

Define $E^\#$ to be the moduli space 
of line bundles of degree zero over $E$, equipped with a flat connection. The tensor product of bundles with connections makes $E^\#$ 
into a commutative algebraic group over $\mathbf k$, fitting in an exact sequence 
\begin{equation}\label{ex:seq:E:Esharp}
0\to \Gamma(E,\Omega^1_E)\stackrel{\mathrm{can}}{\to} E^\#\stackrel{\pi}{\to} E\to 0,
\end{equation}
where $\mathrm{can}:\Gamma(E,\Omega^1_E)\to E^\#$ is the canonical group morphism. 
According to \cite{Me,MaMe} (based on \cite{Ro}; for a recent account see \cite{BoK}), $E^\#$ 
is the universal extension of the algebraic group $E$ by vector spaces.

Given a pair $(\mathcal L,\sigma)$ of a line bundle of degree zero on $E$ equipped with a nonzero rational section, one may use 
$\sigma$ as a rational trivialization in order to express connections on $\mathcal L$; this leads to a bijection 
\begin{align*}
\underline{\mathrm{iso}}:\{\text{connections }\nabla\text{ on }\mathcal L\}\leftrightarrow & \{\text{rational differentials }\psi\text{ on }E\text{ regular except for simple 
poles} \\ &\text{ at the zeroes/poles of }\sigma,\text{ and with }\sum_{P\in E}\mathrm{res}_P(\psi)P=(\sigma)\}, 
\end{align*}
$$\nabla\mapsto\nabla(\sigma)\sigma^{-1}.
$$ 
A morphism 
$$
s:E-\{0\}\to E^\#,
$$ which is a rational section of the morphism $\pi:E^\#\to E$ may be constructed as follows. To a 
point $P$ in $E-\{0\}$, it associates the pair of the bundle $\mathcal O(P-0)$ and of the connection $\nabla_P$ corresponding via 
$\underline{\mathrm{iso}}$ to the rational differential $\sigma(-,P)$ (this makes sense as the divisor of this differential is 
$P-0$). The morphism $s$ gives rise to an isomorphism of schemes
\begin{equation}\label{iso:E:sharp}
\tilde s:(E-\{0\})\times \Gamma(E,\Omega^1_E)\to E^\#-\pi^{-1}(0),
\end{equation} 
given by $(P,\tilde c)\mapsto(\mathcal O(P-0),\nabla_P+\tilde c)$. The behavior of $\tilde s$ with respect to the group structures
on both sides is described by
$$
\tilde s(P+P',\tilde c+\tilde c')=\tilde s(P,\tilde c)+\tilde s(P',\tilde c')+f(P,-P')\cdot \mathrm{can}(\omega)\quad (\text{identity in }E^\#), 
$$
for any $P,P'\in E$ such that $P,P'$ and $P+P'$ are $\neq 0$; here $+$ denotes the addition both in $E$, in $\Gamma(E,\Omega^1_E)$ 
and in $E^\#$, $\mathrm{can}$ is as in (\ref{ex:seq:E:Esharp}), and $f(P,P')\in\mathbf k$ is defined by $\sigma(P,P')=f(P,-P')\omega(P)$.

\subsubsection{Formal neighborhood of $\pi^{-1}(0)$ in $E^\#$}

Let $E$ be as in \S\ref{specializ}. The element $\omega=dx/y=dp\in \Gamma(E,\Omega^1_E)$ is a basis of this vector space and sets up an 
isomorphism $\Gamma(E,\Omega^1_E)\simeq \mathbb A^1$. Combining the isomorphism (\ref{iso:E:sharp}) with this isomorphism,  
we obtain an isomorphism 
\begin{equation}\label{formula:calA}
\Gamma(E^\#,\mathcal O_{E^\#}(*\pi^{-1}(0)))\simeq \mathbf k[(E-\{0\})\times\mathbb A^1]=\mathbf k[x,y,\tilde c]/
(y^2=4x^3-g_2x-g_3). 
\end{equation}
(see \S\ref{specializ}). We set
\begin{equation}\label{def:cal:A}
\mathcal A:=\Gamma(E^\#,\mathcal O_{E^\#}(*\pi^{-1}(0))). 
\end{equation}
There is an isomorphism $\pi^{-1}(0)\simeq \Gamma(E,\Omega^1_E)$, and an isomorphism of the formal
neighborhood of $\pi^{-1}(0)$ in $E^\#$ with $\Gamma(E,\Omega^1_E)\times\mathrm{Spf}(\mathbf k[[p]])\simeq\mathbb A^1
\times\mathrm{Spf}(\mathbf k[[p]])=\mathrm{Spf}(\mathbf k[t][[p]])$, corresponding to the injective 
algebra morphism 
\begin{equation}\label{can:morph}
\mathrm{can}:\mathcal A\to\mathbf k[t]((p))
\end{equation}
given by 
\begin{equation}\label{16bis}
x\mapsto\wp_{E,\omega}(p), \quad y\mapsto \wp'_{E,\omega}(p), \quad \tilde c\mapsto t-{{d\mathrm{log}\tilde\theta_{E,\omega}(p)}\over{dp}}. 
\end{equation}

\subsubsection{Uniformization of $E^\#$}\label{section:unif}

Assume that $\mathbf k=\mathbb C$ and that a uniformization 
\begin{equation}\label{unif:E:an}
E_{an}\simeq\mathbb C/(\mathbb Z+\tau\mathbb Z)  
\end{equation}
is fixed, such that $dx/y$ corresponds to $dp$, $p$ being the canonical coordinate on $\mathbb C$. 

There is a morphism $\mathbb Z^2\to\mathbb C^2$, given by $(n,m)\mapsto (n+m\tau,-2\pi\mathrm{i}m)$. Then there is an analytic isomorphism 
$$
E^\#_{an}\simeq\mathrm{Coker}(\mathbb Z^2\to\mathbb C^2), 
$$
whose inverse takes the class of $(p,c)\in\mathbb C^2$ to the bundle $\mathcal O([p]-[0])$ equipped with the 
connection 
\begin{equation}\label{cnx} 
d+({\theta'\over\theta}(z-p|\tau)-{\theta'\over\theta}(z|\tau)+c)dz,  
\end{equation}
$z$ being the standard coordinate on $\mathbb C$. Here $[p]$ is the class of $p$ in $E_{an}$, $[p]-[0]$ is a degree zero divisor, 
$\mathcal O([p]-[0])$ is the associated line bundle, and $\theta(-|\tau)$ is defined by (\ref{def:theta}). 


Taking into account the role of the section $\sigma$ in the definition of the isomorphism $\tilde s$, we derive 
that the composite map 
$$
E_{an}^\#-\pi^{-1}(0)\stackrel{\sim}\to E^\#(\mathbb C)-\pi^{-1}(0)\stackrel{\tilde s}{\simeq} 
(E(\mathbb C)-\{0\})\times\mathbb C\to\mathbb C,   
$$
where the last map is the projection on the last factor, is the map $\tilde c:E^\#_{an}-\pi^{-1}(0)\to\mathbb C$ 
taking the class of $(p,c)$ to 
$$
\tilde c(p,c):=c-(\theta'/\theta)(p|\tau).
$$
The map $\tilde c$ is a rational function on $E^\#$, regular except at $\pi^{-1}(0)$ where it has a simple pole. Taking into 
account (\ref{16bis}) and (\ref{rel:tilde:theta}), the local coordinate systems $(p,c)$ and $(p,t)$ at the neightborhood of $\pi^{-1}(0)$
are related by 
$$
t=c+G_2(\tau)p. 
$$ 
The relation between the algebraic and analytic coordinates on $E^\#$ is then: 
$$
[(p,c)]\leftrightarrow (x,y,\tilde c)=(\wp(p|\tau),\wp'(p|\tau), c-{\theta'\over\theta}(p|\tau)). 
$$

\remark{There is, up to isomorphism, a unique 2-dimensional bundle $\mathcal E$ over $E$, which is a nontrivial 
extension of $\mathcal O$ by itself. One then has an exact sequence $0\to\mathcal O\to\mathcal E\stackrel{\varphi}{\to}\mathcal O\to 0$. 
Let $\mathrm{Tot}(\mathcal F)$ be the total space of a bundle $\mathcal F$ over $E$. Then $\varphi$ induces a morphism $\tilde\varphi:
\mathrm{Tot}(\mathcal E)\to\mathrm{Tot}(\mathcal O)\simeq E\times\mathbb A^1$, and $E^\#$ identifies with the preimage 
$\tilde\varphi^{-1}(E\times\{1\})$. When $\mathbf k=\mathbb C$, this can be checked using a uniformization 
(\ref{unif:E:an}) of $E_{an}$. The bundle $\mathcal E$ then identifies with the bundle $(\mathbb C\times\mathbb C^2)/\mathbb Z^2$, 
where the action of $\mathbb Z^2$ on $\mathbb C\times\mathbb C^2$ is $(n,m)\cdot(z,\vec v):=(z+n+m\tau,\begin{pmatrix} 1 
& -2\pi\mathrm{i}m \\ 0 & 1\end{pmatrix}\cdot\vec v)$. 
}

\subsection{Algebraic Fay identity on $(E^\#)^n$}\label{sect:afioe}

\subsubsection{Elements in $\mathcal A=\Gamma(E^\#,\mathcal O_{E^\#}(*\pi^{-1}(0)))$}\label{section:elements}

Recall from \S\ref{specializ} that $\partial$ is a derivation of $A=\mathbf k[E-\{0\}]$ (see (\ref{def:A})). It may be uniquely extended 
to a derivation of $\mathcal A\simeq A[\tilde c]$, also denoted 
$\partial$, by $\partial(\tilde c)=x$. This extension is compatible with the morphism $\mathrm{can}:\mathcal A\to\mathbf k[t]((p))$ 
and with the derivation $\partial/\partial p$ of the latter ring. 

\begin{lemma}
Let $z$ be a formal variable. One has
$$
{1\over\tilde\theta_{E,\omega}(z)}\mathrm{exp}\big(-\tilde c z-\sum_{k\geq 2}
{1\over k!}\partial^{k-2}(x)z^k\big)-{1\over z}\in\mathcal A[[z]]. 
$$
\end{lemma}

\proof The element  $\tilde\theta_{E,\omega}(z)$ belongs to $z+z^2\mathbf k[[z]]$, therefore 
$$
{1\over\tilde\theta_{E,\omega}(z)}-{1\over z}\in\mathbf k[[z]], 
$$
and moreover $-\tilde c z-\sum_{k\geq 2}{1\over k!}\partial^{k-2}(x)z^k$ belongs to $z\mathcal A[[z]]$, 
therefore
$$
{1\over\tilde\theta_{E,\omega}(z)}\Big(\mathrm{exp}\big(-\tilde c z-\sum_{k\geq 2}
{1\over k!}\partial^{k-2}(x)z^k\big)-1\Big)\in\mathcal A[[z]]. 
$$
All this implies the result. \hfill\qed\medskip

\begin{definition}
The elements $(f_\alpha)_{\alpha\geq-1}$ of $\mathcal A$ are defined by 
\begin{equation}\label{def:f:alpha}
f_\alpha:=\Big[{1\over\tilde\theta_{E,\omega}(z)}\mathrm{exp}\big(-\tilde c z-\sum_{k\geq 2}
{1\over k!}\partial^{k-2}(x)z^k\big)\Big|z^\alpha\Big], 
\end{equation}
where $[-|z^\alpha]$ means the coefficient of $z^\alpha$ in the formal series 
expansion (so $f_{-1}=1$).
\end{definition}

\begin{lemma} The following identity 
\begin{equation}\label{id:interm}
\tilde\theta_{E,\omega}(z)\Big({1\over z}+\sum_{\alpha\geq 0}\mathrm{can}(f_\alpha) z^\alpha\Big)
=e^{-tz}{\tilde\theta_{E,\omega}(p+z)\over\tilde\theta_{E,\omega}(p)}. 
\end{equation}
holds in $\mathbf k[t]((p))[[z]]$. 
\end{lemma}

\proof By Def. \ref{def:f:alpha}, we have 
\begin{equation}\label{id:f:alpha}
\tilde\theta_{E,\omega}(z)\Big({1\over z}+\sum_{\alpha\geq 0}f_\alpha z^\alpha\Big)=
\mathrm{exp}\big(-\tilde c z-\sum_{k\geq 2}
{1\over k!}\partial^{k-2}(x)z^k\big)
\end{equation}
(identity in $\mathcal A((z))$). 

The image of this identity under the morphism $\mathrm{can}:\mathcal A\to\mathbf k[t]((p))$
(see (\ref{can:morph})) is the first equality in the following computation  
\begin{align*}
&\tilde\theta_{E,\omega}(z)\Big({1\over z}+\sum_{\alpha\geq 0}\mathrm{can}(f_\alpha) z^\alpha\Big)
=\mathrm{exp}\left(-\left(t-{\tilde\theta_{E,\omega}'(p)\over\tilde\theta_{E,\omega}(p)}\right)z-\sum_{k\geq 2}
\partial_p^{k-2}(-\partial_p^2\mathrm{log}\tilde\theta_{E,\omega}(p)){z^k\over k!}\right) \\
& = e^{-tz}\mathrm{exp}\big((e^{z\partial_p}-1)\mathrm{log}\tilde\theta_{E,\omega}(p)\big)
=e^{-tz}{\tilde\theta_{E,\omega}(p+z)\over\tilde\theta_{E,\omega}(p)}
\end{align*}
(equality in $\mathbf k[t]((p))[[z]]$), which implies the result. \hfill\qed\medskip

\subsubsection{Operations on $A$ and $\mathcal A$ arising from the group laws of $E$ and $E^\#$}\label{sect:operations}

Let $(E,0)$ be given by the model $y^2=4x^3-g_2x-g_3$ and the point at infinity, 
so that $A=\mathbf k[E-\{0\}]$ is as in (\ref{def:A}). 

\begin{lemma}
The addition law on $(E,0)$ gives rise to a coproduct morphism 
$$
\Delta_A:A\to (A\otimes A)[{y^{(1)}-y^{(2)}\over x^{(1)}-x^{(2)}}],\quad x\mapsto -x^{(1)}-x^{(2)}+{1\over 4}
\Big({y^{(1)}-y^{(2)}\over x^{(1)}-x^{(2)}}\Big)^2,
$$
$$
y\mapsto 
-{1\over 2}(y^{(1)}+y^{(2)})+{3\over 2}(x^{(1)}+x^{(2)}){y^{(1)}-y^{(2)}\over x^{(1)}-x^{(2)}}-{1\over 4}\Big({y^{(1)}-y^{(2)}\over x^{(1)}
-x^{(2)}}\Big)^3,  
$$
where $a^{(1)}:=a\otimes 1$, $a^{(2)}:=1\otimes a$ for $a=x,y$.  
\end{lemma}

\proof The addition $E\times E\to E$ is given by 
$((x_P,y_P),(x_Q,y_Q))\mapsto(x_R,y_R)$, where  
$$
x_R:=-x_P-x_Q+{1\over 4}\Big({y_P-y_Q\over x_P-x_Q}\Big)^2, 
$$
$$
y_R:=-{1\over 2}(y_P+y_Q)+{3\over 2}(x_P+x_Q){y_P-y_Q\over x_P-x_Q}-{1\over 4}\Big({y_P-y_Q\over x_P-x_Q}\Big)^3. 
$$
\hfill\qed\medskip 

Recall that $\mathcal A:=\mathbf k[E^\#-\pi^{-1}(0)]$ is given by (\ref{formula:calA}). 

\begin{lemma}
The addition law on $E^\#$ gives rise to the coproduct morphism 
$$
\Delta_{\mathcal A}:\mathcal A\to (\mathcal A\otimes\mathcal A)[{y^{(1)}-y^{(2)}\over x^{(1)}-x^{(2)}}], 
$$
extending $\Delta_A$ by 
$$
\tilde c\mapsto \tilde c^{(1)}+\tilde c^{(2)}+{y^{(1)}-y^{(2)}\over x^{(1)}-x^{(2)}}. 
$$
\end{lemma}

\proof This follows from the description of this addition law in \S\ref{sect:functor:EEsharp}.
\hfill\qed\medskip 

We now describe the compatibility of these morphisms with $A\to\mathbf k((p))$, $\mathcal A\to\mathbf k[t]((p))$.

\begin{lemma}
There are well-defined morphisms 
$$
(A\otimes A)[{y^{(1)}-y^{(2)}\over x^{(1)}-x^{(2)}}]\to\mathbf k[[p,p']][1/p,1/p',1/(p+p')], 
$$
$$
x^{(1)}\mapsto\wp_{E,\omega}(p), \quad 
x^{(2)}\mapsto\wp_{E,\omega}(p'), \quad 
y^{(1)}\mapsto\wp_{E,\omega}'(p), \quad 
y^{(2)}\mapsto\wp_{E,\omega}'(p'),
$$
$$ 
{y^{(1)}-y^{(2)}\over x^{(1)}-x^{(2)}}\mapsto {1\over pp'(p+p')}\cdot(\text{element of }\mathbf k[[p,p']])
$$
and 
$$
(\mathcal A\otimes \mathcal A)[{y^{(1)}-y^{(2)}\over x^{(1)}-x^{(2)}}]\to
\mathbf k[t,t'][[p,p']][1/p,1/p',1/(p+p')],
$$
extending the previous morphism by 
$$
\tilde c^{(1)}\mapsto t-{d\mathrm{log}\tilde\theta_{E,\omega}(p)\over dp}, \quad 
1\otimes\tilde c\mapsto t'-{d\mathrm{log}\tilde\theta_{E,\omega}(p')\over dp'}. 
$$
\end{lemma}

\proof The expansion $A\to\mathbf k((p))$ takes $x,y$ to series $1/p^2+\sum_{n\geq 0,\text{even}}a_np^n$ and 
$-2/p^3++\sum_{n\geq 0,\text{odd}}b_np^n$. Then the ratio 
$$
{y^{(1)}-y^{(2)}\over x^{(1)}-x^{(2)}}
$$
expands as 
$$
{(-2/p^3+2/p^{\prime3})+\sum_{n\geq 0,\text{odd}}b_n(p^n-p^{\prime n})\over
(1/p^2-1/p^{\prime2})+\sum_{n\geq 0,\text{even}}a_n(p^n-p^{\prime n})}
$$
which is 
$$
{1\over pp'}{-2(p^{\prime3}-p^3)+\sum_{n\geq 0,\text{odd}}b_n(p^n-p^{\prime n})(pp')^3\over
(p^{\prime2}-p^2)+\sum_{n\geq 0,\text{even}}a_n(p^n-p^{\prime n})(pp')^2}. 
$$
In the second fraction, the numerator factorizes as $(p-p')\cdot(\text{element of }\mathbf k[[p,p']])$, while the denominator 
factorizes as $(p^2-p^{\prime2})\cdot (\text{element of }\mathbf k[[p,p']]\text{ with constant term }-1)$. It follows that the overall 
fraction expands as ${1\over pp'(p+p')}\cdot(\text{element of }\mathbf k[[p,p']])$. All this implies the result. 
\hfill\qed\medskip 

One checks: 
\begin{lemma}
\begin{itemize}
\item[1)] The following diagrams are commutative 
\begin{equation}\label{diag:1}
\xymatrix{A\ar^{\!\!\!\Delta_A}[r]\ar[d] & (A\otimes A)[{y^{(1)}-y^{(2)}\over x^{(1)}-x^{(2)}}]\ar[d]\\ 
\mathbf k((p))\ar[r]& \mathbf k[[p,p']][1/p,1/p',1/(p+p')]} \quad
\xymatrix{\mathcal A\ar^{\!\!\!\Delta_{\mathcal A}}[r]\ar[d] & (\mathcal A\otimes \mathcal A)
[{y^{(1)}-y^{(2)}\over x^{(1)}-x^{(2)}}]\ar[d]\\ 
\mathbf k[t]((p))\ar[r]& \mathbf k[t,t'][[p,p']][1/p,1/p',1/(p+p')]}
\end{equation}
where the bottom maps are given by $p\mapsto p+p'$, $t\mapsto t+t'$. 

\item[2)] The maps $a\mapsto a\otimes 1$, $a\mapsto 1\otimes a$ define morphisms from $A$ (resp., $\mathcal A$) to $A\otimes A$ 
(resp., \linebreak $\mathcal A\otimes\mathcal A$), and therefore also to the localizations $(A\otimes A)[{y^{(1)}-y^{(2)}\over 
x^{(1)}-x^{(2)}}]$ (resp., $(\mathcal A\otimes\mathcal A)[{y^{(1)}-y^{(2)}\over x^{(1)}-x^{(2)}}]$), such that the 
following diagrams commute 
\begin{equation}\label{diag:2}
\xymatrix{A\ar^{\!\!\!-\otimes 1(\text{ resp., }1\otimes-)}[r]\ar[d] & (A\otimes A)[{y^{(1)}-y^{(2)}\over x^{(1)}-x^{(2)}}]\ar[d]\\ 
\mathbf k((p))\ar[r]& \mathbf k[[p,p']][1/p,1/p',1/(p+p')]} \quad
\xymatrix{\mathcal A\ar^{\!\!\!-\otimes 1(\text{ resp., }1\otimes-)}[r]\ar[d] & (\mathcal A\otimes \mathcal A)
[{y^{(1)}-y^{(2)}\over x^{(1)}-x^{(2)}}]\ar[d]\\ 
\mathbf k[t]((p))\ar[r]& \mathbf k[t,t'][[p,p']][1/p,1/p',1/(p+p')]}
\end{equation}
where the bottom maps are given by substitution to $p,t$ of $p,t$ (resp., $p',t'$). 
\end{itemize}
\end{lemma}

\subsubsection{Algebraic Fay identity on $(E^\#)^2$}\label{sect:alg:fay}

Recall $\tilde F_{E,\omega}(p,z)\in{1\over{pz}}\mathbb C[[p,z]]$ from (\ref{def:tilde:F:E:omega}). 
Let also $t,t'$ be two additional formal variables. 

\begin{lemma}
One has 
\begin{align}\label{id:fay:mixte}
& \nonumber \tilde F_{E,\omega}(p,z)e^{-tz}\tilde F_{E,\omega}(p+p',z')e^{-(t+t')z'}-\tilde F_{E,\omega}(p',z')e^{-t'z'}\tilde F_{E,\omega}(p,z+z')e^{-t(z+z')}
\\ & +\tilde F_{E,\omega}(p+p',z+z')e^{-(t+t')(z+z')}\tilde F_{E,\omega}(p',-z)e^{t'z}=0
\end{align}
(identity in ${1\over{pp'zz'(p+p')(z+z')}}\mathbf k[t,t'][[p,p',z,z']]$).
\end{lemma}

\proof Follows from identity (\ref{id:fay:premixte}) in ${1\over{pp'zz'(p+p')(z+z')}}\mathbf k[[p,p',z,z']]$. \hfill \qed\medskip 

Let 
$$
\mathbf f(z):={1\over z}+\sum_{\alpha\geq0}f_\alpha z^\alpha\in{1\over z}\mathcal A[[z]]
$$
be the series defined in \S\ref{section:elements}. 

\begin{lemma} (Algebraic Fay identity.)
We have 
\begin{equation}\label{algebraic:fay}
(\mathbf f(z)\otimes 1)\Delta_{\mathcal A}(\mathbf f(z'))-\mathbf f(z+z')\otimes \mathbf f(z')
+(1\otimes\mathbf f(-z))\Delta_{\mathcal A}(\mathbf f(z+z'))=0
\end{equation}
(identity in ${1\over{zz'(z+z')}}(\mathcal A\otimes\mathcal A)[{y^{(1)}-y^{(2)}\over x^{(1)}-x^{(2)}}][[z,z']]$). 
\end{lemma}

This identity gives a family of relations between the $(f_\alpha)_{\alpha\geq 0}$, which are rational functions on $E^\#$, 
regular except at the preimage of $0\in E$; each of these relations is an identity 
of functions on $(E^\#)^2$, regular except at the preimages of $E\times\{0\}$, $\{0\}\times E$ and 
$E_{antidiag}:=\{(x,y)\in E^2|x+y=0\}$.  

\proof According to (\ref{id:interm}), the image of $\mathbf f(z)\in{1\over z}\mathcal A[[z]]$ under 
$\mathrm{can}:{1\over z}\mathcal A[[z]]\to{1\over z}\mathbf k[t]((p))[[z]]$ is $\tilde F_{E,\omega}(p,z)e^{-tz}$. This and the right sides of (\ref{diag:1})
and (\ref{diag:2}) imply that the elements 
$$
\tilde F_{E,\omega}(p,z)e^{-tz}, \quad 
\tilde F_{E,\omega}(p+p',z')e^{-(t+t')z'}, \quad 
\tilde F_{E,\omega}(p',z')e^{-t'z'}, 
$$
$$ 
\tilde F_{E,\omega}(p,z+z')e^{-t(z+z')}, \quad
\tilde F_{E,\omega}(p+p',z+z')e^{-(t+t')(z+z')}, \quad 
\tilde F_{E,\omega}(p',-z)e^{t'z}
$$
of ${1\over{pp'zz'(p+p')(z+z')}}\mathbf k[t,t'][[p,p',z,z']]$ are the images under the canonical map 
$$
{1\over{zz'(z+z')}}(\mathcal A\otimes\mathcal A)[{y^{(1)}-y^{(2)}\over x^{(1)}-x^{(2)}}][[z,z']]\to 
{1\over{zz'(z+z')}}\mathbf k[t,t'][[p,p']][1/p,1/p',1/(p+p')][[z,z']]
$$
of 
$$
\mathbf f(z)\otimes 1, \quad \Delta_{\mathcal A}(\mathbf f(z')), \quad 1\otimes\mathbf f(z'), \quad 
\mathbf f(z+z')\otimes 1, \quad \Delta_{\mathcal A}(\mathbf f(z+z')), \quad 1\otimes\mathbf f(-z). 
$$
(\ref{algebraic:fay}) now follows from the fact that this map is injective, together with identity (\ref{id:fay:mixte}). \hfill\qed\medskip

\subsubsection{An identity on $(E^\#)^n$}\label{sect:fay:En}

For $i<j\in[n]$, we denote by $m_{ij}:(E^\#)^n\to E^\#$ the morphism $(e_1,\ldots,e_n)\mapsto e_i-e_j$. 
For $i<j<k\in[n]$, we denote by $m_{ijk}:(E^\#)^n\to(E^\#)^2$ the morphism $(e_1,\ldots,e_n)\mapsto(e_i-e_j,e_j-e_k)$. 
We denote by $s:(E^\#)^2\to E^\#$ the sum morphism. Then $s\circ m_{ijk}=m_{ik}$.

 The morphism $m_{ij}$ restricts to a morphism of affine varieties
$$
\{(e_1,\ldots,e_n)\in(E^\#)^n|p_1\neq 0,\ldots,p_n\neq 0,p_i\neq p_j\}\stackrel{m_{ij}}{\to}E^\#-\pi^{-1}(0)
$$ 
where $p_1:=\pi(e_1),\ldots,p_n:=\pi(e_n)$. Similarly, $m_{ijk}$ restricts to a morphism 
\begin{align*}
& \{(e_1,\ldots,e_n)\in(E^\#)^n|p_1\neq 0,\ldots,p_n\neq 0,p_i\neq p_j,p_i\neq p_k,p_j\neq p_k\}
\\ & \stackrel{m_{ijk}}{\to}\{(e,e')\in(E^\#)^2|p\neq 0,p'\neq 0,p+p'\neq 0\}
\end{align*}
where $p:=\pi(e)$, $p':=\pi(e')$
and $s$ restricts to a morphism 
$$
\{(e,e')\in(E^\#)^2|p\neq 0,p'\neq 0,p+p'\neq 0\}\stackrel{s}{\to}E^\#-\pi^{-1}(0). 
$$
For $a\in\mathcal A$ and $i\in[n]$, set $a^{(i)}:=1^{i-1}\otimes a\otimes 1^{\otimes n-i}$. 
The regular function rings of the varieties involved in the above morphisms are the following 

\noindent
\begin{tabular}{|l|l|}
\hline $E^\#-\pi^{-1}(0)$ &$\mathcal A$ \\
\hline $\{(e,e')\in(E^\#)^2|p\neq 0,p'\neq 0,p+p'\neq 0\}$&$(\mathcal A\otimes\mathcal A)[{y^{(1)}-y^{(2)}\over x^{(1)}-x^{(2)}}]$ \\
\hline $\{(e_1,\ldots,e_n)\in(E^\#)^n|p_1\neq 0,\ldots,p_n\neq 0,p_i\neq p_j\}$ &$\mathcal A^{\otimes n}[{y^{(i)}+y^{(j)}\over x^{(i)}-x^{(j)}}]$ \\
\hline $\{(e_1,\ldots,e_n)\in(E^\#)^n|p_1\neq 0,\ldots,p_n\neq 0,p_i\neq p_j,p_i\neq p_k,p_j\neq p_k\}$& 
$\mathcal A^{\otimes n}[{y^{(i)}+y^{(j)}\over x^{(i)}-x^{(j)}},{y^{(i)}+y^{(k)}\over x^{(i)}-x^{(k)}},{y^{(j)}+y^{(k)}\over x^{(j)}-x^{(k)}}]$\\ 
\hline
\end{tabular}

The dual to the morphism $s$ is the morphism 
$$
\Delta_{\mathcal A}:\mathcal A\to (\mathcal A\otimes\mathcal A)[{y^{(1)}-y^{(2)}\over x^{(1)}-x^{(2)}}]
$$
defined in \S\ref{sect:operations}. 

The dual to the morphism $m_{ij}$ is the morphism 
$$
S_\mathcal A^{(j)}\Delta_{\mathcal A}^{(ij)}:\mathcal A\to
\mathcal A^{\otimes n}[{y^{(i)}+y^{(j)}\over x^{(i)}-x^{(j)}}], 
$$
where $\Delta_{\mathcal A}^{(ij)}$ is the composition of $\Delta_{\mathcal A}$ with the morphism induced by 
$\mathcal A\otimes\mathcal A\to\mathcal A^{\otimes n}$, $a\otimes b\mapsto 1^{\otimes i-1}\otimes a\otimes 1^{j-i-1}\otimes b\otimes 
1^{\otimes n-j}$ and ${y^{(1)}-y^{(2)}\over x^{(1)}-x^{(2)}}\mapsto{y^{(i)}-y^{(j)}\over x^{(i)}-x^{(j)}}$, $S_\mathcal A$ is the automorphism 
of $\mathcal A$ induced by $x\mapsto x$, $y\mapsto -y$, $\tilde c\mapsto -\tilde c$, and $S_{\mathcal A}^{(i)}$ is the 
isomorphism $\mathcal A^{\otimes n}[{y^{(i)}-y^{(j)}\over x^{(i)}-x^{(j)}}]\to\mathcal A^{\otimes n}[{y^{(i)}+y^{(j)}\over x^{(i)}-x^{(j)}}]$ extending the automorphism $\mathrm{id}^{\otimes i-1}\otimes S_\mathcal A\otimes\mathrm{id}^{\otimes n-i}$ 
of $\mathcal A^{\otimes n}$ by ${y^{(i)}-y^{(j)}\over x^{(i)}-x^{(j)}}\to-{y^{(i)}+y^{(j)}\over x^{(i)}-x^{(j)}}$. 

The dual of the morphism $m_{ijk}$ is the morphism 
$$
m_{ijk}^*:(\mathcal A\otimes\mathcal A)[{y^{(1)}-y^{(2)}\over x^{(1)}-x^{(2)}}]
\to\mathcal A^{\otimes n}[{y^{(i)}+y^{(j)}\over x^{(i)}-x^{(j)}},{y^{(i)}+y^{(k)}\over x^{(i)}-x^{(k)}},{y^{(j)}+y^{(k)}\over x^{(j)}-x^{(k)}}]
$$
induced by $\mathcal A\otimes\mathcal A\ni a\otimes b\mapsto S_\mathcal A^{(j)}\Delta_{\mathcal A}^{(ij)}(a)\cdot
S_\mathcal A^{(k)}\Delta_{\mathcal A}^{(jk)}(b)$, and 
$$
{y^{(1)}-y^{(2)}\over x^{(1)}-x^{(2)}}\mapsto 
-({y^{(i)}+y^{(j)}\over x^{(i)}-x^{(j)}}+{y^{(j)}+y^{(k)}\over x^{(j)}-x^{(k)}}+{y^{(k)}+y^{(i)}\over x^{(k)}-x^{(i)}}).
$$  

We then have $m_{ijk}^*\circ\Delta_{\mathcal A}=S_{\mathcal A}^{(k)}\Delta_{\mathcal A}^{(ik)}$. 

Applying $m_{ijk}^*$ to identity (\ref{algebraic:fay}), we obtain (see \cite{Fay,Mum,Po}):

\begin{lemma} (Algebraic Fay identities on $(E^\#)^n$)
One has
\begin{align}\label{algebraic:fay:ijk}
& \nonumber 
(S_{\mathcal A}^{(j)}\Delta_{\mathcal A}^{(ij)})(\mathbf f(z))
(S_{\mathcal A}^{(k)}\Delta_{\mathcal A}^{(ik)})(\mathbf f(z'))
-
(S_{\mathcal A}^{(j)}\Delta_{\mathcal A}^{(ij)})(\mathbf f(z+z'))
(S_{\mathcal A}^{(k)}\Delta_{\mathcal A}^{(jk)})(\mathbf f(z'))
\\ & 
+
(S_{\mathcal A}^{(k)}\Delta_{\mathcal A}^{(jk)})(\mathbf f(-z))
(S_{\mathcal A}^{(k)}\Delta_{\mathcal A}^{(ik)})(\mathbf f(z+z'))=0 
\end{align}
(identity in ${1\over{zz'(z+z')}}\mathcal A^{\otimes n}[{{y^{(i)}+y^{(j)}}\over{x^{(i)}-x^{(j)}}}, 
{{y^{(i)}+y^{(k)}}\over{x^{(i)}-x^{(k)}}},{{y^{(j)}+y^{(k)}}\over{x^{(j)}-x^{(k)}}}][[z,z']]$). 

\end{lemma}

\remark{Assume that $\mathbf k=\mathbb C$ and that $(E,0,\omega)=(\mathbb C/(\mathbb Z+\tau\mathbb Z),0,dp)$. 
The image of $\mathbf f(z)$ in ${1\over z}\mathbb C[t]((p))[[z]]$ is then 
$$
{\theta(p+z|\tau)\over\theta(p|\tau)\theta(z|\tau)}e^{-(t-{1\over2}G_2(\tau)p)z}, 
$$ 
so the image of $(S_{\mathcal A}^{(j)}\Delta_{\mathcal A}^{(ij)})(\mathbf f(z))$ in 
${1\over z(p_i-p_j)}\mathbb C[t_1,\ldots,t_n][[p_1,\ldots,p_n,z]]$ 
is 
$$
{\theta(p_{ij}+z|\tau)\over\theta(p_{ij}|\tau)\theta(z|\tau)}e^{-c_{ij}z}, 
$$
where $c_i:=t_i-{1\over2}G_2(\tau)p_i$ and $c_{ij}:=c_i-c_j$. \hfill\qed\medskip}\label{rem:atkatetiof}

Identity (\ref{algebraic:fay:ijk}) then translates as 
\begin{align*}
&  {{\theta(p_{ij}+z|\tau)}\over{\theta(p_{ij}|\tau)\theta(z|\tau)}}e^{-c_{ij}z}{{\theta(p_{ik}+z'|\tau)}\over{\theta(p_{ik}|\tau)\theta(z'|\tau)}}e^{-c_{ik}z'}
-{{\theta(p_{jk}+z'|\tau)}\over{\theta(p_{jk}|\tau)\theta(z'|\tau)}}e^{-c_{jk}z'}{{\theta(p_{ij}+(z+z')|\tau)}\over{\theta(p_{ij}|\tau)\theta(z+z'|\tau)}}e^{-c_{ij}(z+z')}
\\ & +{{\theta(p_{ik}+(z+z')|\tau)}\over{\theta(p_{ik}|\tau)\theta(z+z'|\tau)}}e^{-c_{ik}(z+z')}{{\theta(p_{jk}-z|\tau)}\over{\theta(p_{jk}|\tau)\theta(z|\tau)}}e^{c_{jk}z}
=0.  
\end{align*}
(identity in ${1\over{zz'(z+z')(p_i-p_j)(p_i-p_k)(p_j-p_k)}}\mathbb C[t_1,\ldots,t_n][[p_1,\ldots,p_n,z,z']]$).

\section{Cohomological computations related to $E^\#$}

In this \S, we choose an elliptic curve $E$ with fixed origin $0$, defined over a field $\mathbf k$ of char. 0. 

\subsection{Spaces of differentials on $E^\#$}\label{sect:one:forms}

\subsubsection{Computation of $\Gamma(E^\#,\mathcal O_{E^\#}(\pi^{-1}(0)))$}\label{subsect:4:1}

In \S\ref{section:elements}, we defined elements $f_\alpha\in\mathcal A$ and a morphism $\mathrm{can}:\mathcal A
\to\mathbf k[t]((p))$. As the left-hand side of (\ref{id:interm}) belongs to ${1\over p}\mathbf k[t][[p,z]]$, one has  
$\mathrm{can}(f_\alpha)\in{1\over p}\mathbf k[t][[p]]$ for any $\alpha\geq 0$. 
Since the divisor $\pi^{-1}(0)$ is locally defined by the equation $p=0$ in the plane $(p,t)$, this implies 
that the pole of $f_\alpha$ is simple, therefore
$$
\forall\alpha\geq 0, \quad f_\alpha\in \Gamma(E^\#,\mathcal O_{E^\#}(\pi^{-1}(0))). 
$$
There is an exact sequence 
$$
0\to \Gamma(E^\#,\mathcal O_{E^\#})\to \Gamma(E^\#,\mathcal O_{E^\#}(\pi^{-1}(0)))\stackrel{\mu}{\to}
\Gamma(\mathbb A^1,\mathcal O_{\mathbb A^1}),
$$
where the last map takes a rational function $f$ on $E^\#$ with a simple pole at $\pi^{-1}(0)$ to the restriction at 
$\pi^{-1}(0)\simeq\mathbb A^1$  of the product $f\cdot p$ ($p=0$ being a local equation of $\pi^{-1}(0)$). 

\begin{lemma}\label{lemma:basis:H0EsharpSimplePole}
Set $f_{-1}:=1$. The family $(f_\alpha)_{\alpha\geq-1}$ is a basis of $\Gamma(E^\#,\mathcal O_{E^\#}(\pi^{-1}(0)))$.
The map $\mu$ is such that 
\begin{equation}\label{mu:f:alpha}
\forall\alpha\geq 0, \quad \mu(f_\alpha)=(-t)^\alpha/\alpha!,  
\end{equation}
therefore $\mu$ is surjective.   
\end{lemma}

\proof The map $\mu$ fits in the diagram 
$$
\xymatrix{ 
\Gamma(E^\#,\mathcal O_{E^\#}(\pi^{-1}(0)))\ar^{\mu}@/^2pc/[rrr]\ar[r]\ar@^{^(->}[d]&
{1\over p}\mathbf k[t][[p]]\ar[r]\ar@^{^(->}[d] & \mathbf k[t]\ar^{\!\!\!\!\!\!\!\!\!\!\!\!\sim}[r]& \Gamma(\mathbb A^1,\mathcal O_{\mathbb A^1})
\\ \Gamma(E^\#,\mathcal O_{E^\#}(*\pi^{-1}(0)))\ar^{\ \ \ \ \ \ \ \ \mathrm{can}}[r]& \mathbf k[t]((p))& & }
$$
the map ${1\over p}\mathbf k[t][[p]]\to\mathbf k[t]$ being $f\mapsto (f\cdot p)_{|p=0}$. 
(\ref{id:interm}) implies
$$
e^{-tz}{\tilde\theta_{E,\omega}(p+z)\over\tilde\theta_{E,\omega}(p)}=\tilde\theta_{E,\omega}(z)\big({1\over z}
+\sum_{\alpha\geq 0}\mathrm{can}(f_\alpha)z^\alpha\big)
$$
(identity in ${1\over p}\mathbf k[t][[p,z]]$), therefore 
$$
e^{-tz}\cdot {p\over\tilde\theta_{E,\omega}(p)}\cdot\tilde\theta_{E,\omega}(p+z)
=p\cdot {\tilde\theta_{E,\omega}(z)\over z}+\sum_{\alpha\geq 0}p\mathrm{can}(f_\alpha)\tilde\theta_{E,\omega}(z)z^\alpha
$$
(identity in $\mathbf k[t][[p,z]]$). Evaluating at $p=0$, we get
$$
e^{-tz}\cdot\tilde\theta_{E,\omega}(z)=\sum_{\alpha\geq 0}(p\mathrm{can}(f_\alpha))_{|p=0}\tilde\theta_{E,\omega}(z)z^\alpha
$$ 
(identity in $\mathbf k[t][[z]]$), therefore 
$\sum_{\alpha\geq 0}(p\mathrm{can}(f_\alpha))_{|p=0}z^\alpha=e^{-tz}$, 
therefore for any $\alpha\geq 0$, $\mu(f_\alpha)=(p\mathrm{can}(f_\alpha))_{|p=0}=(-t)^\alpha/\alpha!$, 
proving (\ref{mu:f:alpha}). 

So the image of $(f_\alpha)_{\alpha\geq 0}$ by $\mu$ is a basis of $\Gamma(\mathbb A^1,\mathcal O_{\mathbb A^1})$. This implies 
both that $\mu$ is surjective, and that a basis of $\Gamma(E^\#,\mathcal O_{E^\#}(\pi^{-1}(0)))$ is the union of $(f_\alpha)_{\alpha\geq 0}$
and a basis of $\Gamma(E^\#,\mathcal O_{E^\#})$. The result follows from the fact that this latter space is $\simeq\mathbf k\cdot 1$. 
\hfill \qed\medskip   

\remark{One computes $f_{-1}=1$, $f_0=-\tilde c$, $f_1={1\over 2}(\tilde c^2-x)$. }

\subsubsection{Computation of $\Gamma(E^\#,\Omega^1_{E^\#}(\mathrm{log}\pi^{-1}(0)))$}

As $E^\#$ is an algebraic group, its sheaf $\Omega^1_{E^\#}$ of differentials is isomorphic to a direct sum of 
$\mathrm{dim}E^\#=2$ copies of the trivial bundle, generated by a basis of invariant differentials. 

The differentials 
$$
\underline{dc}:=d\tilde c-xdx/y\quad \text{and}\quad
\underline{dp}:=dx/y
$$
 form such a basis. Note that for $\mathbf k=\mathbb C$ and when a uniformization of 
$E_{an}$ is fixed, they correspond to the differentials $dc,dp$ on $E^\#_{an}$. We have therefore 
$$
\Omega^1_{E^\#}\simeq \mathcal O_{E^\#}\cdot \underline{dc}\oplus\mathcal O_{E^\#}\cdot \underline{dp}. 
$$
One computes 
$$
\Omega^1_{E^\#}(\mathrm{log}\pi^{-1}(0))\simeq \mathcal O_{E^\#}\cdot\underline{dc}\oplus\mathcal O_{E^\#}(\pi^{-1}(0))
\cdot \underline{dp}. 
$$

\begin{lemma}
\begin{itemize}
\item[1)] The family 
\begin{equation}\label{basis:algebraic}
d\tilde c-xdx/y, \quad \omega_\alpha:=f_\alpha\cdot {dx\over y}, \quad \alpha\geq -1, 
\end{equation}
where $f_\alpha$ is as in (\ref{def:f:alpha}), is a basis of $\Gamma(E^\#,\Omega^1_{E^\#}(\mathrm{log}\pi^{-1}(0)))$. 
\item[2)] The residue map 
$$
\Gamma(E^\#,\Omega^1_{E^\#}(\mathrm{log}\pi^{-1}(0)))\to\Gamma(\pi^{-1}(0),\mathcal O_{\pi^{-1}(0)})=\mathbf k[t]
$$
along $\pi^{-1}(0)$ is given by 
$$
\mathrm{res}(\underline{dc})=\mathrm{res}(\omega_{-1})=0, \quad \mathrm{res}(\omega_\alpha)=(-t)^\alpha/\alpha!\quad\mathrm{for}\quad
\alpha\geq 0. 
$$\end{itemize}
\end{lemma}

\proof 1) follows from Lemma \ref{lemma:basis:H0EsharpSimplePole}. 2) follows from (\ref{mu:f:alpha}). \hfill\qed\medskip  

\remark{If $\mathbf k=\mathbb C$ and a uniformization of $E_{an}\simeq\mathbb C/(\mathbb Z+\tau\mathbb Z)$ is fixed as in 
\S\ref{section:unif} (i.e., such that $dx/y\leftrightarrow dp$), the elements (\ref{basis:algebraic}) correspond to the following 
differentials on $E^\#_{an}$  
\begin{equation}\label{diff:forms:E:sharp}
dc,\quad \omega_{-1}:=dp, \quad \omega_\alpha:=[\Big({\theta(p+z|\tau)\over\theta(z|\tau)\theta(p|\tau)}e^{-cz}-{1\over z}\Big)
dp|z^\alpha], \quad 
\alpha\geq 0,  
\end{equation}
where $z$ is a formal variable and the notation $[-|z^\alpha]$ denotes the coefficient of $z^\alpha$ in the expansion in power series at $z=0$
(see \S\ref{subsect:4:1}).}

\subsection{Spaces of differentials  on $(E^\#)^n$}\label{sec:spaces:of:diffs}

\subsubsection{The setup}\label{subsect:5:1}

Let $n$ be an integer $\geq 1$; we set $[n]:=\{1,\ldots,n\}$ and $I:=\{(i,j)|i,j\in[n]$ and $i<j\}$. 

We assume $(X,D)$ to be as in \S\ref{sect:examples}. To emphasize the 
dependence of $D$ in $n$, we sometimes denote it $D^{(n)}$. 

\subsubsection{Computation of $\Gamma(X,\Omega^1_X(\mathrm{log}D_{ij}))$}\label{sect:4:2}\label{sec:coh}

Let $(i,j)\in I$. According to \S\ref{sect:1:2}, the space $\Gamma(X,\Omega^1_X(\mathrm{log}D_{ij}))$ consists of all forms $\alpha$ in 
$\Gamma_{rat}(X,\Omega^1_X)$ such that both $\alpha$ and $d\alpha$ are regular except for simple poles along $D_{ij}$. 
Moreover, the residue along $D_{ij}$ induces an exact sequence 
\begin{equation}\label{exact:D:ij}
0\to \Gamma(X,\Omega^1_X)\to \Gamma(X,\Omega^1_X(\mathrm{log}D_{ij}))\stackrel{\mu_{ij}}{\to} \Gamma(D_{ij},\mathcal O_{D_{ij}}). 
\end{equation}

The automorphism $\mathrm{map}_{ij}$ of $(E^\#)^n$ given by 
$$
\mathrm{map}_{ij}:(e_1,\ldots,e_n)\mapsto(e_i-e_j,e_1,\ldots,e_{j-1},e_{j+1},\ldots,e_n)
$$
induces an isomorphism $D_{ij}\simeq\pi^{-1}(0)\times(E^\#)^{n-1}$. The isomorphism $\pi^{-1}(0)\simeq\mathbb A^1$ 
then induces an isomorphism 
\begin{equation}\label{iso:Dij}
D_{ij}\simeq\mathbb A^1\times(E^\#)^{n-1}. 
\end{equation}
The isomorphisms $\Gamma(E^\#,\mathcal O_{E^\#})\simeq{\mathbf k}$, 
$\Gamma(\mathbb A^1,\mathcal O_{\mathbb A^1})\simeq{\mathbf k}[t]$ then induce an isomorphism  
\begin{equation}\label{H0:Dij}
\Gamma(D_{ij},\mathcal O_{D_{ij}})\simeq{\mathbf k}[t] 
\end{equation}
(in the analytic context, $t=c_{ij}$, see Remark \ref{rem:atkatetiof}). 

\definition{
For $\alpha\geq 0$ and $(i,j)\in I$, define $\omega_{ij}^\alpha$ as the image of $\omega_\alpha$ under the composed map 
\begin{align*}
\Gamma(E^\#,\Omega^1_{E^\#}(\mathrm{log}\pi^{-1}(0)))
& \stackrel{-\otimes 1^{\otimes n-1}}{\to} 
\Gamma(E^\#\times(E^\#)^{n-1},\Omega^1_{E^\#}(\mathrm{log}\pi^{-1}(0))\boxtimes\mathcal O_{E^\#}^{\boxtimes n-1}) \\ & 
\stackrel{\mathrm{map}_{ij}^*}{\to} \Gamma((E^\#)^n,\Omega^1_{(E^\#)^n}(\mathrm{log}D_{ij})). 
\end{align*}}

\begin{lemma}\label{tlminis}
The map $\mu_{ij}$ in (\ref{exact:D:ij}) is surjective. 
\end{lemma}

\proof One checks that the image of $\omega_{ij}^\alpha$ by the residue map $\Gamma(X,\Omega^1_X(\mathrm{log}D_{ij}))
\to \Gamma(D_{ij},\mathcal O_{D_{ij}})\simeq{\mathbf k}[t]$ equals $(-t)^\alpha/\alpha!$. So the composed map 
$\mathrm{Span}((\omega_{ij}^\alpha)_{\alpha\geq 0})\to \Gamma(X,\Omega^1_X(\mathrm{log}D_{ij}))\to \Gamma(D_{ij},\mathcal O_{D_{ij}})$
is surjective. The result follows. \hfill\qed\medskip 

For $i\in[n]$, set 
$$
\underline{dc}_i:=d\tilde c_{(i)}-x_{(i)}dx_{(i)}/y_{(i)},\quad\underline{dp}_i:=dx_{(i)}/y_{(i)}. 
$$
(see \S\ref{subsect:5:1}). 

\begin{lemma}\label{stat:basis:0}
The family $ (\underline{dc}_i,\underline{dp}_i)_{i\in[n]}$ is a basis of $\Gamma(X,\Omega^1_X)$. 
\end{lemma}

\proof As $\Gamma(E^\#,\mathcal O_{E^\#})\simeq\mathbf k$, we have $\Gamma(X,\Omega^1_X)\simeq \Gamma(E^\#,\Omega^1_{E^\#})^{\oplus n}$, which implies the result. \hfill \qed\medskip 

It follows from the definition of $\omega_{ij}^\alpha$ and from Lemma \ref{stat:basis:0} that 
\begin{equation}\label{basis:Dij}
(\omega_{ij}^\alpha)_{\alpha\geq 0}, \quad (\underline{dc}_i)_{i\in[n]}, \quad (\underline{dp}_i)_{i\in[n]}
\end{equation}
constitutes a family in $\Gamma(X,\Omega^1_X(\mathrm{log}D_{ij}))$. 

\begin{lemma}\label{lemma:basis:Dij}
Family (\ref{basis:Dij}) is a basis of $\Gamma(X,\Omega^1_X(\mathrm{log}D_{ij}))$. 
\end{lemma}

\proof Since the composed map 
$$
\Gamma(X,\Omega^1_X)\simeq\mathrm{Span}((\underline{dc}_i)_{i\in[n]},(\underline{dp}_i)_{i\in[n]})\to\mathrm{Span}
((\omega_{ij}^\alpha)_{\alpha\geq 0},(\underline{dc}_i)_{i\in[n]},(\underline{dp}_i)_{i\in[n]})
\to 
\Gamma(X,\Omega^1_X(\mathrm{log}D_{ij}))
$$ 
is the canonical map $\Gamma(X,\Omega^1_X)\to \Gamma(X,\Omega^1_X(\mathrm{log}D_{ij}))$, the map 
$$
\mathrm{Span}((\omega_{ij}^\alpha)_{\alpha\geq 0},(\underline{dc}_i)_{i\in[n]},(\underline{dp}_i)_{i\in[n]})
\to \Gamma(X,\Omega^1_X(\mathrm{log}D_{ij}))
$$
is surjective. Moreover, the image of $(\omega_{ij}^\alpha)_{\alpha\geq 0}$ under the residue map is a linearly independent family in 
$\Gamma(D_{ij},\mathcal O_{D_{ij}})$, and the family $((\underline{dc}_i)_{i\in[n]},(\underline{dp}_i)_{i\in[n]})$ is also linearly 
independent, which implies that family (\ref{basis:Dij}) is linearly independent. All this proves the result. 
\hfill\qed\medskip

It follows from the definition of $\omega_{ij}^\alpha$ that 
$$
{\underline{dp}_{ij}\over z}+\sum_{\alpha\geq 0}\omega_{ij}^\alpha z^\alpha
=(S_{\mathcal A}^{(j)}\Delta_{\mathcal A}^{(ij)})(\mathbf f(z))\cdot\underline{dp}_{ij}. 
$$ 
where the notation is as in \S\S\ref{sect:alg:fay}, \ref{sect:fay:En} (identity in $\Gamma_{rat}((E^\#)^n,\Omega^1_{(E^\#)^n})$).

\remark Assume that $\mathbf k=\mathbb C$ and a uniformization $E_{an}\simeq\mathbb C/(\mathbb Z+\tau\mathbb Z)$ as in 
\S\ref{section:unif} (i.e., such that $dx/y\leftrightarrow dp$) is fixed. Then a uniformization of $(E^\#)^n$ is $(E^\#)^n_{an}=
\mathrm{Coker}(\mathbb Z^{2n}\to\mathbb C^n)$, where the morphism $\mathbb Z^{2n}\to\mathbb C^n$ is 
$\delta_{2i-1}\mapsto e_i$, $\delta_{2i}\mapsto \tau e_i$, where $i\in[n]$ (the canonical bases of $\mathbb Z^{2n}$ and $\mathbb C^n$ 
being $(\delta_i)_{i\in[2n]}$ and $(e_i)_{i\in[n]}$). Let $(p_1,c_1,\ldots,p_n,c_n)$ be the standard coordinates on $\mathbb C^n$. Then 
$$
(S_{\mathcal A}^{(j)}\Delta_{\mathcal A}^{(ij)})(\mathbf f(z))={\theta(p_{ij}+z|\tau)\over\theta(z|\tau)\theta(p_{ij}|\tau)}e^{-c_{ij}z}, 
$$ 
so 
\begin{equation}\label{form:omega:ijalpha}
\omega_{ij}^\alpha=[\Big({\theta(p_{ij}+z|\tau)\over\theta(z|\tau)\theta(p_{ij}|\tau)}e^{-c_{ij}z}-{1\over z}\Big)dp_{ij}|z^\alpha] 
\end{equation}
and $p_{ij}:=p_i-p_j$.
The elements $dc_i,dp_i$ ($i\in[n]$) form a basis of $\Gamma((E^\#)^n,\Omega^1_{(E^\#)^n})$
and are the images of the elements $\underline{dc}_i,\underline{dp}_i$.

\subsubsection{Computation of $\sum_{(i,j)\in I}\Gamma(X,\Omega^1_X(\mathrm{log}D_{ij}))$}

Recall that for each $(i,j)\in I$, $\Gamma(X,\Omega^1_X(\mathrm{log}D_{ij}))$ is a vector subspace of $\Gamma_{rat}(X,\Omega^1_X)$. 
Then according to \S\ref{sect:cos},
$$
\bm{\Omega}^1=\sum_{(i,j)\in I}\Gamma(X,\Omega^1_X(\mathrm{log}D_{ij}))
$$
is the sum of these subspaces. It follows from Lemma \ref{lemma:basis:Dij} that the following family 
\begin{equation}\label{diff:E:n}
\underline{dc}_i \quad (i\in[n]),\quad \underline{dp}_i \quad (i\in[n]),\quad \omega_{ij}^\alpha \quad (i<j\in[n], \quad \alpha\geq 0), 
\end{equation}
lies in, and spans, this vector space. 

According to \S\ref{subsect:maps}, the residue along $D_{ij}$ induces a linear map 
$\mathrm{Res}_{D_{ij}}^{(1)}:\bm{\Omega}^1\to \Gamma(D_{ij},\mathcal O_{D_{ij}})$. Composing this map with the isomorphism 
$\Gamma(D_{ij},\mathcal O_{D_{ij}})\simeq{\mathbf k}[t]$
(see (\ref{H0:Dij})), we obtain a linear map 
\begin{equation}\label{rho:i:j}
\varrho_{ij}:\bm{\Omega}^1\to{\mathbf k}[t]. 
\end{equation}

\begin{lemma}\label{Poinc:1}
The map $\varrho_{ij}$ is given by 
$$
\underline{dc}_{i'}\mapsto 0\quad (i'\in[n]),\quad \underline{dp}_{i'}\mapsto 0\quad (i'\in[n]),\quad 
\omega_{i'j'}^\alpha\mapsto \delta_{(i',j'),(i,j)}\cdot(-t)^\alpha/\alpha! \quad ((i,j)\in I).  
$$ 
\end{lemma}

\proof This follows from the fact that $\varrho_{ij}$ is the sum of the residue maps computed in \S\ref{sect:4:2}, from (\ref{mu:f:alpha}), 
and from the fact that if $(i,j)\neq(i',j')\in I$, then the residue along $D_{ij}$ of $\omega_{i'j'}^\alpha$ is zero. \hfill\qed\medskip 

\begin{lemma}\label{basis:Esharp:n}
The family (\ref{diff:E:n}) is a basis of $\bm{\Omega}^1=\sum_{(i,j)\in I}\Gamma((E^\#)^n,\Omega^1_{(E^\#)^n}(\mathrm{log}D_{ij}))$.
\end{lemma}

\proof Lemma \ref{Poinc:1} implies that the family $(\omega_{ij}^\alpha)_{(i,j)\in I,\alpha\geq 0}$ maps to a basis under the map $\bm{\Omega}^1
\to\oplus_{(i,j)\in I}\Gamma(D_{ij},\mathcal O_{D_{ij}})$. The remaining elements in family (\ref{diff:E:n}) form a basis in the kernel 
$\Gamma(X,\Omega^1_X)$ of this map. It follows that the family (\ref{diff:E:n}) is also linearly independent. \hfill\qed\medskip 

Recall that $\bm{\Omega}^1\hookrightarrow\Gamma_{rat}((E^\#)^n,\Omega^1_{(E^\#)^n})$ and that the bundle 
$\Omega^1_{(E^\#)^n}$ is trivial, namely 
\begin{equation}\label{eq:Omega:1}
\Omega^1_{(E^\#)^n}\simeq\mathcal O_{(E^\#)^n}\otimes
\big(\oplus_{i=1}^n({\mathbf k}\underline{dp}_i\oplus{\mathbf k}\underline{dc}_i)\big).
\end{equation} 
This induces an isomorphism 
$$
\Gamma_{rat}((E^\#)^n,\Omega^1_{(E^\#)^n})\simeq {\mathbf k}((E^\#)^n)
\otimes\big(\oplus_{i=1}^n({\mathbf k}\underline{dp}_i\oplus{\mathbf k}\underline{dc}_i)\big)=\Gamma_{rat}^p\oplus\Gamma_{rat}^c, 
$$
where 
$$
\Gamma_{rat}^p:={\mathbf k}((E^\#)^n)
\otimes(\oplus_{i=1}^n{\mathbf k}\underline{dp}_i),\quad\Gamma_{rat}^c:={\mathbf k}((E^\#)^n)
\otimes(\oplus_{i=1}^n{\mathbf k}\underline{dc}_i). 
$$
Lemma \ref{basis:Esharp:n} then implies: 

\begin{lemma}\label{lemma:basis:HH1}
\begin{itemize}
\item $\bm{\Omega}^1$ is graded for this decomposition, that is 
$$
\bm{\Omega}^1=\bm{\Omega}^1_p\oplus\bm{\Omega}^1_c, \quad\text{where}\quad
\bm{\Omega}^1_p:=\bm{\Omega}^1\cap\Gamma_{rat}^p, \quad 
\bm{\Omega}^1_c:=\bm{\Omega}^1\cap\Gamma_{rat}^c,
$$
\item 
the elements $\underline{dp}_i$ ($i\in[n]$), $\omega_{ij}^\alpha$ ($i<j\in[n]$, $\alpha\geq 0$) form a $\mathbf k$-basis of 
$\bm{\Omega}^1_p$ and therefore 
a linearly independent family of ${\mathbf k}((E^\#)^n)\otimes(\oplus_{i\in[n]}{\mathbf k}\cdot\underline{dp}_i)$,   
\item the elements $\underline{dc}_i$ ($i\in[n]$) form a $\mathbf k$-basis of $\bm{\Omega}^1_c$. 
\end{itemize}
\end{lemma}

\subsubsection{Rational $2$-forms on $(E^\#)^n$}

It follows from (\ref{eq:Omega:1}) that 
$$
\Omega^2_{(E^\#)^n}\simeq\mathcal O_{(E^\#)^n}\otimes\Lambda^2(\bigoplus_{i\in[n]}{\mathbf k}\cdot \underline{dc}_i
\oplus\bigoplus_{i\in[n]}{\mathbf k}\cdot \underline{dp}_i), 
$$
therefore 
$$
\Gamma_{rat}((E^\#)^n,\Omega^2_{(E^\#)^n})\simeq{\mathbf k}((E^\#)^n)\otimes
\Lambda^2(\bigoplus_{i\in[n]}{\mathbf k}\cdot \underline{dc}_i\oplus\bigoplus_{i\in[n]}{\mathbf k}\cdot \underline{dp}_i). 
$$
We set 
\begin{align}\label{decomp}
& \Gamma_{rat}^{cc}:={\mathbf k}((E^\#)^n)\otimes\Lambda^2(\oplus_{i\in[n]}{\mathbf k}\cdot \underline{dc}_i), \quad 
\Gamma_{rat}^{cp}:={\mathbf k}((E^\#)^n)\otimes(\oplus_{i,j\in[n]}{\mathbf k}\cdot (\underline{dc}_i\wedge \underline{dp}_j)),\\ & \nonumber 
\Gamma_{rat}^{pp}:={\mathbf k}((E^\#)^n)\otimes\Lambda^2(\oplus_{i\in[n]}{\mathbf k}\cdot \underline{dp}_i). 
\end{align}
Then 
$$
\Gamma_{rat}((E^\#)^n,\Omega^2_{(E^\#)^n})=\Gamma_{rat}^{cc}\oplus\Gamma_{rat}^{cp}\oplus\Gamma_{rat}^{pp}. 
$$
It follows from (\ref{decomp}) that 
\begin{equation}\label{freeness:cc}
\text{the family }(\underline{dc}_i\wedge \underline{dc}_j)_{i<j\in[n]}\text{ is linearly independent over $\mathbf k$ in }\Gamma_{rat}^{cc}, 
\end{equation}
and from (\ref{decomp}) and Lemma \ref{lemma:basis:HH1} that 
\begin{align}\label{freeness:cp}
& \text{the union of two families }(\underline{dc}_i\wedge \underline{dp}_j)_{i,j\in[n]}, 
(\underline{dc}_i\wedge\omega_{jk}^\alpha)_{i\in[n],j<k\in[n],\alpha\geq 0}
\\ & \nonumber \text{ is linearly independent over $\mathbf k$ in }\Gamma_{rat}^{cp},   
\end{align}
where $\wedge$ is the product in the algebra $\mathbf k((E^\#)^n)\otimes\Lambda^\bullet(\oplus_{i\in[n]}\mathbf k\underline{dp}_i\oplus
\mathbf k\underline{dc}_i)$.

\subsection{Computation of the kernel of $\owedge:\Lambda^2(\bm{\Omega}^1)\to\bm{\Omega}^2$}

In this section, we drop $(E^\#)^n$ from the global section notation, so $\Gamma(\mathcal L)$ will mean 
$\Gamma((E^\#)^n,\mathcal L)$. Likewise, $\Gamma_{rat}(\mathcal L)$ means $\Gamma_{rat}((E^\#)^n,\mathcal L)$. 

\subsubsection{Construction of linear maps}\label{sect:351}

According to \S\ref{subsect:maps}, the wedge product of differential forms induces a linear map
\begin{equation}\label{lin:map}
\Lambda^2(\bm{\Omega}^1)\to\bm{\Omega}^2\hookrightarrow\Gamma_{rat}(\Omega^2_{(E^\#)^n})
\end{equation}

According to Lemma \ref{lemma:basis:HH1}, (1), the source of this morphism decomposes as follows 
$$
\Lambda^2(\bm{\Omega}^1)\simeq\Lambda^2(\bm{\Omega}^1_c)
\oplus \bm{\Omega}^1_c\otimes \bm{\Omega}^1_p\oplus\Lambda^2(\bm{\Omega}^1_p), 
$$
while the target decomposes as $\Gamma_{rat}(\Omega^2_{(E^\#)^n})\simeq\Gamma^{cc}_{rat}\oplus\Gamma^{cp}_{rat}
\oplus\Gamma^{pp}_{rat}$. The linear map (\ref{lin:map}) then decomposes as the direct sum of linear maps
\begin{equation}\label{lin:map:cc}
\Lambda^2(\bm{\Omega}^1_c)\to\Gamma^{cc}_{rat},
\end{equation}
\begin{equation}\label{lin:map:cp}
\bm{\Omega}^1_c\otimes \bm{\Omega}^1_p\to\Gamma^{cp}_{rat},
\end{equation}
\begin{equation}\label{lin:map:pp}
\Lambda^2(\bm{\Omega}^1_p)\to\Gamma^{pp}_{rat}. 
\end{equation}
and (\ref{freeness:cc}) and (\ref{freeness:cp}) imply: 
\begin{lemma}
The linear maps (\ref{lin:map:cc}) and (\ref{lin:map:cp}) are injective. 
\end{lemma}

We will now compute the kernel of the map (\ref{lin:map:pp}). 

\subsubsection{Construction of a subspace $\mathbf K$ of $\mathrm{Ker}(\Lambda^2(\bm{\Omega}^1_p)\to\Gamma^{pp}_{rat})$}
\label{sect:352}

By Lemma \ref{lemma:basis:HH1}, a basis of $\bm{\Omega}^1_p$ is given by the family 
$$
\underline{dp}_i \quad (i\in[n]),\quad \omega_{ij}^\alpha \quad (i<j\in[n],\quad \alpha\geq 0). 
$$
One derives from there the following basis of $\Lambda^2(\bm{\Omega}^1_p)$: 
\begin{equation}\label{elt:P}
P(i,j):=\underline{dp}_i\underline\wedge \underline{dp}_j \quad (i<j\in[n]),
\end{equation}
\begin{equation}\label{elt:Q3}
Q(i,j,k,\alpha):=\underline{dp}_i\underline\wedge\omega_{jk}^\alpha \quad 
(i,j,k\in[n] \mathrm{\ all\ different,} \quad j<k,\quad \alpha\geq 0), 
\end{equation}
\begin{equation}\label{elt:Q'2}
Q'(i,j,\alpha):=\underline{dp}_i\underline\wedge\omega_{ij}^\alpha \quad 
(i<j\in[n],\quad \alpha\geq 0), 
\end{equation}
\begin{equation}\label{elt:Q''2}
Q''(i,j,\alpha):=\underline{dp}_j\underline\wedge\omega_{ij}^\alpha \quad 
(i<j\in[n],\quad \alpha\geq 0), 
\end{equation}
\begin{equation}\label{elt:S4}
S(i,j,k,l,\alpha,\beta):=
\omega_{ij}^\alpha\underline\wedge\omega_{kl}^\beta \quad (
i,j,k,l\in[n]\mathrm{\ all\ different,\ }i<j,\quad k<l,\quad i<k,\quad \alpha,\beta\geq 0), 
\end{equation}
\begin{equation}\label{elt:S'3}
S'(i,j,k,\alpha,\beta):=\omega_{ij}^\alpha\underline\wedge\omega_{ik}^\beta \quad (i<j<k\in[n],\quad \alpha,\beta\geq 0), 
\end{equation}
\begin{equation}\label{elt:S''3}
S''(i,j,k,\alpha,\beta):=\omega_{ij}^\alpha\underline\wedge\omega_{jk}^\beta \quad (i<j<k\in[n],\quad \alpha,\beta\geq 0), 
\end{equation}
\begin{equation}\label{elt:S'''3}
S'''(i,j,k,\alpha,\beta):=\omega_{ik}^\alpha\underline\wedge\omega_{jk}^\beta \quad (i<j<k\in[n],\quad \alpha,\beta\geq 0), 
\end{equation}
\begin{equation}\label{elt:S2}
S(i,j,\alpha,\beta):=\omega_{ij}^\alpha\underline\wedge\omega_{ij}^\beta \quad (i<j\in[n],\quad 0\leq \alpha<\beta), 
\end{equation}
where $\underline\wedge$ is the product in the algebra $\Lambda^\bullet(\bm{\Omega}^1_p)$. 

\begin{definition}\label{def:R}
 $\mathbf K\subset\Lambda^2(\bm{\Omega}^1_p)$ is the subspace spanned by 
\begin{equation}\label{elt:R}
R(i,j,\alpha):=
(Q'-Q'')(i,j,\alpha), 
\quad i<j\in[n],\quad \alpha\geq 0, 
\end{equation}
$$
S(i,j,\alpha,\beta), 
\quad i<j,\quad \alpha>\beta\geq 0, 
$$
\begin{align}\label{elt:T}
& T(i,j,k,\alpha,\beta):=
S'(i,j,k,\alpha,\beta)-\sum_{\stackrel{\gamma,\delta\geq 0,}{{\gamma+\delta=\alpha+\beta}}}
\begin{pmatrix} \delta \\ \alpha\end{pmatrix}S''(i,j,k,\delta,\gamma) 
+\sum_{\stackrel{\gamma,\delta\geq 0,}{{\gamma+\delta=\alpha+\beta}}}(-1)^\delta\begin{pmatrix} 
\gamma \\ \beta\end{pmatrix}S'''(i,j,k,\gamma,\delta) \\ & \nonumber
-\begin{pmatrix} \alpha+\beta+1 \\ \beta\end{pmatrix}(Q'(i,k,\alpha+\beta+1)-Q(j,i,k,\alpha+\beta+1))
\\ & \nonumber+\begin{pmatrix} \alpha+\beta+1 \\ \alpha\end{pmatrix}
(Q''(i,j,\alpha+\beta+1)-Q(k,i,j,\alpha+\beta+1))\\ & \nonumber+(-1)^{\alpha+1}(Q(i,j,k,\alpha+\beta+1)-Q''(j,k,\alpha+\beta+1)), 
 \quad i<j<k\in[n],\quad \alpha,\beta\geq 0. 
\end{align}
\end{definition}
The wedge product of forms defines a linear map $\Lambda^2(\bm{\Omega}^1_p)\to\Gamma^{pp}_{rat}$. If $i<j\in[n]$ and 
$\alpha\geq 0$, the image of $R(i,j,\alpha)$ is then $\underline{dp}_{ij}\owedge\omega_{ij}^\alpha$, where $\owedge$ is the 
operation introduced in \S\ref{subsect:maps}, and $\underline{dp}_{ij}:=\underline{dp}_i-\underline{dp}_j$. 
According to (\ref{basis:algebraic}), $\omega_\alpha=f_\alpha\cdot\omega_{-1}$ for $\alpha\geq-1$, so that 
$\omega_{ij}^\alpha=\mathrm{map}_{ij}^*(f_\alpha)\cdot\omega_{ij}^{-1}$ for $\alpha\geq-1$. As $\omega_{ij}^{-1}
=\underline{dp}_{ij}$, this implies $\underline{dp}_{ij}\owedge\omega_{ij}^\alpha=0$. 
The image of $S(i,j,\alpha,\beta)$ is
$\omega_{ij}^\alpha\owedge\omega_{ij}^\beta$, which is $0$ for the same reason.  


The product of identity (\ref{algebraic:fay}) by $\underline{dp}_{ij}\owedge\underline{dp}_{ik}$ expresses as the statement that the element
\begin{align*}
&\Big({{\underline{dp}_{ij}}\over z}+\sum_{\alpha\geq 0}\omega_{ij}^\alpha z^\alpha\Big)\underline\wedge\Big(
{{\underline{dp}_{ik}}\over u}+\sum_{\beta\geq 0}\omega_{ik}^\beta u^\beta\Big)
\\ & +\Big({{\underline{dp}_{jk}}\over u}+\sum_{\alpha\geq 0}\omega_{jk}^\alpha u^\alpha\Big)\underline\wedge\Big(
{{\underline{dp}_{ij}}\over z+u}+\sum_{\beta\geq 0}\omega_{ij}^\beta (z+u)^\beta\Big)
\\ & +\Big({{\underline{dp}_{ik}}\over z+u}+\sum_{\alpha\geq 0}\omega_{ik}^\alpha (z+u)^\alpha\Big)\underline\wedge\Big(
{{\underline{dp}_{jk}}\over -z}+\sum_{\beta\geq 0}\omega_{jk}^\beta (-z)^\beta\Big)
\end{align*}
of the space ${1\over zu(z+u)}\Lambda^2(\bm{\Omega}^1_p)[[z,u]]$ belongs to the kernel of the map
$$
{1\over zu(z+u)}\Lambda^2(\bm{\Omega}^1_p)[[z,u]]\to
{1\over zu(z+u)}\Gamma^{pp}_{rat}[[z,u]]
$$
induced by (\ref{lin:map:pp}). 
As this element is equal to $\sum_{\alpha,\beta\geq 0}T(i,j,k,\alpha,\beta)z^\alpha u^\beta$, this implies that each 
$T(i,j,k,\alpha,\beta)$ belongs to the kernel of the map $\Lambda^2(\bm{\Omega}^1_p)\to \Gamma^{pp}_{rat}$ for $i<j<k\in[n]$, 
$\alpha,\beta\geq 0$. All this implies: 

\begin{lemma}
$\mathbf K$ is contained in the kernel $\mathrm{Ker}(\Lambda^2(\bm{\Omega}^1_p)\to\Gamma^{pp}_{rat})$. 
\end{lemma}

\subsubsection{A complementary subspace of $\mathbf K$ in $\Lambda^2(\bm{\Omega}^1_p)$}
\label{sec:4:5:3}

Let $\mathbf P_2$ (resp., $\mathbf Q_3$, $\mathbf Q'_2$, $\mathbf Q''_2$, $\mathbf S_4$, $\mathbf S'_3$, $\mathbf S''_3$, 
$\mathbf S'''_3$, $\mathbf S_2$) be the linear span of elements (\ref{elt:P}) (resp., (\ref{elt:Q3}), (\ref{elt:Q'2}), (\ref{elt:Q''2}), 
(\ref{elt:S4}), (\ref{elt:S'3}), (\ref{elt:S''3}), (\ref{elt:S'''3}), (\ref{elt:S2})); the notation of these vector spaces has been chosen in 
agreement with the notation for their generating sets, the index indicating the number of free Latin indices. 

According to the beginning of \S \ref{sect:352}, the space $\Lambda^2(\bm{\Omega}^1_p)$ decomposes as a 
direct sum 
$$
\Lambda^2(\bm{\Omega}^1_p)=\mathbf P_2\oplus\mathbf Q_3\oplus\mathbf Q'_2\oplus\mathbf Q''_2\oplus\mathbf S_4\oplus
\mathbf S'_3\oplus\mathbf S''_3\oplus\mathbf S'''_3\oplus\mathbf S_2. 
$$ 
The subspace $\mathbf K\subset \Lambda^2(\Gamma(\mathcal E))$ decomposes as 
$$
\mathbf K=\mathbf S_2+\mathbf Y+\mathbf Z, 
$$
where $\mathbf Y$ is the image of the linear map 
$$
y:\mathbf Q'_2\to \mathbf Q'_2\oplus \mathbf Q''_2, \quad Q'(i,j,\alpha)\mapsto Q'(i,j,\alpha)\oplus (-Q''(i,j,\alpha)), 
$$
and $\mathbf Z$ is the image of the linear map 
$$z:\mathbf S'_3\to \mathbf S'_3\oplus \mathbf S''_3\oplus \mathbf S'''_3\oplus \mathbf Q_3\oplus \mathbf Q'_2\oplus \mathbf Q''_2,
$$
\begin{align*}
&S'(i,j,k,\alpha,\beta)\mapsto S'(i,j,k,\alpha,\beta)\oplus\Big(-\sum_{\gamma,\delta\geq 0,\gamma+\delta=\alpha+\beta}
\begin{pmatrix}\delta\\ \alpha\end{pmatrix}S''(i,j,k,\delta,\gamma)\Big)\\ & 
\oplus\Big(\sum_{\gamma,\delta\geq 0,\gamma+\delta=\alpha+\beta}
(-1)^\delta\begin{pmatrix}\gamma\\ 
\beta\end{pmatrix} S'''(i,j,k,\gamma,\delta)\Big)
\oplus
\Big(\begin{pmatrix}\alpha+\beta+1\\ \beta\end{pmatrix}Q(j,i,k,\alpha+\beta+1)
\\ & -\begin{pmatrix}\alpha+\beta+1\\ \alpha\end{pmatrix}Q(k,i,j,\alpha+\beta+1)+
(-1)^{\alpha+1}Q(i,j,k,\alpha+\beta+1)\Big)
\\ & \oplus\Big(-\begin{pmatrix}\alpha+\beta+1\\ \beta\end{pmatrix}Q'(i,k,\alpha+\beta+1)\Big)\\ & 
\oplus\Big(\begin{pmatrix}\alpha+\beta+1\\ \alpha\end{pmatrix}Q''(i,j,\alpha+\beta+1)+(-1)^\alpha
Q''(j,k,\alpha+\beta+1)\Big).
\end{align*}
We recall that if $\phi:X\to Y$ is a linear map between the vector spaces $X$ and $Y$, then 
the graph $\mathrm{Graph}(\phi)$ of $\phi$ is the subspace of $X\oplus Y$, image of the linear map
$X\to X\oplus Y$, $x\mapsto x\oplus f(x)$. 

The spaces $\mathbf Y$ and $\mathbf Z$ are therefore the graphs of the linear maps $\tilde y:\mathbf Q'_2\to \mathbf Q''_2$ and 
$\tilde z:\mathbf S'_3\to\mathbf S''_3\oplus \mathbf S'''_3\oplus \mathbf Q_3\oplus \mathbf Q'_2\oplus \mathbf Q''_2$ obtained 
from $y$ and $z$ by composition with the direct sum of all the summands but the first of their target spaces. 

\begin{lemma}
Let $A,B,C,D$ be vector spaces and $f:A\to B\oplus C\oplus D$, $g:C\to D$ be linear maps. Then 
the direct sum $A\oplus B\oplus C\oplus D$ admits a direct sum decomposition
$$
A\oplus B\oplus C\oplus D\simeq (B\oplus D)\oplus\mathrm{Graph}(f)\oplus\mathrm{Graph}(g).
$$ 
\end{lemma}

{\em Proof.} Recall that if $\phi:X\to Y$ is a linear map, then $\mathrm{Graph}(\phi)$ is a complement of $X$
in $X\oplus Y$. It follows that there is a direct sum decomposition 
\begin{equation}\label{partial:decomp}
A\oplus B\oplus C\oplus D\simeq(B\oplus C\oplus D)\oplus\mathrm{Graph}(f). 
\end{equation}
One similarly has a direct sum decomposition 
$$
C\oplus D\simeq D\oplus\mathrm{Graph}(g),  
$$
therefore 
$$
B\oplus C\oplus D\simeq(B\oplus D)\oplus\mathrm{Graph}(g). 
$$
The result follows from the combination of this decomposition with (\ref{partial:decomp}). 
\hfill \qed\medskip 

Applying this lemma with $A:=\mathbf S'_3$, $B:=\mathbf S''_3\oplus\mathbf S'''_3\oplus\mathbf Q_3$, $C:=\mathbf Q'_2$, 
$D:=\mathbf Q''_2$, $f:=\tilde z$, $g:=\tilde y$, 
we obtains a direct sum decomposition
$$
\mathbf Q_3\oplus\mathbf Q'_2\oplus\mathbf Q''_2\oplus\mathbf S'_3\oplus\mathbf S''_2\oplus\mathbf S'''_3\simeq\mathbf K\oplus 
\mathbf T\oplus(\mathbf S''_3\oplus\mathbf S'''_3\oplus\mathbf Q_3\oplus\mathbf Q''_2). 
$$
Taking the direct sum with $\mathbf P_2\oplus\mathbf S_4\oplus\mathbf S_2$ and distributing the summands, we obtain 
$$
\Lambda^2(\Gamma(\mathcal E))\simeq(\mathbf K\oplus\mathbf S_2\oplus\mathbf T)\oplus(\mathbf S''_3\oplus\mathbf S'''_3\oplus
\mathbf Q_3\oplus\mathbf Q''_2\oplus\mathbf P_2\oplus\mathbf S_4). 
$$
This implies: 
\begin{lemma}\label{lemma:suppl}
A complement of $\mathbf K$ in $\Lambda^2(\bm{\Omega}^1_p)$ is 
$$
\Sigma:=\mathbf S''_3\oplus \mathbf S'''_3\oplus \mathbf Q_3\oplus\mathbf Q''_2\oplus\mathbf P_2\oplus\mathbf S_4. 
$$
A basis of $\Sigma$ is 
\begin{equation}\label{basis:1}
P(i,j)=\underline{dp}_i\underline\wedge\underline{dp}_j, \quad i<j\in[n], 
\end{equation}
\begin{equation}\label{basis:2}
Q(i,j,k,\alpha)=\underline{dp}_i \underline\wedge \omega_{jk}^\alpha, \quad j<k, \quad i\neq j,k, \quad i,j,k\in[n], \quad \alpha\geq 0, 
\end{equation}
\begin{equation}\label{basis:3}
Q'(i,j,\alpha)=\underline{dp}_i\underline\wedge \omega_{ij}^\alpha, \quad i<j\in[n], \quad \alpha\geq 0, 
\end{equation}
\begin{equation}\label{basis:4}
S''(i,j,k,\alpha,\beta)=\omega_{ij}^\alpha\underline\wedge\omega_{jk}^\beta,\quad i<j<k\in[n], \quad \alpha,\beta\geq 0, 
\end{equation}
\begin{equation}\label{basis:5}
S'''(i,j,k,\alpha,\beta)=\omega_{ik}^\alpha\underline\wedge\omega_{jk}^\beta,\quad i<j<k\in[n], \quad \alpha,\beta\geq 0,  
\end{equation}
\begin{align}\label{basis:6}
S(i,j,k,l,\alpha,\beta)=\omega_{ij}^\alpha\underline\wedge\omega_{kl}^\beta, \quad i<j,\quad k<l,\quad i<k,\quad i,j,k,l&
\mathrm{\ all\ different\ in\ }[n], 
\\ & \nonumber \alpha,\beta\geq 0. 
\end{align}
\end{lemma}

\subsubsection{Residue maps}
\label{sect:4:5:4}

According to \S\ref{subsect:maps}, there exists for any $(i,j)\in I$ a linear map 
\begin{equation}\label{res:Dij:2}
\mathrm{Res}_{D_{ij}}^{(2)}:\bm{\Omega}^2\to\sum_{(k,l)\in I|(k,l)\neq(i,j)}\Gamma(D_{ij},\Omega^1_{D_{ij}}(\mathrm{log}D_{ij}\cap D_{kl}))
=:\mathbf D_{ij}^1
\end{equation}
and for any $(k,l)\neq(i,j)\in I$, a linear map 
$$
\mathrm{Res}_{D_{ij}\cap D_{kl}}^{(1)}:\mathbf D_{ij}^1\to \Gamma(D_{ij}\cap D_{kl},\mathcal O_{D_{ij}\cap D_{kl}})=:\mathbf D_{ij,kl}^0. 
$$
The restriction of this map to the summand corresponding to $(k',l')$ is 0 unless $D_{ij}\cap D_{kl}=D_{ij}\cap D_{k'l'}$, which 
happens only if $(k,l)=(k',l')$ or $\{k,l\}\cup\{i,j\}=\{k',l'\}\cup\{i,j\}$.  

\subsubsection{Computation of residue maps}
\label{subsect:4:5:4}

Let $(i,j)\in I$. For any $(k,l)\neq (i,j)\in I$, the isomorphism $D_{ij}\simeq\mathbb A^1\times(E^\#)^{n-1}$ (see (\ref{iso:Dij})) 
takes the divisor $D_{ij}\cap D_{kl}$ to $\mathbb A^1\times D^{(n-1)}_{f_{ij}(k)f_{ij}(l)}$, where $f_{ij}:[n]\to[n-1]$ is the map given by 
$i,j\mapsto j$ and to induce an increasing bijection $[n]-\{i,j\}\to[n-1]-\{j\}$ (the exponent $(n-1)$ means that the divisor is in 
$(E^\#)^{n-1}$). 
This induces an isomorphism of the summand $(k,l)$ of $\mathbf D_{ij}^1$
with 
$$
\Gamma(\mathbb A^1\times(E^\#)^{n-1},\Omega^1_{\mathbb A^1\times(E^\#)^{n-1}}
(\mathrm{log}(\mathbb A^1\times D^{(n-1)}_{f_{ij}(k)f_{ij}(l)}))). 
$$
For $X,Y$ nonsingular varieties and $D\subset Y$ a nonsingular divisor, one has $\Omega^1_{X\times Y}(\mathrm{log}(X\times D))
\simeq\Omega^1_X\boxtimes\mathcal O_Y\oplus\mathcal O_X\boxtimes\Omega^1_Y(\mathrm{log}D)$. Using the identifications
$\Gamma(\mathbb A^1,\mathcal O_{\mathbb A^1})\simeq{\mathbf k}[t]$, $\Gamma(\mathbb A^1,\Omega_{\mathbb A^1})\simeq{\mathbf k}[t]dt$, 
the latter space then identifies
with 
$$
{\mathbf k}[t]dt\oplus \Gamma((E^\#)^{n-1},\Omega^1_{(E^\#)^{n-1}}(\mathrm{log}D^{(n-1)}_{f_{ij}(k)f_{ij}(l)}))[t]. 
$$
Let $I_{(n-1)}$, $\bm{\Omega}^1_{(n-1)}$ be analogues of $I,\bm{\Omega}^1$ with $n$ replaced by $n-1$. The map 
$(k,l)\mapsto(f_{ij}(k),f_{ij}(l))$ induces a surjective map $I-\{(i,j)\}\to I_{(n-1)}$. It follows that  
\begin{equation}\label{isom:Dij:1}
\mathbf D_{ij}^1\simeq{\mathbf k}[t]dt\oplus\bm{\Omega}^1_{(n-1)}[t]. 
\end{equation}

If $(i,j)\neq(k,l)\in I$, then the composition of the isomorphism $D_{ij}\cap D_{kl}\simeq\mathbb A^1\times D^{(n-1)}_{f_{ij}(k)f_{ij}(l)}$
induced by (\ref{iso:Dij}) and of the isomorphism $D^{(n-1)}_{f_{ij}(k)f_{ij}(l)}\simeq\mathbb A^1\times(E^\#)^{n-2}$ induces an 
isomorphism $D_{ij}\cap D_{kl}\simeq(\mathbb A^1)^2\times(E^\#)^{n-2}$ and therefore an isomorphism 
$$
\mathbf D^0_{ij,kl}\simeq{\mathbf k}[t,t']. 
$$

Let $\owedge:\Lambda^2(\bm{\Omega}^1_p)\to\bm{\Omega}^2$ be the linear map induced by the wedge product. One checks that the 
composed map 
$$
\Lambda^2(\bm{\Omega}^1_p)\stackrel{\owedge}{\to}\bm{\Omega}^2\stackrel{\mathrm{Res}_{D_{ij}}^{(2)}}{\to}\mathbf D_{ij}^1
\simeq\bm{\Omega}^1_{(n-1)}[t]\oplus{\mathbf k}[t]dt
$$
has its image contained in $\bm{\Omega}^1_{(n-1)}[t]$. We denote by $\varrho_{ij}^{(2)}:\Lambda^2(\bm{\Omega}^1_p)\to\bm{\Omega}^1_{(n-1)}[t]$ 
the resulting corestricted map. Then the diagram 
\begin{equation}\label{diag:rho:i:j:2}
\xymatrix{ \Lambda^2(\bm{\Omega}^1_p)\ar^{\owedge}[r]\ar_{\varrho_{ij}^{(2)}}[rrdd]&\bm{\Omega}^2\ar^{\mathrm{Res}_{D_{ij}}^{(2)}}[r] & \mathbf D_{ij}^1\\ & & \bm{\Omega}^1_{(n-1)}[t]\oplus\mathbf k[t]dt\ar_{\sim}[u]\\ & & \bm{\Omega}^1_{(n-1)}[t]\ar@{^(->}[u]}
\end{equation}
commutes. 

This diagram implies: 
\begin{lemma}\label{lemma:crit}
The kernel $\mathrm{Ker}(\owedge:\Lambda^2(\bm{\Omega}^1_p)\to\bm{\Omega}^2)$ is contained in the intersection over all $(i,j)\in I$ 
of the kernels of the maps $\varrho_{ij}^{(2)}:\Lambda^2(\bm{\Omega}^1_p)\to\bm{\Omega}^1_{(n-1)}[t]$. 
\end{lemma}
This implies: 
\begin{cor}\label{cor:crit}
The kernel $\mathrm{Ker}(\owedge:\Lambda^2(\bm{\Omega}_p^1)\to\bm{\Omega}^2)$ 
is contained in the intersection over all pairs $(i,j)\neq(k,l)\in I$ of the kernels of the maps
$$
\varrho_{f_{ij}(k)f_{ij}(l)}\circ\varrho_{ij}^{(2)}:\Lambda^2(\bm{\Omega}_p^1)\to{\mathbf k}[t,t']. 
$$
\end{cor}

We now turn to the computation of $\varrho_{ij}^{(2)}$. 
As $\owedge$ vanishes on $\mathbf K\subset\Lambda^2(\bm{\Omega}^1_p)$, so does $\varrho_{ij}^{(2)}$. The restriction of 
$\varrho_{ij}^{(2)}$ to $\Sigma$ can be computed using 
 Lemma \ref{Poinc:1} and \S\ref{subsect:pr}. One gets: 
\begin{lemma} 
Let $(i_0,j_0)\in I$. The map 
$$
\varrho_{i_0j_0}^{(2)}:\Lambda^2(\bm{\Omega}^1_p)\to\bm{\Omega}^1_{(n-1)}[t]
$$
is given by $R\mapsto 0$, and the following formulas: 
\begin{itemize}
\item for $i<j\in[n]$, $P(i,j)\mapsto 0$,  

\item for $i,j,k\in[n]$, $j<k$, $i\neq j,k$, and $\alpha\geq 0$, 
$$
Q(i,j,k,\alpha)\mapsto \left\{
                \begin{array}{ll}
                  -{(-t)^\alpha\over\alpha!}\cdot\underline{dp}_{f_{i_0j_0}(i)}\ \text{if}\  (j,k)=(i_0,j_0), \\
0\ \text{else},                
                \end{array}
              \right.
$$
\item for $i<j\in[n]$, $\alpha\geq 0$, 
$Q'(i,j,\alpha)\mapsto \left\{ \begin{array}{ll} -(-t)^\alpha/\alpha!\cdot\underline{dp}_{i_0}\ \text{if}\  (i,j)=(i_0,j_0), \\ 
 0\ \text{else},\end{array}\right.$
\item for $i<j<k\in[n]$, $\alpha,\beta\geq 0$, 
$$
S''(i,j,k,\alpha,\beta)\mapsto \left\{
                \begin{array}{ll}
             {(-t)^\alpha\over\alpha!}\cdot \omega_{i_0f_{i_0j_0}(k)}^\beta\ \text{if}\  (i,j)=(i_0,j_0),\\
 -{(-t)^\beta\over\beta!}\cdot \omega_{f_{i_0j_0}(i)i_0}^\alpha\ \text{if}\  (j,k)=(i_0,j_0),\\
0\ \text{else}, 
                \end{array}
              \right.
$$
\item for $i<j<k\in[n]$, $\alpha,\beta\geq 0$, 
$$
S'''(i,j,k,\alpha,\beta)\mapsto \left\{
                \begin{array}{ll}
                {(-t)^\alpha\over\alpha!}\cdot (-1)^\beta\omega_{i_0f_{i_0j_0}(j)}^\beta\ \text{if}\  (i,k)=(i_0,j_0),\\
 -{(-t)^\beta\over\beta!}\cdot \omega_{f_{i_0j_0}(i)i_0}^\alpha\ \text{if}\  (j,k)=(i_0,j_0),\\
0\ \text{else}, 
                \end{array}
              \right.
$$
\item for $i,j,k,l$ all different in $[n]$, $ i<j$, $k<l$, $i<k$, $\alpha,\beta\geq 0$, 
$$
S(i,j,k,l,\alpha,\beta)\mapsto \left\{
                \begin{array}{ll}
                {(-t)^\alpha\over\alpha!}\cdot \omega_{f_{i_0j_0}(k)f_{i_0j_0}(l)}^\beta\ \text{if}\  (i,j)=(i_0,j_0),\\
                  -{(-t)^\beta\over\beta!}\cdot \omega_{f_{i_0j_0}(i)f_{i_0j_0}(j)}^\alpha\ \text{if}\  (k,l)=(i_0,j_0),\\0\ \text{else}, 
                \end{array}
              \right.
$$
\end{itemize}
where $f_{i_0j_0}:[n]\to[n-1]$ is as in the beginning of \S\ref{subsect:4:5:4}. 
\end{lemma}

\remark{
Let $(k,l)\neq(i,j)\in I$. The commutative diagram 
$$
\xymatrix{
D_{ij}\cap D_{kl}\ar^{\sim}[r]\ar@{^(->}[d]& \mathbb A^1\times D^{(n-1)}_{f_{ij}(k)f_{ij}(l)}\ar@{^(->}[d]\\ 
D_{ij}\ar^{\sim}[r]& \mathbb A^1\times(E^\#)^{n-1}}
$$
implies that the following diagram commutes
$$
\xymatrix{
\bm{\Omega}^1_{(n-1)}[t]\ar@{^(->}[d]\ar^{\varrho_{f_{ij}(k)f_{ij}(l)}}[r]& {\mathbf k}[t',t]\\ \mathbf D_{ij}^1
\ar^{\mathrm{Res}_{D_{ij}\cap D_{kl}}^{(1)}}[r]& \mathbf D_{ij,kl}^0\ar[u]}
$$
where the upper map the tensor product with ${\mathbf k}[t]$ of the analogue of the map from (\ref{rho:i:j}),
with $(n,i,j)$ replaced by $(n-1,f_{ij}(k),f_{ij}(l))$ and $t$ replaced by $t'$.  

Combining diagram (\ref{diag:rho:i:j:2}) with the above diagram, one gets, for any $(i,j)\neq(k,l)\in I$, a commutative
diagram 
$$
\xymatrix{
\Lambda^2(\bm{\Omega}_p^1)\ar^{\owedge}[r]\ar_{\varrho_{ij}^{(2)}}[rrdd]&\bm{\Omega}^2\ar^{\mathrm{Res}_{D_{ij}}^{(2)}}[r] & 
\mathbf D_{ij}^1\ar^{\mathrm{Res}_{D_{ij}\cap D_{kl}}^{(1)}}[r]& \mathbf D_{ij,kl}^0 \\
& & \bm{\Omega}_{(n-1)}^1[t]\oplus{\mathbf k}[t]dt\ar_\sim[u]&  \\
& & \bm{\Omega}_{(n-1)}^1[t]\ar^{\varrho_{f_{ij}(k)f_{ij}(l)}}[r]\ar@{^(->}[u]& \mathbb C[t,t']\ar_{\sim}[uu] }
$$
which gives an interpretation of the map in Corollary \ref{cor:crit}. \hfill \qed\medskip

\subsubsection{Compositions of residue maps}

Computation yields: 

\begin{lemma}
Assume that $i_0,j_0,k_0,l_0\in[n]$ are all distinct, such that $i_0<j_0$, $k_0<l_0$, $i_0<k_0$. Then
$$
\varrho_{f_{i_0j_0}(k_0)f_{i_0j_0}(l_0)}\circ\varrho^{(2)}_{i_0j_0|\Sigma}:\Sigma\to{\mathbf k}[t,t']
$$ 
takes $P_2\oplus Q_3\oplus Q_2\oplus S''_3\oplus S'''_3$ to $0$, and its restriction to $S_4$ is given by 
$$
S(i,j,k,l,\alpha,\beta)\mapsto \left\{
                \begin{array}{ll}
                 {(-t)^\alpha(-t')^\beta\over \alpha!\beta!}\ \text{if}\  (i,j,k,l)=(i_0,j_0,k_0,l_0),\\
                 0\ \text{else}, 
                \end{array}
              \right.
$$
for distinct $i,j,k,l\in[n]$, such that $i<j$, $k<l$, $i<k$, and $\alpha,\beta\geq 0$. 
\end{lemma}

\begin{lemma} Assume that $i_0<j_0<k_0\in[n]$. Then 
$$
\varrho_{i_0f_{i_0k_0}(j_0)}\circ\varrho^{(2)}_{i_0k_0|\Sigma}:\Sigma\to{\mathbf k}[t,t']
$$ 
takes $P_2\oplus Q_3\oplus Q_2\oplus S''_3\oplus S_4$ to $0$, and its restriction to $S'''_3$ is given by 
$$
S'''(i,j,k,\alpha,\beta)\mapsto \left\{
                \begin{array}{ll}
                 {(-t)^\alpha(t')^\beta\over\alpha!\beta!}\ \text{if}\  (i,j,k)=(i_0,j_0,k_0),\\
                 0\ \text{else}, 
                \end{array}
              \right.
$$
for $i<j<k\in[n]$, and $\alpha,\beta\geq 0$. 
\end{lemma}

\begin{lemma} Assume that $i_0<j_0<k_0\in[n]$. Then 
$$
\varrho_{i_0f_{i_0j_0}(k_0)}\circ\varrho^{(2)}_{i_0j_0|\Sigma}:\Sigma\to{\mathbf k}[t,t']
$$ 
takes $P_2\oplus Q_3\oplus Q_2\oplus S'''_3\oplus S_4$ to $0$, and its restriction to $S''_3$ is given by 
$$
S''(i,j,k,\alpha,\beta)\mapsto \left\{
                \begin{array}{ll}
                 {(-t)^\alpha(-t')^\beta\over\alpha!\beta!}\ \text{if}\  (i,j,k)=(i_0,j_0,k_0),\\
                 0\ \text{else}, 
                \end{array}
              \right.
$$
for $i<j<k\in[n]$, and $\alpha,\beta\geq 0$. 
\end{lemma}

\subsubsection{Computation of $\mathrm{Ker}(\Lambda^2(\bm{\Omega}^1_p)\to\Gamma^{pp}_{rat})$}

\begin{lemma}\label{map:Sigma:injective}
The map $\owedge_{|\Sigma}:\Sigma\subset\Lambda^2(\bm{\Omega}^1_p)\to\bm{\Omega}^2$ is injective.
\end{lemma}

\proof 
Let $\sigma\in\Sigma$ be an element of $\mathrm{Ker}(\owedge:\Lambda^2(\bm{\Omega}^1_p)\to\bm{\Omega}^2)$. 
Then $\sigma$ can be decomposed as 
\begin{align*}
&  \sigma=\sum_{i<j\in[n]}p(i,j)P(i,j)+\underset{\alpha\geq 0}{\underset{\#\{i,j,k\}=3, j<k}{\sum_{i,j,k\in[n],}}}q(i,j,k,\alpha)Q(i,j,k,\alpha)+
\underset{\alpha\geq 0}{\sum_{i<j\in[n]}}q'(i,j,\alpha)Q'(i,j,\alpha)
\\ &  +\underset{\alpha,\beta\geq 0}{\sum_{i<j<k\in[n]}}s''(i,j,k,\alpha,\beta)S''(i,j,k,\alpha,\beta)
+\underset{\alpha,\beta\geq 0}{\sum_{i<j<k\in[n]}}s'''(i,j,k,\alpha,\beta)S'''(i,j,k,\alpha,\beta)
\\ & +\underset{\alpha,\beta\geq 0}{\underset{i<j,k<l,i<k}{\sum_{i,j,k,l\in[n],\#\{i,j,k,l\}=3,}}}s(i,j,k,l,\alpha,\beta)S(i,j,k,l,\alpha,\beta), 
\end{align*}
where $p(i,j)$, etc., are suitable scalars. 

Assume that $i_0,j_0,k_0,l_0\in[n]$ are all distinct, and that $i_0<j_0$, $k_0<l_0$, $i_0<k_0$. By Corollary \ref{cor:crit},  
$\varrho_{f_{i_0j_0}(k_0)f_{i_0j_0}(l_0)}\circ\varrho^{(2)}_{i_0j_0}(\sigma)=0$, so 
$\sum_{\alpha,\beta}s(i_0,j_0,k_0,l_0,\alpha,\beta){(-t)^\alpha(-t')^\beta\over \alpha!\beta!}=0$, therefore 
$s(i_0,j_0,k_0,l_0,\alpha,\beta)=0$. 

Assume that $i_0<j_0<k_0\in[n]$. 

By Corollary \ref{cor:crit}, $\varrho_{i_0f_{i_0k_0}(j_0)}\circ\varrho^{(2)}_{i_0k_0}(\sigma)=0$, so 
$$\sum_{\alpha,\beta}s'''(i_0,j_0,k_0,\alpha,\beta){(-t)^\alpha(t')^\beta\over\alpha!\beta!}=0,$$ therefore 
$s'''(i_0,j_0,k_0,\alpha,\beta)=0$. 

By Corollary \ref{cor:crit}, $\varrho_{i_0f_{i_0j_0}(k_0)}\circ\varrho^{(2)}_{i_0j_0}(\sigma)=0$, so 
$$\sum_{\alpha,\beta}s''(i_0,j_0,k_0,\alpha,\beta){(-t)^\alpha(-t')^\beta\over\alpha!\beta!}=0,$$ therefore 
$s''(i_0,j_0,k_0,\alpha,\beta)=0$. 

It follows that
$$
\sigma=\sum_{i<j\in[n]}p(i,j)P(i,j)+\underset{\alpha\geq 0}{\underset{\#\{i,j,k\}=3, j<k}{\sum_{i,j,k\in[n],}}}q(i,j,k,\alpha)Q(i,j,k,\alpha)+
\underset{\alpha\geq 0}{\sum_{i<j\in[n]}}q'(i,j,\alpha)Q'(i,j,\alpha). 
$$

Assume that $i_0<j_0<k_0\in[n]$. 

By Lemma \ref{lemma:crit}, $\varrho^{(2)}_{i_0j_0}(\sigma)=0$, so 
$$
\sum_{\alpha\geq 0}
(\sum_{i\in[n]-\{i_0,j_0\}}q(i,i_0,k_0,\alpha)\underline{dp}_{f_{i_0j_0}(i)}+
q'(i_0,j_0,\alpha)\underline{dp}_{i_0})(-{(-t)^\alpha\over \alpha!})=0. 
$$
As the restriction of the map $f_{i_0j_0}$ to $[n]-\{i_0,j_0\}$ is injective and as its image does not contain $i_0$, 
we have $q(i,i_0,k_0,\alpha)=0$ for any $i\in[n]-\{i_0,j_0\}$, $\alpha\geq 0$ and $q'(i_0,j_0,\alpha)=0$ for any $\alpha\geq 0$. 

This implies that $\sigma=\sum_{i<j\in[n]}p(i,j)P(i,j)$. The relation $\owedge(\sigma)=0$ yields 
$$
\sum_{i<j\in[n]}p(i,j)\underline{dp}_i\owedge\underline{dp}_j=0.
$$
As the family $(\underline{dp}_i\owedge\underline{dp}_j)$ is 
linearly independent over $\mathbf k$ in 
$\bm{\Omega}^2$, we obtain $p(i,j)=0$ for any $i,j$, therefore $\sigma=0$. 
\hfill\qed\medskip

\begin{lemma}\label{lemma:comp:kernel}
The kernel of $\owedge:\Lambda^2(\bm{\Omega}^1_p)\to\bm{\Omega}^2$, and therefore also of the map  
$\Lambda^2(\bm{\Omega}^1_p)\stackrel{\owedge}{\to}\bm{\Omega}^2\hookrightarrow\Gamma_{rat}^{pp}$, is equal to $\mathbf K$. 
\end{lemma}

\proof  Recall that $\Lambda^2(\bm{\Omega}^1_p)=\mathbf K\oplus\Sigma$, that 
$\owedge$ is a map $\Lambda^2(\bm{\Omega}^1_p)\to\bm{\Omega}^2$ and that 
$\mathbf K\subset\mathrm{Ker}(\owedge)$. The result now follows from Lemma \ref{map:Sigma:injective}. 
\hfill\qed\medskip 

Combining this result with those from \S\ref{sect:351}, we obtain: 
\begin{prop}
The kernel of the wedge product map $\Lambda^2(\bm{\Omega}^1)\to\bm{\Omega}^2\hookrightarrow\Gamma_{rat}(\Omega^2_{(E^\#)^n})$
is equal to the subspace $\mathbf K$ of $\Lambda^2(\bm{\Omega}^1_p)\subset\Lambda^2(\bm{\Omega}^1)$. 
\end{prop}

\section{Presentation and computation of the Lie algebra $\mathfrak G$ (equivs. ($\mathrm{f}$) and ($\mathrm{g}$))}
\label{section:iso:G:t}

\subsection{Grading on $\bm{\Omega}^1$}
\label{sec:5:1} 

Taking into account Lemma \ref{basis:Esharp:n} which makes explicit a basis of $\bm{\Omega}^1$, we may define a
grading on $\bm{\Omega}^1$ by 
$$
\mathrm{deg}(\omega_{ij}^\alpha)=\alpha+1\quad\mathrm{for}\quad\alpha\geq 0,\quad i<j\in[n], \quad
\mathrm{deg}(\underline{dc}_i)=1, \quad \mathrm{deg}(\underline{dp}_i)=0\quad \mathrm{for}\quad i\in[n].  
$$
Then $\bm{\Omega}^1$ is graded in degrees $\geq 0$. For $d\geq 0$, we denote by $\bm{\Omega}^1[d]$ the degree $d$ part of 
$\bm{\Omega}^1$. So 
$$
\bm{\Omega}^1=\oplus_{d\geq 0}\bm{\Omega}^1[d],\quad \bm{\Omega}^1[0]=\oplus_{i\in[n]}\mathbf k \underline{dp}_i,\quad \bm{\Omega}^1[1]=(\oplus_{i\in[n]}\mathbf k \underline{dc}_i)\oplus(\oplus_{i<j\in[n]}\mathbf k\omega_{ij}^0),
$$
$$
\bm{\Omega}^1[d]=\oplus_{i<j\in[n]}\mathbf k\omega_{ij}^{d-1}\quad\mathrm{for}\quad d\geq 2. 
$$

\subsection{A graded space $\mathbf I$}
\label{sec:5:2} 

The linear map $\owedge:\Lambda^2(\bm{\Omega}^1)\to\bm{\Omega}^2$ has been defined in \S\ref{subsect:maps} and its kernel has been 
identified with $\mathbf K$ (see Def. \ref{def:R}) in Lemma \ref{lemma:comp:kernel}. The grading of $\bm{\Omega}^1$ from \S\ref{sec:5:1} 
induces a grading on $\Lambda^2(\bm{\Omega}^1)$. The generators of $\mathbf K\subset\Lambda^2(\bm{\Omega}^1)$ are homogeneous for
this grading: $R(i,j,\alpha)$ is pure of degree $\alpha+1$, and $S(i,j,\alpha,\beta)$ and $T(i,j,k,\alpha,\beta)$ are pure of degree 
$\alpha+\beta+2$. This implies that {\it $\mathbf K$ is a graded subspace of $\Lambda^2(\bm{\Omega}^1)$.} It follows that the quotient space
$\Lambda^2(\bm{\Omega}^1)/\mathbf K$ inherits a grading from $\Lambda^2(\bm{\Omega}^1)$.

We make the following definition: 
\begin{definition}\label{def:I}
$\mathbf I:=\mathrm{im}(\owedge:\Lambda^2(\bm{\Omega}^1)\to\bm{\Omega}^2)$
\end{definition}

As $\mathbf I$ is isomorphic to $\Lambda^2(\bm{\Omega}^1)/\mathbf K$, one defines a grading on $\mathbf I$ by transport of structure. 
Then the composed map $\Lambda^2(\bm{\Omega}^1)\to\Lambda^2(\bm{\Omega}^1)/\mathbf K\simeq\mathbf I$ is compatible with the gradings. 
Therefore: 
\begin{lemma}
The space $\mathbf I$ is equipped with a grading which is compatible with the map $\owedge:\Lambda^2(\bm{\Omega}^1)\to\mathbf I$.  
\end{lemma}
Recall that 
\begin{equation}\label{decomp:Lambda2}
\Lambda^2(\bm{\Omega}^1)\simeq\Lambda^2(\bm{\Omega}^1_c)\oplus(\bm{\Omega}^1_c\otimes\bm{\Omega}^1_p)\oplus
\Lambda^2(\bm{\Omega}^1_p),
\end{equation}
that $\mathbf K\subset\Lambda^2(\bm{\Omega}^1_p)$ and that a complement $\Sigma$ of $\mathbf K$ in $\Lambda^2(\bm{\Omega}^1_p)$ has been constructed
in Lemma \ref{lemma:suppl}. It follows that {\it there is an isomorphism}
\begin{equation}\label{isom:I}
\mathbf I\simeq\Lambda^2(\bm{\Omega}^1_c)\oplus(\bm{\Omega}^1_c\otimes\bm{\Omega}^1_p)\oplus\Sigma. 
\end{equation}
One derives from there: 

{\it the images under $\Lambda^2(\bm{\Omega}^1)/\mathbf K\simeq\mathbf I$ of the classes in $\Lambda^2(\bm{\Omega}^1)/\mathbf K$ 
of the elements of the family
$$
(\underline{dc}_i\underline{\wedge}\underline{dc}_j)_{i<j\in[n]}, (\underline{dc}_i\underline{\wedge} \underline{dp}_j)_{i,j\in[n]}, (\underline{dc}_i\underline{\wedge}\omega_{jk}^\alpha)_{i\in[n],j<k\in[n],\alpha\geq 0}, 
(P(i,j))_{i<j\in[n]},  (Q(i,j,k,\alpha))_{i\neq j,k\in[n],\alpha\geq 0}, 
$$
$$ (Q'(i,j,\alpha))_{i<j\in[n],\alpha\geq 0},\quad 
(S''(i,j,k,\alpha,\beta))_{i<j<k\in[n],\alpha,\beta\geq 0},\quad(S''(i,j,k,\alpha,\beta))_{i<j<k\in[n],\alpha,\beta\geq 0},
$$
$$
(S(i,j,k,l,\alpha,\beta))_{i<j,k<l,i<k\in[n],\#\{i,j,k,l\}=4,\alpha,\beta\geq 0},
$$
(see (\ref{basis:1})-(\ref{basis:6})) of elements in $\Lambda^2(\bm{\Omega}^1)$ form a graded basis of $\mathbf I$.}

\subsection{Computation of the map $d:\bm{\Omega}^1\to\bm{\Omega}^2$} \label{comput:d}
\label{subsect:5:3}

The map $d:\bm{\Omega}^1\to\bm{\Omega}^2$ may be computed as follows: 
$$
\underline{dc}_i\mapsto 0, \quad i\in[n], \quad \underline{dp}_i\mapsto 0, \quad i\in[n], \quad \omega_{ij}^\alpha\mapsto \left\{
                \begin{array}{ll}
                  -\underline{dc}_{ij}\owedge\omega_{ij}^{\alpha-1}\ \text{if}\  \alpha>0,\\
-\underline{dc}_{ij}\owedge \underline{dp}_{ij}\ \text{if}\  \alpha=0,       
                \end{array}
              \right.
i<j\in[n]. 
$$
where $\underline{dc}_{ij}:=\underline{dc}_i-\underline{dc}_j$, $\underline{dp}_{ij}:=\underline{dp}_i-\underline{dp}_j$.

\subsection{Grading on the coalgebra $\mathbf C$ and the Lie coalgebra $\mathfrak C$}
\label{subsect:5:4}

It follows from \S\ref{comput:d} that the image of $d$ is contained in $\mathbf I$, and that the resulting corestricted map 
$d:\bm{\Omega}^1\to\mathbf I$ is graded. Recall also that $\owedge:\Lambda^2(\bm{\Omega}^1)\to\bm{\Omega}^2$ corestricts to a 
graded map $\owedge:\Lambda^2(\bm{\Omega}^1)\to\mathbf I$. All this implies that the map $\mu:T(\bm{\Omega}^1)\to
T(\bm{\Omega}^1)\otimes\bm{\Omega}^2\otimes T(\bm{\Omega}^1)$ from \S\ref{alagatapxd} corestricts to a map $T(\bm{\Omega}^1)\to
T(\bm{\Omega}^1)\otimes\mathbf I\otimes T(\bm{\Omega}^1)$. This implies that the bialgebra $\mathbf C:=\mathrm{Ker}(T(\bm{\Omega}^1)\to
T(\bm{\Omega}^1)\otimes\mathbf I\otimes T(\bm{\Omega}^1))$ is graded. Since $\bm{\Omega}^1$ has finite dimensional graded parts, so does 
$\mathbf C$.  

It follows that the Lie coalgebra $\mathfrak C:=\mathrm{Coprim}(\mathbf C)$ is also graded with finite dimensional graded parts.  
We denote by $\mathfrak G$ the graded dual Lie algebra. 

\subsection{Computation of the Lie algebra  $\mathfrak G$} 
\label{subsect:5:5}

Denote by $V^*$ the graded dual of a graded vector space $V$; so for $V=\oplus_{d\in\mathbb Z}V[d]$, 
$V^*=\oplus_{d\in\mathbb Z}V^*[d]$, where $V^*[d]:=V[-d]^*$. For $V=\oplus_{d\in\mathbb Z}V[d]$, 
$W=\oplus_{d\in\mathbb Z}W[d]$ two such spaces, we denote by $V\hat\otimes W$ their completed
tensor product, equal to $\prod_{d,d'}V[d]\otimes W[d']$. 

Denote also by $\mathfrak L$ the free Lie algebra functor and by $\mathfrak L_k$ is degree $k$ component.  

Then $\mathfrak G$ may be presented as the quotient 
$$
\mathfrak G=\mathfrak L((\bm{\Omega}^1)^*)/(\mathbf R), 
$$
where 
$$
\mathbf R:=\mathrm{im}(\mathbf I^*\stackrel{d^*\oplus{1\over2}\owedge^*}{\to}(\bm{\Omega}^1)^*\oplus(\Lambda^2(\bm{\Omega}^1))^*\simeq
\mathfrak L_1((\bm{\Omega}^1)^*)\oplus\mathfrak L_2((\bm{\Omega}^1)^*)) 
$$
and $(\mathbf R)$ is the ideal generated by $\mathbf R$. 

We now compute the space $\mathbf R$. Let 
$$
(X_i)_{i\in[n]}, \quad (Y_i)_{i\in[n]}, \quad(T_{ij}^\alpha)_{i<j\in[n],\alpha\geq 0}
$$
be the basis of $(\bm{\Omega}^1)^*$ dual to the basis (\ref{diff:E:n}) of $\bm{\Omega}^1$.  

\begin{lemma}
The space $\mathbf R$ decomposes as 
\begin{equation}\label{decomp:of:R}
\mathbf R=\mathbf R_{cc}+\mathbf R_{cp}+\mathbf R_{pp}, 
\end{equation}
where
\begin{equation}\label{def:R:cc}
\mathbf R_{cc}=\mathrm{Span}\{[X_i,X_j]|i<j\in[n]\}, 
\end{equation}
\begin{align}\label{def:R:cp}
\mathbf R_{cp}=\mathrm{Span}\{ & 
[X_i,Y_i]-\sum_{j\in[n]|j>i}T_{ij}^0-\sum_{j\in[n]|j<i}T_{ji}^0, \quad i\in[n] ; 
\nonumber \\& 
[X_i,Y_j]+T_{ij}^0, \quad i<j\in[n] ; \quad [X_i,Y_j]+T_{ji}^0, \quad j<i\in[n] ; \quad
\nonumber\\& 
[X_k,T_{ij}^\alpha], \quad i<j\in[n],\quad k\in[n]-\{i,j\},\quad\alpha\geq 0 ; 
\nonumber\\& 
[X_i,T_{ij}^\alpha]-T_{ij}^{\alpha+1}, \quad i<j\in[n],\quad\alpha\geq 0 ; 
\nonumber\\& 
[X_j,T_{ij}^\alpha]+T_{ij}^{\alpha+1}, \quad i<j\in[n],\quad\alpha\geq 0\},  
\end{align}
\begin{align}\label{def:R:pp}
\mathbf R_{pp}=\mathrm{Span}\{& \pi(i,j), \quad i<j\in[n], \nonumber\\& 
\sigma(i,j,k,l,\alpha,\beta), \quad i<j\in[n],\quad k<l\in[n], \quad i<k,\quad\#\{i,j,k,l\}=4,\quad\alpha,\beta\geq 0, \nonumber\\& 
\sigma''(i,j,k,\alpha,\beta),\quad i<j<k\in[n],\quad\alpha,\beta\geq 0, \nonumber\\& 
\sigma'''(i,j,k,\alpha,\beta),\quad i<j<k\in[n],\quad\alpha,\beta\geq 0, \nonumber\\& 
\kappa(k,i,j,\alpha),\quad i<j\in[n], \quad k\in[n]-\{i,j\},\quad \alpha\geq 0, \nonumber\\& 
\kappa'(i,j,\alpha),\quad i<j\in[n],\quad\alpha\geq 0\},  
\end{align}
where
\begin{equation}\label{pi:ij}
\pi(i,j):=[Y_i,Y_j] \quad \text{for}\quad i<j\in[n], 
\end{equation}
\begin{equation}\label{sigma:ijklab}
\sigma(i,j,k,l,\alpha,\beta):=[T_{ij}^\alpha,T_{kl}^\beta]\quad \text{for}\quad i<j\in[n], \ k<l\in[n],\ i<k, \ \#\{i,j,k,l\}=4,\ \alpha,\beta\geq 0, 
\end{equation}
\begin{equation}\label{sigma'':ijkab}
\sigma''(i,j,k,\alpha,\beta):=[T_{ij}^\alpha,T_{jk}^\beta]
+\sum_{\stackrel{\gamma,\delta\geq 0|}{\stackrel{\gamma+\delta=\alpha+\beta}{}}}
\begin{pmatrix} \alpha \\ \gamma\end{pmatrix}[T_{ij}^\gamma,T_{ik}^\delta] 
\quad\text{for}\quad i<j<k\in[n], \quad\alpha,\beta\geq 0,  
\end{equation}
\begin{equation}\label{sigma''':ijkab}
\sigma'''(i,j,k,\alpha,\beta):=[T_{ik}^\alpha,T_{jk}^\beta]
+\sum_{\stackrel{\gamma,\delta\geq 0|}{\stackrel{\gamma+\delta=\alpha+\beta}{}}}
(-1)^{\beta+1}\begin{pmatrix} \alpha \\ \delta\end{pmatrix}[T_{ij}^\gamma,T_{ik}^\delta] 
\quad\text{for}\quad i<j<k\in[n], \quad\alpha,\beta\geq 0,  
\end{equation}
\begin{equation}\label{kappa:ija} 
\kappa(k,i,j,\alpha)=\left\{
                \begin{array}{ll}
[Y_k,T_{ij}^\alpha]
+\sum_{\stackrel{\gamma,\delta\geq 0|}{\stackrel{\gamma+\delta=\alpha+\beta}{}}}
(-1)^\gamma[T_{ki}^\gamma,T_{kj}^\delta]\ \text{if}\  k<i, \\
{}[Y_k,T_{ij}^\alpha]+\sum_{\stackrel{\gamma,\delta\geq 0|}{\stackrel{\gamma+\delta=\alpha+\beta}{}}}
(-1)\cdot \begin{pmatrix} \alpha \\ \delta\end{pmatrix}[T_{ik}^\gamma,T_{ij}^\delta]\ \text{if}\  i<k<j, \\       
{}[Y_k,T_{ij}^\alpha]+\sum_{\stackrel{\gamma,\delta\geq 0|}{\stackrel{\gamma+\delta=\alpha+\beta}{}}}
\begin{pmatrix} \alpha \\ \gamma\end{pmatrix}[T_{ij}^\gamma,T_{ik}^\delta]\ \text{if}\  k>j, 
                \end{array}
              \right.
\quad\text{for}\quad i<j\in[n], \quad\alpha\geq 0, 
\end{equation}
\begin{align}\label{kappa':ija}
\kappa'(i,j,\alpha) & =[Y_i+Y_j,T_{ij}^\alpha]+\sum_{\stackrel{\stackrel{i<k<j}{\gamma+\delta=\alpha-1}}{}}
\begin{pmatrix} \alpha \\ \delta\end{pmatrix}[T_{ik}^\gamma,T_{ij}^\delta]
+\sum_{\stackrel{\stackrel{k>j}{\gamma+\delta=\alpha-1}}{}}
(-1)\cdot \begin{pmatrix} \alpha \\ \gamma\end{pmatrix}[T_{ij}^\gamma,T_{ik}^\delta]
+\sum_{\stackrel{\stackrel{k<i}{\gamma+\delta=\alpha-1}}{}}
(-1)^{\gamma+1}[T_{ki}^\gamma,T_{kj}^\delta] \nonumber \\ & 
\quad\text{for}\quad i<j\in[n], \quad\alpha\geq 0. 
\end{align}
\end{lemma}

\proof Set 
\begin{equation}\label{def:omega}
\omega:=\sum_{i\in[n]}X_i\otimes\underline{dc}_i+\sum_{i\in[n]}Y_i\otimes\underline{dp}_i
+\sum_{i<j\in[n],\alpha\geq 0}T_{ij}^\alpha\otimes\omega_{ij}^\alpha\in(\bm{\Omega}^1)^*\hat\otimes\bm{\Omega}^1; 
\end{equation}
this is the canonical element of $(\bm{\Omega}^1)^*\hat\otimes\bm{\Omega}^1$. For any Lie algebra $(\mathfrak g,[,]_{\mathfrak g})$, 
define the product $(\mathfrak g\otimes\bm{\Omega}^1)^2\to\mathfrak g\otimes\bm{\Omega}^2$ by 
\begin{equation}\label{product:Omega}
(x\otimes h)\cdot(x'\otimes h'):=[x,x']_{\mathfrak g}\otimes(h\owedge h').
\end{equation} 
Then $d\omega+{1\over2}\omega^2$ is an element of 
$\mathfrak L((\bm{\Omega}^1)^*)\hat\otimes\mathbf I$, and $\mathbf R$ is the image of the map $\mathbf I^*\to\mathfrak 
L((\bm{\Omega}^1)^*)$ induced by this element. So modding out by $(\mathbf R)$ corresponds to formally imposing the relation 
$d\omega+{1\over2}\omega^2=0$. 

Recall that $\bm{\Omega}^1$ decomposes as $\bm{\Omega}^1=\bm{\Omega}^1_c\oplus\bm{\Omega}^1_p$. This induces a decomposition of $(\bm{\Omega}^1)^*\hat\otimes\bm{\Omega}^1$. The corresponding decomposition of $\omega$ is $$ \omega=\omega_c+\omega_p, $$ where 
$$
\omega_c =\sum_{i\in[n]}X_i\otimes\underline{dc}_i, \quad 
\omega_p=\sum_{i\in[n]}Y_i\otimes\underline{dp}_i+\sum_{i<j\in[n],\alpha\geq 0}T_{ij}^\alpha\otimes\omega_{ij}^\alpha. 
$$
It follows from the construction of $\mathbf I$ that the map $\owedge:\Lambda^2(\bm{\Omega}^1)\to \mathbf I$ is compatible with the decompositions (\ref{decomp:Lambda2}) and (\ref{isom:I}) of both sides. Namely, 
$$
\owedge(\Lambda^2(\bm{\Omega}^1_c))=\Lambda^2(\bm{\Omega}^1_c), \quad 
\owedge(\bm{\Omega}^1_c\otimes\bm{\Omega}^1_p)=\bm{\Omega}^1_c\otimes\bm{\Omega}^1_p, \quad
\owedge(\Lambda^2(\bm{\Omega}^1_p))=\Sigma. 
$$
On the other hand, $d:\bm{\Omega}^1\to\mathbf I$ is such that 
$$
d(\bm{\Omega}^1_c)=0, \quad d(\bm{\Omega}^1_p)\subset\bm{\Omega}^1_c\otimes\bm{\Omega}^1_p. 
$$ 
Then $d\omega+{1\over2}\omega^2$ decomposes as 
$$
d\omega+{1\over2}\omega^2={1\over2}\omega_c^2+(d\omega_p+\omega_p\cdot\omega_c)+{1\over2}\omega_p^2
$$
according to the decomposition of $\mathfrak L((\bm{\Omega}^1)^*)\hat\otimes\mathbf I$ induced by (\ref{isom:I}).  
The decomposition (\ref{isom:I}) induces a decomposition of $\mathbf I^*$, and therefore of the map $\mathbf I^*\to\mathfrak L((\bm{\Omega}^1)^*)$ as the sum of three maps 
\begin{equation}\label{maps}
\Lambda^2(\bm{\Omega}^1_c)^*\to\mathfrak L((\bm{\Omega}^1)^*), \quad (\bm{\Omega}^1_c\otimes\bm{\Omega}^1_p)^*\to\mathfrak L((\bm{\Omega}^1)^*), \quad \Sigma^*\to\mathfrak L((\bm{\Omega}^1)^*). 
\end{equation}
These maps correspond respectively to ${1\over2}\omega_c^2$, $d\omega_p+\omega_c\cdot\omega_p$, and ${1\over2}\omega_p^2$. 
One then has the decomposition (\ref{decomp:of:R})
where $\mathbf R_{cc}$, $\mathbf R_{cp}$ and $\mathbf R_{pp},$ are the images of the maps induced by (\ref{maps}). 

One computes 
$$
\omega_c^2={1\over2}\sum_{i,j\in[n]}[X_i,X_j]\otimes(\underline{dc}_i\owedge\underline{dc}_j)
=\sum_{i<j\in[n]}[X_i,X_j]\otimes (\underline{dc}_i\owedge\underline{dc}_j) ; 
$$
since $(\underline{dc}_i\owedge\underline{dc}_j)_{i<j\in[n]}$ is a linearly independent family of $\Lambda^2(\bm{\Omega}^1_c)$, one obtains (\ref{def:R:cc}).

One computes 
\begin{align*}
& d\omega_p+{1\over2}\omega_p\cdot\omega_c\\ & 
=\sum_{i<j\in[n],\alpha\geq 0}T_{ij}^\alpha\otimes d(\omega_{ij}^\alpha)
+\sum_{i,j\in[n]}[X_i,Y_j]\otimes(\underline{dc}_i\owedge\underline{dp}_j)
+{1\over2}\sum_{k\in[n],i<j\in[n],\alpha\geq 0}[X_k,T_{ij}^\alpha]\otimes(\underline{dc}_k\owedge\omega_{ij}^\alpha) 
\\ & 
=\sum_{i\in[n]}([X_i,Y_i]-\sum_{j|j>i}T_{ij}^0-\sum_{j|j<i}T_{ji}^0)\otimes(\underline{dc}_i\owedge\underline{dp}_i)
\\ & 
+\sum_{i<j\in[n]}([X_i,Y_j]+T_{ij}^0)\otimes(\underline{dc}_i\owedge\underline{dp}_j)
+\sum_{i>j\in[n]}([X_i,Y_j]+T_{ji}^0)\otimes(\underline{dc}_i\owedge\underline{dp}_j)
\\ & 
+\sum_{i<j\in[n],k\neq i,j,\alpha\geq 0}[X_k,T_{ij}^\alpha]\otimes(\underline{dc}_k\owedge\omega_{ij}^\alpha)
\\ & 
+\sum_{i<j\in[n]}([X_i,T_{ij}^\alpha]-T_{ij}^{\alpha+1})\otimes(\underline{dc}_i\owedge\omega_{ij}^\alpha)
+\sum_{i<j\in[n]}([X_j,T_{ij}^\alpha]+T_{ij}^{\alpha+1})\otimes(\underline{dc}_j\owedge\omega_{ij}^\alpha). 
\end{align*}

As the second factors in the last expressions form a basis of $\bm{\Omega}_c^1\otimes\bm{\Omega}_p^1$, one gets (\ref{def:R:cp}).

One computes 
\begin{align*}
& {1\over2}\omega_p^2={1\over2}\sum_{i,j\in[n]}[Y_i,Y_j]\otimes(\underline{dp}_i\owedge\underline{dp}_j)
+\sum_{i<j\in[n],k\in[n],\alpha\geq 0}[Y_k,T_{ij}^\alpha]\otimes(\underline{dp}_k\owedge\omega_{ij}^\alpha)
\\ & +{1\over2}\sum_{i<j\in[n],k<l\in[n],\alpha,\beta\geq 0}[T_{ij}^\alpha,T_{kl}^\beta]\otimes
(\omega_{ij}^\alpha\owedge\omega_{kl}^\beta)
\end{align*}
which expresses as follows 
\begin{align}\label{expr:omega:p:2}
{1\over2}\omega_p^2= & \sum_{i<j\in[n]}\pi(i,j)\otimes P(i,j)
+\sum_{\stackrel{i<j\in[n],k<l\in[n],i<k,}{\stackrel{\#\{i,j,k,l\}=4,}{\stackrel{\alpha,\beta\geq0}{}}}}
\sigma(i,j,k,l,\alpha,\beta)\otimes S(i,j,k,l,\alpha,\beta)
\nonumber \\ & +\sum_{\stackrel{i<j<k\in[n],}{\stackrel{\alpha,\beta\geq 0}{}}}\sigma''(i,j,k,\alpha,\beta)\otimes S''(i,j,k,\alpha,\beta)
+\sum_{\stackrel{i<j<k\in[n],}{\stackrel{\alpha,\beta\geq 0}{}}}\sigma'''(i,j,k,\alpha,\beta)\otimes S'''(i,j,k,\alpha,\beta)
\nonumber \\ & 
+\sum_{\stackrel{i<j\in[n],}{\stackrel{k\neq i,j,}{\stackrel{\alpha\geq 0}{}}}}\kappa(k,i,j,\alpha)\otimes Q(k,i,j,\alpha)
+\sum_{\stackrel{i<j\in[n],}{\stackrel{\alpha\geq 0}{}}}\kappa'(i,j,\alpha)\otimes Q'(i,j,\alpha), 
\end{align}
where $\pi(i,j)$, $\sigma(i,j,k,l,\alpha,\beta)$, $\sigma''(i,j,k,\alpha,\beta)$, $\sigma'''(i,j,k,\alpha,\beta)$, $\kappa(k,i,j,\alpha)$, and $\kappa'(i,j,\alpha)$ 
are given by (\ref{pi:ij}), (\ref{sigma:ijklab}), (\ref{sigma'':ijkab}), (\ref{sigma''':ijkab}), (\ref{kappa:ija}), (\ref{kappa':ija}). 

As the second factors in (\ref{expr:omega:p:2}) form a basis of $\Sigma$, one gets (\ref{def:R:pp}). \hfill\qed\medskip

\subsection{An isomorphism $\mathfrak t_{1,n}^{\mathbb C}\simeq\mathfrak G$}

Recall that $\mathfrak G=\mathfrak L((\bm{\Omega}^1)^*)/(\mathbf R)$, where 
$\mathbf R=\mathbf R_{pp}+\mathbf R_{cp}+\mathbf R_{pp}$, and $\mathbf R_{pp}, \mathbf R_{cp},\mathbf R_{pp}$
are given by (\ref{def:R:cc}), (\ref{def:R:cp}), (\ref{def:R:pp}). 

Recall that $\mathfrak t_{1,n}$ is the Lie algebra with generators $x_i,y_i$ ($i\in[n]$), $t_{ij}$ ($i\neq j\in[n]$), 
and relations (\ref{rel:t:xx:yy}), (\ref{rel:t:xy:t}), (\ref{rel:t:xy:tt}), (\ref{rel:t:x:t:y:t}), (\ref{rel:t:xx:t:yy:t}).

\begin{lemma}\label{lemma:5:2}
There is a unique morphism of Lie algebras $\mathfrak t_{1,n}\to\mathfrak G$, given by 
\begin{equation}\label{assignment:t:G}
x_i\mapsto X_i\quad (i\in[n]), \quad y_i\mapsto Y_i\quad (i\in[n]), \quad t_{ij}\mapsto -T_{ij}^0\quad (i<j\in[n]), \quad 
t_{ij}\mapsto -T_{ji}^0\quad (i>j\in[n]). 
\end{equation}
\end{lemma}

\proof There is a unique morphism from the free Lie algebra with generators $x_i,y_i$ ($i\in[n]$), $t_{ij}$ ($i\neq j\in[n]$) to 
$\mathfrak G$, 
defined by (\ref{assignment:t:G}). 
By relation (\ref{def:R:cc}) of $\mathfrak G$, this morphism takes the first part of relation (\ref{rel:t:xx:yy}) of $\mathfrak t_{1,n}$ to 0. 
By the first line of relation (\ref{def:R:pp}) of $\mathfrak G$, it takes the second part of relation (\ref{rel:t:xx:yy}) of $\mathfrak t_{1,n}$ to 0. 
By the second line of relation (\ref{def:R:cp}) of $\mathfrak G$, it satisfies the first part of relation (\ref{rel:t:xy:t}) of $\mathfrak t_{1,n}$. 
The second part of relation (\ref{rel:t:xy:t}) of $\mathfrak t_{1,n}$ is satisfied by construction. By the first line of relation (\ref{def:R:cp}) 
of $\mathfrak G$, it satisfies relation (\ref{rel:t:xy:tt}) of $\mathfrak t_{1,n}$. By the third line of relation (\ref{def:R:cp}) 
of $\mathfrak G$ for $\alpha=0$, it satisfies the first part of relation (\ref{rel:t:x:t:y:t}) of $\mathfrak G$. By the fifth line of 
relation (\ref{def:R:pp}) of $\mathfrak G$ for $\alpha=0$, it satisfies the second part of relation (\ref{rel:t:x:t:y:t}) of $\mathfrak G$. 
By the sum of the two last lines of relation (\ref{def:R:cp}) of $\mathfrak G$ for $\alpha=0$, it satisfies the first part of relation 
(\ref{rel:t:xx:t:yy:t}) of $\mathfrak G$. By the last line of relation (\ref{def:R:pp}) of $\mathfrak G$ for $\alpha=0$, it satisfies the 
second part of relation (\ref{rel:t:xx:t:yy:t}) of $\mathfrak G$. 

All this implies that the above morphism factors through a morphism $\mathfrak t_{1,n}\to\mathfrak G$. \hfill\qed\medskip

\begin{lemma}\label{lemma:5:3}
There is a unique Lie algebra morphism $\mathfrak G\to\mathfrak t_{1,n}$, such that 
\begin{equation}\label{assignment:G:t}
X_i\mapsto x_i\quad (i\in[n]), \quad Y_i\mapsto y_i\quad (i\in[n]), \quad T_{ij}^\alpha\mapsto -(\mathrm{ad}x_i)^\alpha(t_{ij})
\quad (i<j\in[n],\alpha\geq 0).  
\end{equation}
\end{lemma}

\proof There is a unique morphism from the free Lie algebra with generators $X_i,Y_i$ ($i\in[n]$), $T_{ij}^\alpha$ ($i<j\in[n],\alpha\geq 0$) 
to $\mathfrak t_{1,n}$, defined by (\ref{assignment:G:t}). 
By the first part of relation (\ref{rel:t:xx:yy}) of $\mathfrak t_{1,n}$, this morphism takes relation (\ref{def:R:cc}) of $\mathfrak G$ to 0. 
By relation (\ref{rel:t:xy:tt}) of $\mathfrak t_{1,n}$, this morphism takes the first line of relation (\ref{def:R:cp}) of $\mathfrak G$ to 0. 
By relation (\ref{rel:t:xy:t}) of $\mathfrak t_{1,n}$, this morphism takes the second line of relation (\ref{def:R:cp}) of $\mathfrak G$ to 0.  
By combining relations $[x_k,x_i]=0$ and $[x_k,t_{ij}]=0$ ($i<j\in[n]$, $k\in[n]-\{i,j\}$) of $\mathfrak t_{1,n}$, we see that 
this morphism takes the third line of relation (\ref{def:R:cp}) of $\mathfrak G$ to 0. By the definitions of the images of 
$X_i$ and $T_{ij}^\alpha$, we see that this morphism takes the fourth line of relation (\ref{def:R:cp}) of $\mathfrak G$ to 0. 
Combining the definitions of the images of $X_j$ and $T_{ij}^\alpha$ with the relations $[x_i,x_j]=0$ and $[x_i+x_j,t_{ij}]=0$
of $\mathfrak t_{1,n}$, we see that this morphism takes the last line of relation (\ref{def:R:cp}) of $\mathfrak G$ to 0. 
By the second relations of (\ref{rel:t:xx:yy}) of $\mathfrak t_{1,n}$, this morphism takes the first line of relation (\ref{def:R:pp}) 
of $\mathfrak G$ to 0. 

Let $i,j,k,l,\alpha,\beta$ be as in the second line of relation (\ref{def:R:pp}) of $\mathfrak G$. It follows from relations 
$[x_i,t_{kl}]=[y_j,t_{kl}]=0$ and $[x_i,y_j]=t_{ij}$ in $\mathfrak{t}_{1,n}$ that $[t_{ij},t_{kl}]=0$ holds in $\mathfrak{t}_{1,n}$.  Moreover, 
relations $[x_i,t_{kl}]=0$, 
$[x_k,t_{ij}]=0$ and $[x_i,x_k]=0$ imply that relation $[(\mathrm{ad}x_i)^\alpha(t_{ij}),(\mathrm{ad}x_k)^\beta(t_{kl})]=0$ holds in $\mathfrak{t}_{1,n}$. This implies that the morphism takes the second line of relation (\ref{def:R:pp}) of $\mathfrak G$ to 0. 

Let $i,j,k,\alpha,\beta$ be as in the third line of relation (\ref{def:R:pp}) of $\mathfrak G$. The following equalities hold in 
$\mathfrak{t}_{1,n}$:  
\begin{align}\label{eq:ijkalphabeta}
&\nonumber [(\mathrm{ad}x_i)^\alpha(t_{ij}),(\mathrm{ad}x_j)^\beta(t_{jk})]
\\ & \nonumber 
= [(\mathrm{ad}x_i)^\alpha(t_{ij}),(-\mathrm{ad}x_k)^\beta(t_{jk})] \ (\text{by}\ [x_j+x_k,t_{jk}]=0\ \text{and}\ [x_j,x_k]=0)\\
&\nonumber =(\mathrm{ad}x_i)^\alpha(-\mathrm{ad}x_k)^\beta([t_{ij},t_{jk}]) \ (\text{by}\ [x_i,t_{jk}]=[x_k,t_{ij}]=0)\\ 
&\nonumber=-(\mathrm{ad}x_i)^\alpha(-\mathrm{ad}x_k)^\beta([t_{ij},t_{ik}]) \ (\text{by}\ [t_{ij},x_i+x_j]=[t_{ij},y_k]=0\ \text{and}\ [x_i,y_k]=t_{ik},\ 
[x_j,y_k]=t_{jk},\\ &  \nonumber\text{which imply}\ [t_{ij}+t_{ik},t_{jk}]=0)\\
&\nonumber=
-\sum_{\gamma=0}^\alpha\begin{pmatrix}\alpha\\ \gamma\end{pmatrix}[(\mathrm{ad}x_i)^\gamma(t_{ij}),(\mathrm{ad}x_i)^{\alpha-\gamma}
(-\mathrm{ad}x_k)^\beta(t_{ik})]\ (\text{by}\ [x_k,t_{ij}]=0)\\
&=
-\sum_{\gamma=0}^\alpha\begin{pmatrix}\alpha\\ \gamma\end{pmatrix}[(\mathrm{ad}x_i)^\gamma(t_{ij}),
(\mathrm{ad}x_i)^{\alpha+\beta-\gamma}(t_{ik})]\ (\text{by}\ [x_i+x_k,t_{ik}]=0\ \text{and}\ [x_i,x_k]=0). 
\end{align}
It follows that the morphism takes the third line of relation (\ref{def:R:pp}) of $\mathfrak G$ to 0. 

Let $i,j,k,\alpha,\beta$ be as in the fourth line of relation (\ref{def:R:pp}) of $\mathfrak G$. Taking into account that identity 
(\ref{eq:ijkalphabeta}) holds more generally under the assumption $\#\{i,j,k\}=3$, exchanging $j$ and $k$ in this identity, and 
replacing the mute index $\gamma$ by $\delta$, one gets 
$$
[(\mathrm{ad}x_i)^\alpha(t_{ik}),(\mathrm{ad}x_k)^\beta(t_{jk})]
=-\sum_{\delta=0}^\alpha\begin{pmatrix}\alpha\\ \delta\end{pmatrix}[(\mathrm{ad}x_i)^\delta(t_{ik}),
(\mathrm{ad}x_i)^{\alpha+\beta-\delta}(t_{ij})], 
$$
which using $[x_j,x_k]=0$ and $[x_j+x_k,t_{jk}]=0$ for rewriting the second factor of the first bracket, gives
$$
[(\mathrm{ad}x_i)^\alpha(t_{ik}),(\mathrm{ad}x_j)^\beta(t_{jk})]
=(-1)^\beta\sum_{\delta=0}^\alpha\begin{pmatrix}\alpha\\ \delta\end{pmatrix}[(\mathrm{ad}x_i)^{\alpha+\beta-\delta}(t_{ij}),
(\mathrm{ad}x_i)^\delta(t_{ik})].  
$$
It follows that the morphism takes the fourth line of relation (\ref{def:R:pp}) of $\mathfrak G$ to 0. 

Let $i,j,k,\alpha$ be as in the fifth line of relation (\ref{def:R:pp}) of $\mathfrak G$. Then 
\begin{align}\label{eq:ijkalpha}
& \nonumber [y_k,(\mathrm{ad}x_i)^\alpha(t_{ij})]=-\sum_{\stackrel{\alpha',\alpha''\geq 0}{\stackrel{\alpha'+\alpha''=\alpha-1}{}}}
(\mathrm{ad}x_i)^{\alpha'}([t_{ik},(\mathrm{ad}x_i)^{\alpha''}(t_{ij})])\ (\text{by}\ [y_k,t_{ij}]=0\ \text{and}\ [y_k,x_i]=-t_{ik})\\
& \nonumber=-\sum_{\stackrel{\alpha',\alpha''\geq 0}{\stackrel{\alpha'+\alpha''=\alpha-1}{}}}
(\mathrm{ad}x_i)^{\alpha'}([t_{ik},(-\mathrm{ad}x_j)^{\alpha''}(t_{ij})])\ (\text{by}\ [x_i,x_j]=[x_i+x_j,t_{ij}]=0)\\
& \nonumber=-\sum_{\stackrel{\alpha',\alpha'',\alpha'''\geq 0}{\stackrel{\alpha'+\alpha''+\alpha'''=\alpha-1}{}}}
\begin{pmatrix}\alpha'+\alpha''\\ \alpha'\end{pmatrix}
[(\mathrm{ad}x_i)^{\alpha'}(t_{ik}),(\mathrm{ad}x_i)^{\alpha''}(-\mathrm{ad}x_j)^{\alpha'''}(t_{ij})] \\ 
&\nonumber=-\sum_{\stackrel{\alpha',\alpha'',\alpha'''\geq 0}{\stackrel{\alpha'+\alpha''+\alpha'''=\alpha-1}{}}}
\begin{pmatrix}\alpha'+\alpha''\\ \alpha'\end{pmatrix}
[(\mathrm{ad}x_i)^{\alpha'}(t_{ik}),(\mathrm{ad}x_i)^{\alpha''+\alpha'''}(t_{ij})]\ (\text{by}\ [x_i,x_j]=[x_i+x_j,t_{ij}]=0)\\
&\nonumber=-\sum_{\stackrel{\gamma,\delta\geq 0}{\stackrel{\gamma+\delta=\alpha-1}{}}}
(\sum_{\alpha''=0}^\delta\begin{pmatrix}\gamma+\alpha''\\ \gamma\end{pmatrix})
[(\mathrm{ad}x_i)^\gamma(t_{ik}),(\mathrm{ad}x_i)^\gamma(t_{ij})]\ (\text{replacing}\ \alpha',\alpha'+\alpha''\ \text{by}\ \gamma,\delta)\\
&=-\sum_{\stackrel{\gamma,\delta\geq 0}{\stackrel{\gamma+\delta=\alpha-1}{}}}
\begin{pmatrix}\alpha\\ \delta\end{pmatrix}
[(\mathrm{ad}x_i)^\gamma(t_{ik}),(\mathrm{ad}x_i)^\delta(t_{ij})]\ (\text{using}\ 
\sum_{\alpha''=0}^\delta\begin{pmatrix}\gamma+\alpha''\\ \gamma\end{pmatrix}=\begin{pmatrix}\gamma+\delta+1\\ \gamma+1\end{pmatrix}
=\begin{pmatrix}\alpha\\ \delta\end{pmatrix})
\end{align}
It follows that the morphism takes the fifth line of relation (\ref{def:R:pp}) of $\mathfrak G$ to 0 when $i<k<j$. Exchanging the roles of 
$\gamma$ and $\delta$, equality (\ref{eq:ijkalpha}) can be rewritten as follows 
$$
[y_k,(\mathrm{ad}x_i)^\alpha(t_{ij})]=\sum_{\stackrel{\gamma,\delta\geq 0}{\stackrel{\gamma+\delta=\alpha-1}{}}}
\begin{pmatrix}\alpha\\ \gamma\end{pmatrix}[(\mathrm{ad}x_i)^\gamma(t_{ij}),(\mathrm{ad}x_i)^\delta(t_{ik})],  
$$
which implies that the morphism takes the fifth line of relation (\ref{def:R:pp}) of $\mathfrak G$ to 0 when $k>j$.

One also has 
\begin{align*}
& [y_k,(\mathrm{ad}x_i)^\alpha(t_{ij})]=-\sum_{\stackrel{\alpha',\alpha''\geq 0}{\stackrel{\alpha'+\alpha''=\alpha-1}{}}}
(\mathrm{ad}x_i)^{\alpha'}([t_{ik},(-\mathrm{ad}x_j)^{\alpha''}(t_{ij})])\ (\text{by the beginning of (\ref{eq:ijkalpha})})\\ 
& =-\sum_{\stackrel{\alpha',\alpha''\geq 0}{\stackrel{\alpha'+\alpha''=\alpha-1}{}}}
(\mathrm{ad}x_i)^{\alpha'}(-\mathrm{ad}x_j)^{\alpha''}([t_{ik},t_{ij}])\ (\text{using}\ [x_j,t_{ik}]=0)\\ 
& =\sum_{\stackrel{\alpha',\alpha''\geq 0}{\stackrel{\alpha'+\alpha''=\alpha-1}{}}}
(\mathrm{ad}x_i)^{\alpha'}(-\mathrm{ad}x_j)^{\alpha''}([t_{ki},t_{kj}])\ (\text{using}\ t_{ik}=t_{ki}\ \text{and}\ [t_{ik},t_{ij}+t_{jk}]=0)\\ 
& =\sum_{\stackrel{\alpha',\alpha''\geq 0}{\stackrel{\alpha'+\alpha''=\alpha-1}{}}}
[(\mathrm{ad}x_i)^{\alpha'}(t_{ki}),(-\mathrm{ad}x_j)^{\alpha''}(t_{kj})]\ (\text{using}\ [x_i,t_{kj}]=[x_j,t_{ki}]=0)\\ 
& =\sum_{\stackrel{\gamma,\delta\geq 0}{\stackrel{\gamma+\delta=\alpha-1}{}}}
[(-\mathrm{ad}x_k)^\gamma(t_{ki}),(\mathrm{ad}x_k)^\delta(t_{kj})]\ (\text{replacing}\ \alpha',\alpha''\ \text{by}\ \gamma,\delta\ \text{and 
using }[x_i,x_k]=[x_i+x_k,t_{ik}]=0), 
\end{align*}
which implies that the morphism takes the fifth line of relation (\ref{def:R:pp}) of $\mathfrak G$ to 0 when $k<j$. All this implies
that the morphism takes the fifth line of relation (\ref{def:R:pp}) of $\mathfrak G$ to 0 in all cases. 

Let $i,j,\alpha$ be as in the last line of relation (\ref{def:R:pp}) of $\mathfrak G$. One has
\begin{align}\label{eq:Term:partial}
& \nonumber [y_i+y_j,(\mathrm{ad}x_i)^\alpha(t_{ij})]=\sum_{\alpha'+\alpha''=\alpha-1}(\mathrm{ad}x_i)^{\alpha'}([-\sum_{k\neq i,j}t_{ik},
(\mathrm{ad}x_i)^{\alpha''-1}(t_{ij})]) \\ 
&=\sum_{k\neq i,j}\sum_{\alpha'+\alpha''=\alpha-1}-(\mathrm{ad}x_i)^{\alpha'}(-\mathrm{ad}x_j)^{\alpha''}([t_{ik},t_{ij}])
=\sum_{k\neq i,j}\mathrm{Term}_k, 
\end{align}
where the second equality relies on $[y_i+y_j,x_i]=-\sum_{k\neq i,j}t_{ik}$ and where $\mathrm{Term}_k$ is the summand 
corresponding to $k$ in the last expression. 
One has 
\begin{align}\label{eq:Termk:1}
&\nonumber \mathrm{Term}_k=
\sum_{\alpha'+\alpha''=\alpha-1}(\mathrm{ad}x_i)^{\alpha'}(-\mathrm{ad}x_j)^{\alpha''}([t_{ki},t_{kj}])
=\sum_{\gamma+\delta=\alpha-1}[(\mathrm{ad}x_i)^\gamma(t_{ki}),(-\mathrm{ad}x_j)^\delta(t_{kj})]
\\&=\sum_{\gamma+\delta=\alpha-1}(-1)^\gamma[(\mathrm{ad}x_k)^\gamma(t_{ki}),(\mathrm{ad}x_k)^\delta(t_{kj})],  
\end{align}
where the first equality relies on $[t_{ki},t_{ij}+t_{kj}]=0$ and the second equality is obtained by replacing $\alpha',\alpha''$ by 
$\gamma,\delta$ and by using $[x_j,t_{ki}]=[x_i,t_{kj}]=0$. 

One also has 
\begin{align}\label{eq:Termk:2}
&\nonumber \mathrm{Term}_k=
\sum_{\stackrel{\alpha',\alpha'',\alpha'''\geq 0}{\stackrel{\alpha'+\alpha''+\alpha'''=\alpha-1}{}}}
-\begin{pmatrix}\alpha'+\alpha''\\ \alpha'\end{pmatrix}
[(\mathrm{ad}x_i)^{\alpha'}(t_{ik}),(\mathrm{ad}x_i)^{\alpha''}(-\mathrm{ad}x_j)^{\alpha'''}(t_{ij})]\ (\text{using}\ [x_j,t_{ik}]=0)
\\ 
&\nonumber
=\sum_{\stackrel{\gamma,\delta\geq 0}{\stackrel{\gamma+\delta=\alpha-1}{}}}
\sum_{\alpha''=0}^\delta -\begin{pmatrix}\gamma+\alpha''\\ \gamma\end{pmatrix}
[(\mathrm{ad}x_i)^\gamma(t_{ik}),(\mathrm{ad}x_i)^\delta(t_{ij})]\ (\text{renaming}\ \alpha',\alpha''+\alpha'''\text{ as }
\gamma,\delta)
\\ &
=\sum_{\stackrel{\gamma,\delta\geq 0}{\stackrel{\gamma+\delta=\alpha-1}{}}}
-\begin{pmatrix}\alpha\\ \delta\end{pmatrix}
[(\mathrm{ad}x_i)^\gamma(t_{ik}),(\mathrm{ad}x_i)^\delta(t_{ij})]
\ (\text{using}\ \sum_{\alpha''=0}^\delta\begin{pmatrix}\gamma+\alpha''\\ \gamma\end{pmatrix}
=\begin{pmatrix}\gamma+\delta+1\\ \gamma+1\end{pmatrix}=\begin{pmatrix}\alpha\\ \delta\end{pmatrix}),
\end{align}
which implies 
\begin{equation}\label{eq:Termk:3}
\mathrm{Term}_k=\sum_{\stackrel{\gamma,\delta\geq 0}{\stackrel{\gamma+\delta=\alpha-1}{}}}
\begin{pmatrix}\alpha\\ \gamma\end{pmatrix}[(\mathrm{ad}x_i)^\gamma(t_{ij}),(\mathrm{ad}x_i)^\delta(t_{ik})]. 
\end{equation}
Substituting in (\ref{eq:Term:partial}) the identities (\ref{eq:Termk:1}) when $k<i$, (\ref{eq:Termk:2}) when $i<k<j$ and (\ref{eq:Termk:3}) 
when $k>j$, one sees that the morphism takes the last line of line of relation (\ref{def:R:pp}) of $\mathfrak G$ to 0. 
All this proves Lemma \ref{lemma:5:3}. \hfill\qed\medskip
 
Combining Lemmas \ref{lemma:5:2} and \ref{lemma:5:3}, one obtains:  

\begin{prop}\label{prop:iso:G:t}
Formula (\ref{assignment:G:t}) gives rise to an isomorphism of Lie algebras $\mathfrak G\to\mathfrak t_{1,n}^{\mathbb C}$. 
\end{prop}

\section{Elements of a description of $\mathrm{VBFC}(X,D)_{unip}$ (equiv. ($\mathrm e$))}\label{sect:description}

\subsection{Reduction of a space of forms to $\Sigma_{log}$}

Let $\g$ be a finite dimensional nilpotent Lie algebra over $\mathbb C$, $G$ the corresponding group. 
Let $X$ be a smooth complex algebraic variety with a divisor $D$. 

Recall also that a 1-form $\alpha$ on $X- D$ is called logarithmic at $D$ if both $\alpha$ and $d\alpha$ have simple poles at $D$
(see \S\ref{sect:ldffrncd}). 
This is equivalent to saying that if $D$ is locally defined by the equation $z=0$ near its generic point then $\alpha=f\frac{dz}{z}+\beta$, where $f$ is a regular function and $\beta$ a regular 1-form. 

Let $\omega\in \Omega^1(X- D,\g)$ be a 1-form satisfying the Maurer-Cartan equation 
$$
d\omega+{1\over 2}\omega^2=0 
$$
where $\omega^2$ is defined using (\ref{product:Omega}).
In other words, $d+\omega$ is a flat connection on the trivial $G$-bundle on $X- D$.

\begin{lemma}\label{l1} If $\omega$ has a first order pole at $D$ then it is logarithmic. 
\end{lemma} 

\begin{proof} Let $z,x_1,...,x_n$ be local coordinates such that $D$ is locally defined by $z=0$ near a smooth point. 
Let $\omega=f\frac{dz}{z}+\sum_i \frac{g_i}{z}dx_i$, where $f,g_i$ are regular functions. Our job is to show that 
$g_i$ vanish at $z=0$. 

Let $\psi=z\omega$. Then $\psi=fdz+\sum_i g_idx_i$, a regular 1-form. We have $d\psi=dz\wedge \omega+zd\omega$. Thus from the Maurer-Cartan equation we have that 
$$
zd\psi=dz\wedge \psi-{1\over 2}\psi^2.
$$
Thus, $dz\wedge \psi-{1\over 2}\psi^2$ vanishes coefficientwise at $z=0$. In particular, for the coefficient of $dz\wedge dx_i$ we get 
that $g_i-[f,g_i]$ vanishes at $z=0$. In other words, if $g_{i0},f_0$ are restrictions of $g_i,f$ to $z=0$, then 
$$
g_{i0}=[f_0,g_{i0}].
$$
But $f_0\in \g$, which is a nilpotent Lie algebra. Thus, $g_{i0}=0$, as desired. 
\end{proof} 

Now assume that $D=\cup_{j=1}^N D_j$ is a union of smooth irreducible divisors, intersecting pairwise transversally. 

\begin{lemma}\label{l2}  Let $\omega\in\Omega^1_{log}(X-D)$. Then for any $1\le m\le N$, the residue of $\omega$ at $D_m$ 
(originally defined generically on $D_m$) extends to a regular function on the whole $D_m$. 
\end{lemma}

\begin{proof} For each $m$, let $f_{m0}$ be the residue of $\omega$ at $D_m$. This is a regular function on 
\linebreak $D_m- \cup_{j:j\ne m}D_j$. 

Let $p\in D_m$ be a generic point and $z,x_1,...,x_n$ be local coordinates near $p$ such that $D_m$ is locally defined by $z=0$. 
Using that $\omega\in\Omega^1_{log}(X-D)$, on a small ball $B_p$ around $p$ one may write $\omega$ as $f\frac{dz}{z}+\beta$, 
where $\beta$ is regular. Let $\beta_0$ be the restriction of $\beta$ 
to $B_p\cap D_m$. Then $\beta_0$ is a regular 1-form. Moreover, it is easy to check that $\beta_0$ is canonically attached to $\omega$, 
i.e., it does not depend on the choice of coordinates and the representation $\omega=f\frac{dz}{z}+\beta$. Therefore, $\beta_0$ is defined 
globally on $D_m- \cup_{i\ne m}D_i$. 

Moreover, the Maurer-Cartan equation for 
$\omega$ implies the Maurer-Cartan equation for $\beta_0$, i.e., $d+\beta_0$ is a flat connection on the trivial bundle over 
$D_m- \cup_{i\ne m}D_i$. Namely, $d+\beta_0$ is the "restriction" of the flat connection 
$d+\omega$ to the pole divisor $D_m$. Concretely, if $D_m^\varepsilon$ is a perturbation of $D_m$ 
(defined locally near some point) then $d+\beta_0$ is the limit of the restriction of $d+\omega$ 
to $D_m^\varepsilon$ as $\varepsilon\to 0$. 

Since the connection $d+\omega$ has first order poles, it follows from the work of Deligne (\cite{Del}) that so does the 
connection $d+\beta_0$. Indeed, pick $q\in D_m\cap D_i$ for some $i$, and let us work on a small ball $B_q$ around $q$. 
Let $C$ be a generic smooth curve in $D_m\cap B_q$ passing through $q$. We need to show that 
the restriction of $d+\beta_0$ to $C$ has a first order pole at $q$. To this end,  
consider a generic perturbation $C_\varepsilon$ of $C$ in $X$, still passing through $q$ (i.e., $C_\varepsilon$ is not contained in $D_m$). 
Then the restriction of $d+\omega$ to $C_\varepsilon$ has a pole only at $q$ inside $B_q$ (because of pairwise transversality of $D_j$). Moreover, this pole is simple, 
since by \cite{Del}, flat sections of $d+\omega|_{C_\varepsilon}$ have logarithmic growth near $q$ (as $\omega$ has first order poles).
Hence, taking the limit $\varepsilon\to 0$, we find that $d+\beta_0|_C$ has a simple pole at $q$, as claimed.  
  
Also, $f_{m0}$ is a flat section of the adjoint bundle for the trivial 
$G$-bundle with the connection $d+\beta_0$, i.e., 
$$
df_{m0}+[\beta_0,f_{m0}]=0.
$$
Since $G$ is unipotent and $\beta_0$ has simple poles, this means that $f_{m0}$ has logarithmic growth when approaching $D_m\cap(\cup_{i|i\neq n} D_i)$. 
Hence $f_{m0}$ cannot have a pole at $D_m\cap(\cup_{i|i\neq n} D_i)$, i.e., it extends to a regular function on $D_m$, as desired. 
\end{proof} 

\remark{One can also prove that $\omega$ can locally be written as $\sum_j \omega_j$, 
where $\omega_j$ is a logarithmic form with pole only on $D_j$. Indeed, according to Lemma \ref{l2}, 
locally near each point $p\in D_j$ we may extend $f_{j0}$ to a regular function $f_j$ on a neighborhood of $p$. 
Also near $p$ the divisor $D_j$ may be defined by the equation $z_j=0$. Then since $\omega\in\Omega^1_{log}(X-D)$, 
$\omega-\sum_{j: p\in D_j} f_j\frac{dz_j}{z_j}$ is regular near $p$, which implies the required statement.  
\hfill \qed\medskip} 

\subsection{Equality $\bm{\Omega}^1=\Sigma_{log}$}\label{sect:ehs}

Recall that $\bm{\Omega}^1$ is the subspace 
$$
\sum_{i\in I}\Gamma(X,\Omega^1_X(\mathrm{log}D_i))
$$ 
of $\Gamma_{rat}(X,\Omega^1_X(D))$. 

On the other hand, $\Sigma_{log}$ is the subspace 
$$
\{\alpha\in\Gamma_{rat}(X,\Omega^1_X(D))|\alpha\in \Gamma(X,\Omega^1_X(\mathrm{log}D))\text{ and }\forall i\in I,\quad 
\mathrm{res}_{D_i}(\alpha)\in \Gamma(D_i,\mathcal O_{D_i})\}
$$
of $\Gamma_{rat}(X,\Omega^1_X(D))$. 

\begin{lemma}
One has $\bm{\Omega}^1=\Sigma_{log}$. 
\end{lemma}

\proof For each $i\in I$, one has the inclusion $\Gamma(X,\Omega^1_X(\mathrm{log}D_i))\subset \Gamma(X,\Omega^1_X(\mathrm{log}D))$ and 
the map $\mathrm{res}_{D_i}:\Gamma(X,\Omega^1_X(\mathrm{log}D))\to \Gamma(D_i-(\cup_{j|j\neq i}D_j),\mathcal O_{D_i})$ maps 
$\Gamma(X,\Omega^1_X(\mathrm{log}D_j))$ to 0 if $j\neq i$, and to $\Gamma(D_i,\mathcal O_{D_i})$ if $j=i$. All this implies that
$$
\bm{\Omega}^1\subset\Sigma_{log}. 
$$ 
It follows from the definition of $\Sigma_{log}$ that this space fits in an exact sequence 
$$
0\to \Gamma(X,\Omega^1_X)\to\Sigma_{log}\stackrel{\oplus_{i\in I}\mathrm{res}_{D_i}}{\to}\oplus_{i\in I}\Gamma(D_i,\mathcal O_{D_i}). 
$$
In \S\ref{sec:coh}, it is proved that for each $i\in I$, the map 
$\mathrm{res}_{D_i}:\Gamma(X,\Omega^1_X(\mathrm{log}D_i))\to \Gamma(D_i,\mathcal O_{D_i})$
is surjective (Lemma \ref{tlminis}). 
Together with the fact that the restriction of $\mathrm{res}_{D_i}$ to $\Gamma(D_j,\mathcal O_{D_j})$
is 0 if $i\neq j$, this implies that the composed map 
$$
\bm{\Omega}^1\hookrightarrow\Sigma_{log}\stackrel{\oplus_{i\in I}\mathrm{res}_{D_i}}{\to}
\oplus_{i\in I}\Gamma(D_i,\mathcal O_{D_i})
$$
is surjective. 

On the other hand, we also have an inclusion $\Gamma(X,\Omega^1_X)\subset\bm{\Omega}^1$. We have therefore a 
commutative diagram 
$$
\xymatrix{ & \Sigma_{log}\ar^{\oplus_{i\in I}\mathrm{res}_{D_i}}[rd]&
\\ \Gamma(X,\Omega^1_X)\ar@{^{(}->}[ru]\ar@{^{(}->}[rd]  & & \oplus_{i\in I}\Gamma(D_i,\mathcal O_{D_i})\\ & \bm{\Omega}^1
\ar@{^{(}->}[uu]\ar^a[ur]& }
$$
where the map $a$ is surjective. It follows that the map $\oplus_{i\in I}\mathrm{res}_{D_i}$ is surjective as well and that the map 
$\bm{\Omega}^1\hookrightarrow\Sigma_{log}$ is an isomorphism. This proves the result. \hfill\qed\medskip

\section{Relation with the universal KZB connection}\label{sect:rwtukc}

Let us fix $\tau\in\mathfrak H$. In \S\ref{sect:rwtuzc}, we attach to $\tau$ a principal bundle with flat connection over $\mathbb C^n$  
$$
(\text{trivial }\mathrm{exp}(\hat{\mathfrak t}_{1,n}^{\mathbb C})\text{-bundle},
d+A_{\mathrm{KZB}}), 
$$ 
a principal bundle with flat connection over $(E_\tau^\#)^n$, 
$$
(\text{trivial }\mathrm{exp}(\mathfrak G)\text{-bundle},
d+\omega), 
$$
and maps $\mathbb C^{2n}\to\mathbb C^n$, $\mathbb C^{2n}\to(E_\tau^\#)^n$.

We will construct an isomorphism between the lifts to $\mathbb C^{2n}$ of these two pairs of bundles with flat connection, 
thereby proving Theorem \ref{tisaibtfpbwfcoC}. 

\subsection{A flat connection on $(E^\#_\tau)^n$}

Recall from (\ref{def:omega}) the element $\omega\in\Gamma_{rat}((E^\#_\tau)^n,\Omega^1_{(E^\#_\tau)^n})\hat\otimes\mathfrak G$. 
It follows from Lemma \ref{lemma:diafcottpebov} that $d+\omega$
is a flat connection on the trivial bundle over $(E^\#_\tau)^n$ with group $\mathrm{exp}(\mathfrak G)$. 
It follows from (\ref{def:omega}) and (\ref{form:omega:ijalpha}) that the lift to $\mathbb C^{2n}$ of $\omega$ is
$$
\tilde\omega=\sum_{i\in[n]}X_i\cdot dc_i+\sum_{i\in[n]}Y_i\cdot dp_i
+\sum_{i<j\in[n],\alpha\geq 0}T_{ij}^\alpha[\Big({\theta(p_{ij}+z|\tau)
\over\theta(z|\tau)\theta(p_{ij}|\tau)}-{1\over z}\Big)e^{-c_{ij}z}dp_{ij}|z^\alpha]
$$
where $(p_1,c_1,\ldots,p_n,c_n)$ are the standard coordinates on $\mathbb C^n$. 

\subsection{The universal KZB system}

The corresponding configuration space for $E_\tau$ is 
$$
C(E_\tau,n):=\{(p_1,\ldots,p_n)\in E_\tau|p_i\neq p_j\text{ for }i<j\in[n]\}. 
$$
In \cite{CEE}, we defined a pair $(\mathcal P_{\mathrm{KZB}},\nabla_{\mathrm{KZB}})$ of a principal $\mathrm{exp}(\hat{\mathfrak t}_{1,n}^{\mathbb C})$-bundle $\mathcal P_{\mathrm{KZB}}$ over $C(E_\tau,n)$ and of a flat connection $\nabla_{\mathrm{KZB}}$ over it.  
The projection $\mathbb C\to E_\tau$ gives rise to a fibered product 
$$
\tilde C(E_\tau,n):=C(E_\tau,n)\times_{(E_\tau)^n}\mathbb C^n.
$$
Then 
$$
\tilde C(E_\tau,n)=\{(p_1,\ldots,p_n)\in\mathbb C^n|p_i-p_j\notin\mathbb Z+\tau\mathbb Z\text{ for }i<j\in[n]\}. 
$$
There is a natural projection $\tilde C(E_\tau,n)\stackrel{p}{\to}C(E_\tau,n)$ with covering group $\mathbb Z^{2n}$. There is a 
natural isomorphism of $p^*\mathcal P_{\mathrm{KZB}}$ with the trivial principal bundle 
over $C(E_\tau,n)$ with group $\mathrm{exp}(\hat{\mathfrak t}_{1,n}^{\mathbb C})$. The pull-back $p^*\nabla_{\mathrm{KZB}}$ of the 
KZB connection over $\tilde C(E_\tau,n)$ is then the operator 
$$
d+A_{\mathrm{KZB}},
$$ 
where $A_{\mathrm{KZB}}\in \Gamma(\tilde C(E_\tau,n),
\Omega^1_{\tilde C(E_\tau,n)}\otimes\hat{\mathfrak t}_{1,n}^{\mathbb C})$ is given by 
$$
A_{\mathrm{KZB}}:=-\sum_{i\in[n]}\Big(-y_i+\sum_{j|j\in[n],j\neq i}
\big({\theta(p_{ij}+\mathrm{ad}x_i|\tau)\mathrm{ad}x_i\over\theta(p_{ij}|\tau)\theta(\mathrm{ad}x_i|\tau)}-1\big)(y_j)\Big) dp_i,  
$$
where $p_{ij}:=p_i-p_j$; the expression under the second sum sign should be computed as $\sum_{\alpha\geq 0}F_\alpha(p_{ij})
(\mathrm{ad}x_i)^\alpha(y_j)dp_i$, where the function 
$(p,x)\mapsto{{\theta(p+x|\tau)x}\over{\theta(p|\tau)\theta(x|\tau)}}-1$ is viewed as formal in $x$ and meromorphic in $p$, 
expanding as $\sum_{\alpha\geq 0}F_\alpha(p)x^\alpha$. Note also that the function $F_0(p)$ in this expansion is $0$. 
Therefore 
$$
A_{\mathrm{KZB}}=-\sum_{i\in[n]}\Big(-y_i+\sum_{j|j\in[n],j\neq i}
\big({\theta(p_{ij}+\mathrm{ad}x_i|\tau)\over\theta(p_{ij}|\tau)\theta(\mathrm{ad}x_i|\tau)}-{1\over\mathrm{ad}x_i}\big)(t_{ij})\Big) dp_i,  
$$ 
with the same conventions as above, based on the fact that $(p,x)\mapsto{{\theta(p+x|\tau)}\over{\theta(p|\tau)\theta(x|\tau)}}-{1\over x}
=:g(p,x)$ may be viewed as formal in $x$ and meromorphic in $p$. As $g(p,x)=g(-p,-x)$, one has
$$
\text{if }i\neq j\in[n]\text{, then }
\big({\theta(p_{ij}+\mathrm{ad}x_i|\tau)\over\theta(p_{ij}|\tau)\theta(\mathrm{ad}x_i|\tau)}-{1\over\mathrm{ad}x_i}\big)(t_{ij})
+\big({\theta(p_{ji}+\mathrm{ad}x_j|\tau)\over\theta(p_{ji}|\tau)\theta(\mathrm{ad}x_j|\tau)}-{1\over\mathrm{ad}x_j}\big)(t_{ij})=0,
$$
so 
$$
A_{\mathrm{KZB}}:=\sum_{i\in[n]}y_i dp_i
-{1\over 2}\sum_{i\neq j\in[n]}\big({\theta(p_{ij}+\mathrm{ad}x_i|\tau)\over\theta(p_{ij}|\tau)\theta(\mathrm{ad}x_i|\tau)}-
{1\over\mathrm{ad}x_i}\big)(t_{ij})\cdot dp_{ij}.   
$$

\subsection{Relation between the two systems}

The image of $\tilde\omega$ under the isomorphism $\mathfrak{t}_{1,n}\simeq\mathfrak G$ is 
$$
\mathrm{im}(\tilde\omega)=
\sum_{i\in[n]}x_i\cdot dc_i+\sum_{i\in[n]}y_i\cdot dp_i
-\sum_{i<j\in[n],\alpha\geq 0}
e^{-c_{ij}\mathrm{ad}x_i}
\Big({\theta(p_{ij}+\mathrm{ad}x_i|\tau)
\over\theta(\mathrm{ad}x_i|\tau)\theta(p_{ij}|\tau)}-{1\over \mathrm{ad}x_i}\Big)(t_{ij})
dp_{ij}
$$
The expression $(p_1,\ldots,c_n)\mapsto e^{\sum_{i\in[n]} c_i x_i}$ defines a holomorphic map $\mathbb C^{2n}\to
\mathrm{exp}(\mathfrak G)$. One may therefore conjugate $d+\mathrm{im}(\tilde\omega)$ by this map. 
One has 
$$
e^{\sum_{i\in[n]} c_i x_i}de^{-\sum_{i\in[n]} c_i x_i}=d-\sum_{i\in[n]}x_i\cdot dc_i, 
$$
and 
$$
e^{\sum_{i\in[n]} c_i x_i}t_{ij}e^{-\sum_{i\in[n]} c_i x_i}=e^{c_{ij}\mathrm{ad}x_i}(t_{ij}),  
$$
moreover $e^{\sum_{i\in[n]} c_i x_i}$ commutes with all the $x_k$, $k\in[n]$, so  
\begin{align*}
e^{\sum_{i\in[n]} c_i x_i}(d+\mathrm{im}(\tilde\omega))e^{-\sum_{i\in[n]} c_i x_i}
& =d+\sum_{i\in[n]}x_i\cdot dc_i
-\sum_{i<j\in[n],\alpha\geq 0}
\Big({\theta(p_{ij}+\mathrm{ad}x_i|\tau)
\over\theta(\mathrm{ad}x_i|\tau)\theta(p_{ij}|\tau)}-{1\over \mathrm{ad}x_i}\Big)(t_{ij})
\\ & =d+A_{\mathrm{KZB}}. 
\end{align*}

We have proved: 
\begin{thm}
The map $\mathbb C^{2n}\to\mathrm{exp}(\mathfrak G)$, $(p_1,\ldots,c_n)\mapsto e^{\sum_{i\in[n]}c_ix_i}$ 
sets up an isomorphism between the following principal bundles with flat connections over $\mathbb C^{2n}$: 
\begin{itemize}
\item the pull-back under $\mathbb C^{2n}\to\mathbb C^{n}\to(E_\tau)^n$ of $(\mathcal P_{\mathrm{KZB}},\nabla_{\mathrm{KZB}})$; 
\item the pull-back under $\mathbb C^{2n}\to(E_\tau^\#)^n$ of $((E_\tau^\#)^n\times\mathrm{exp}(\mathfrak G),
d+\omega)$ given by Lemma \ref{lemma:diafcottpebov}. 
\end{itemize}
\end{thm}

\end{document}